\documentclass[11pt,a4paper,leqno]{article}

\usepackage{amssymb}
\usepackage{amsmath, amsthm}
\usepackage{graphicx, color}
\usepackage{setspace}

\def\RR{{\bf R}}
\def\ZZ{{\bf Z}}
\def\QQ{{\bf Q}}
\def\mGamma{{\mit\Gamma}}
\def\mLambda{{\mit\Lambda}}

\numberwithin{equation}{section}

\newcommand{\pr}{\mathop{\rm Pr} }

\newcommand{\diam}{\mathop{\rm diam} }

\newcommand{\Conv}{\mathop{\rm Conv} }

\newcommand{\dom}{\mathop{\rm dom} }

\newtheorem{Thm}{Theorem}[section]
\newtheorem{Prop}[Thm]{Proposition}
\newtheorem{Lem}[Thm]{Lemma}
\newtheorem{Cor}[Thm]{Corollary}
\theoremstyle{definition}

\newtheorem{Rem}[Thm]{Remark}

\newenvironment{myitem}{
\refstepcounter{equation}
\\[1.0em] 
$(\thesection.\arabic{equation})$ \quad \quad
\begin{minipage}[t]{12.5cm}}
{\end{minipage}\\[1.0em]\noindent}

\newenvironment{myitem1}{
\refstepcounter{equation}
\\[1.0em]
$(\thesection.\arabic{equation})$  
\begin{minipage}[t]{13.5cm} \begin{itemize}}
{\end{itemize}\end{minipage}\\[1.0em]}

\title{Discrete Convexity and Polynomial Solvability\\ 
in Minimum 0-Extension Problems}
\author{Hiroshi HIRAI \\
Department of Mathematical Informatics, \\
Graduate School of Information Science and Technology,   \\
The University of Tokyo, Tokyo, 113-8656, Japan.\\
\texttt{\normalsize hirai@mist.i.u-tokyo.ac.jp}}
\date{October, 2012\\ May, 2014 (revised) \\ September, 2014 (final)}

\begin{document}

\maketitle

\begin{abstract}
A $0$-extension of graph $\mGamma$ 
is a metric $d$ on a set $V$ containing the vertex set $V_{\mGamma}$ of $\mGamma$
such that $d$ extends the shortest path metric of $\mGamma$ and for all $x \in V$ 
there exists a vertex $s$ in $\mGamma$ with $d(x,s) = 0$.
The minimum $0$-extension problem {\bf 0-Ext}$[\mGamma]$ on $\mGamma$
is: given  a set $V \supseteq V_{\mGamma}$ and a nonnegative cost function $c$ 
defined on the set of all pairs of $V$,
find a $0$-extension $d$ of $\mGamma$ with $\sum_{xy}c(xy) d(x,y)$ minimum.
The $0$-extension problem generalizes 
a number of basic combinatorial optimization problems, 
such as minimum $(s,t)$-cut problem and 
multiway cut problem.

Karzanov proved the polynomial solvability of {\bf 0-Ext}$[\mGamma]$
for a certain large class of modular graphs $\mGamma$, 
and raised the question: What are the graphs $\mGamma$ 
for which {\bf 0-Ext}$[\mGamma]$ 
can be solved in polynomial time?
He also proved that {\bf 0-Ext}$[\mGamma]$ is NP-hard
if $\mGamma$ is not modular or not orientable 
(in a certain sense).

In this paper, we prove the converse:
if $\mGamma$ is orientable and modular, 
then {\bf 0-Ext}$[\mGamma]$ can be solved in polynomial time.
This completes the classification of graphs $\mGamma$ for which {\bf 0-Ext}$[\mGamma]$ is tractable.
To prove our main result, 
we develop a theory of discrete convex functions 
on orientable modular graphs, 
analogous to discrete convex analysis by Murota, 
and utilize a recent result of 
Thapper and \v{Z}ivn\'y on valued CSP.
\end{abstract}

\section{Introduction}
%

By a {\em (semi)metric} $d$ on a finite set $V$ 
we mean a nonnegative symmetric function on $V \times V$ 
satisfying $d(x,x) = 0$ for all $x \in V$ and 
the triangle inequalities 
$d(x,y) + d(y,z) \geq d(x,z)$ 
for all $x,y,z \in V$.
An {\em extension} of a metric space $(S, \mu)$ 
is a metric space $(V,d)$ with $V \supseteq S$ 
and $d(s,t) = \mu(s,t)$ for $s,t \in S$.
An extension $(V,d)$ of $(S,\mu)$ is 
called a {\em $0$-extension} if for all $x \in V$ 
there exists $s \in S$ with $d(s,x) = 0$.

Let $\mGamma$ be a simple connected undirected graph with 
vertex set $V_{\mGamma}$. 
Let $d_\mGamma$ denote the shortest path metric 
on $V_\mGamma$ 
with respect to the uniform unit edge-length of $\mGamma$.
The {\em minimum $0$-extension problem} 
{\bf 0-Ext}$[\mGamma]$
on $\mGamma$ is formulated as:
\begin{description}
\item[{\bf 0-Ext}${[\mGamma]}$:\quad]   
Given $V \supseteq V_{\mGamma}$ and $\displaystyle c: {V \choose 2} \to \QQ_+$,
\item[\hspace{+2cm}]  
minimize 
$\displaystyle \sum_{xy \in {V \choose 2}} c(xy) d(x,y)$ over all 
$0$-extensions $(V,d)$ of $(V_{\mGamma},d_{\mGamma})$.
\end{description}
Here ${V \choose 2}$ denotes the set of all pairs of $V$.
The minimum $0$-extension problem is formulated by Karzanov~\cite{Kar98a}, and 
is equivalent to
the following classical facility location problem, known as
{\em multifacility location problem}~\cite{TFL83}, 
where we let $V \setminus V_{\mGamma} := \{1,2,\ldots, n\}$:
\begin{eqnarray}\label{eqn:multifac}
\mbox{Min.}  && \sum_{s \in V_{\mGamma}} \sum_{1 \leq j \leq n} c (sj) d_{\mGamma}(s, \rho_j) + \sum_{1 \leq i < j \leq n} c(ij) d_{\mGamma}(\rho_i, \rho_j)  \\[0.2em]
\mbox{s.t.} && \rho = (\rho_1,\rho_2,\ldots, \rho_n) \in V_{\mGamma} \times V_{\mGamma} \times \cdots \times V_{\mGamma}. \nonumber
\end{eqnarray}
This problem can be interpreted as follows:
We are going to locate $n$ new facilities $1,2,\ldots,n$ on 
graph $\mGamma$,
where the facilities communicate each other and communicate existing facilities on $\mGamma$.
The cost of the communication is propositional to the distance.
Our goal is to find a location of minimum total communication cost.
This classic facility location problem arises in many practical situations
such as the image segmentation in computer vision, 
and related clustering problems in machine learning; 
see \cite{KleinbergTardos02}.
Also {\bf 0-Ext}$[\mGamma]$ includes
a number of basic combinatorial optimization problems.
For example, take as $\mGamma$ the graph $K_{2}$ 
consisting of a single edge $st$.
Then {\bf 0-Ext}$[K_2]$ is 
the minimum $(s,t)$-cut problem.
More generally, 
{\bf 0-Ext}$[K_{m}]$ is the multiway cut problem 
on $m$ terminals.
Therefore {\bf 0-Ext}$[K_m]$ 
is solvable in polynomial time if $m=2$ and
is NP-hard if $m>2$~\cite{DJPSY94}.

This paper addresses the following 
problem considered by Karzanov~\cite{Kar98a, Kar04a, Kar04b}.
\begin{quote}
{\em What are the graphs $\mGamma$ for which {\bf 0-Ext}$[\mGamma]$
is solvable in polynomial time?}
\end{quote}
Here such a graph is simply called {\em tractable}.

A classical result in location theory in the 1970's is:
\begin{Thm}[{\cite{PicardRatliff}; also see \cite{Kolen}}]
\label{thm:tree}
If $\mGamma$ is a tree,
then {\bf 0-Ext}$[\mGamma]$ is solvable in polynomial time.
\end{Thm}
The tractability of graphs $\mGamma$ is preserved under taking 
Cartesian products.
Therefore, cubes, grid graphs, 
and the Cartesian product of trees
are tractable.
Chepoi~\cite{Chepoi96} extended
this classical result to median graphs as follows.
\begin{figure} 
\begin{center} 
\includegraphics[scale=0.8]{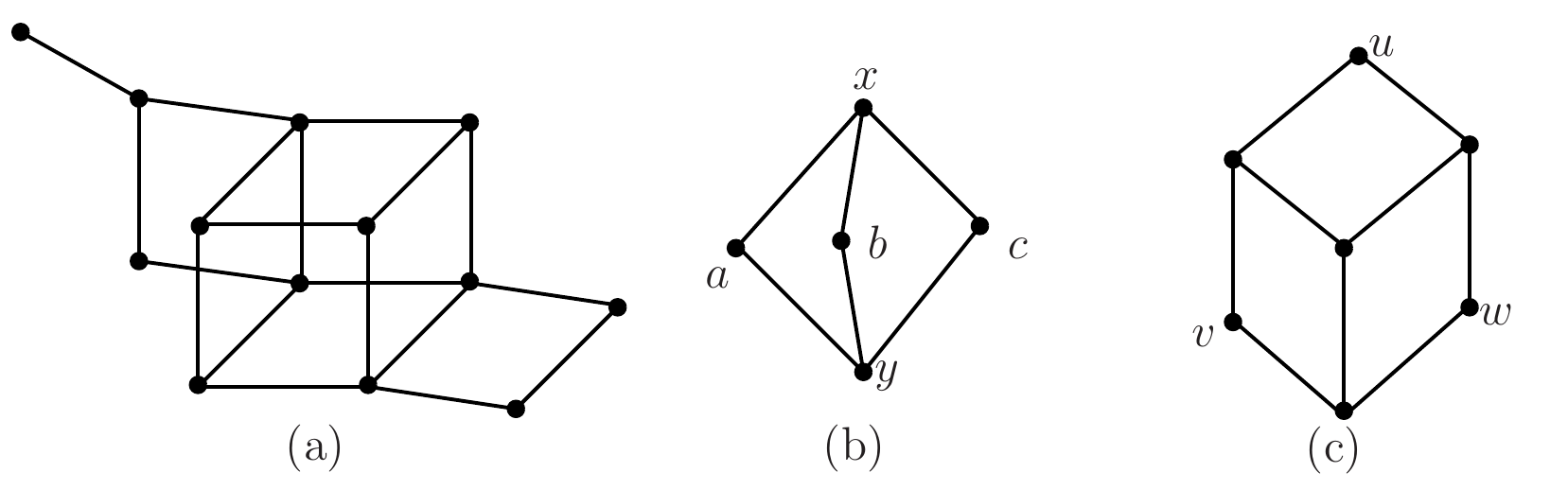}
\caption{(a) a median graph, (b) $a,b,c$ have two medians $x,y$, and 
(c) $u,v,w$ have no median}  
\label{fig:median_graph}         
\end{center}
\end{figure}
A {\em median} of a triple $p_1,p_2,p_3$ of vertices
is a vertex $m$ satisfying 
$d_{\mGamma}(p_i,p_j) = d_{\mGamma}(p_i,m) + d_{\mGamma}(m, p_j)$
for $1 \leq i < j \leq 3$.
A {\em median graph} is a graph
in which every triple of vertices 
has a {\em unique} median.
Trees and their products are median graphs. 
See Figure~\ref{fig:median_graph} for 
illustration of the median concept.
\begin{Thm}[\cite{Chepoi96}]\label{thm:median}
If $\mGamma$ is a median graph,
then {\bf 0-Ext}$[\mGamma]$ 
is solvable in polynomial time.
\end{Thm}
Karzanov~\cite{Kar98a} introduced the following LP-relaxation 
of {\bf 0-Ext}$[\mGamma]$.
\begin{description}
\item[{\bf Ext}${[\mGamma]}$:\quad] 
Given $V \supseteq V_{\mGamma}$ 
and $c: \displaystyle{V \choose 2} \to \QQ_+$,\\
\item[\hspace{+2cm}]  minimize 
$\displaystyle \sum_{xy \in {V \choose 2}} c(xy) d(x,y)$ over all 
extensions $(V,d)$ of $(V_{\mGamma},d_{\mGamma})$.
\end{description}
This relaxation {\bf Ext}$[\mGamma]$ is a linear program
with size polynomial in the input size.
Therefore, if for every input $(V,c)$, 
{\bf Ext}$[\mGamma]$ has an optimal solution that is a $0$-extension,
then {\bf 0-Ext}$[\mGamma]$ is solvable in polynomial time.
In this case 
we say that {\em {\bf Ext}$[\mGamma]$ is exact}. 
In the same paper,
Karzanov gave a combinatorial 
characterization of graphs $\mGamma$ 
for which {\bf Ext}$[\mGamma]$ is exact.
A graph $\mGamma$ is called a {\em frame} if
\begin{myitem1}
\item[(1)] $\mGamma$ is bipartite,
\item[(2)] $\mGamma$ has no isometric cycle of length greater than $4$, and
\item[(3)] $\mGamma$ has an orientation $o$ with the property that
for every 4-cycle $uv,vv',v'u',u'u$, one has 
$u \swarrow_o v$ if and only if $u' \swarrow_o v'$. \label{eqn:frame}
\end{myitem1}\noindent
Here an {\em isometric cycle} in $\mGamma$ means a cycle $C$
such that every pair of vertices in $C$
has a shortest path for $\mGamma$ in this cycle $C$, 
and $p \swarrow_o q$ means that
edge $pq$ is oriented from $q$ to $p$ by $o$.
\begin{Thm}[{\cite{Kar98a}}]\label{thm:minimizable}
{\bf Ext}$[\mGamma]$ is exact 
if and only if $\mGamma$ is a frame. 
\end{Thm}
\begin{Thm}[{\cite{Kar98a}}]\label{thm:frame}
If $\mGamma$ is a frame,
then {\bf 0-Ext}$[\mGamma]$ is solvable in polynomial time.
\end{Thm}
It is noted that the class of frames is {\em not} closed 
under taking Cartesian products, 
whereas the tractability of graphs
is preserved under taking Cartesian products.
Also it should be noted that {\bf Ext}$[\mGamma]$ 
is the LP-dual 
to the $d_{\mGamma}$-weighted maximum {\em multiflow} problem, and
{\bf 0-Ext}$[\mGamma]$ describes a combinatorial 
dual problem~\cite{Kar98a, Kar98b}; 
see also~\cite{HH09JCTB, HH11folder, HHbounded, HHMPA} 
for further elaboration of this duality.

Karzanov~\cite{Kar98a} also proved the following hardness result.
For an undirected graph $\mGamma$,
an orientation  with the property (\ref{eqn:frame})~(3) 
is said to be {\em admissible}.
$\mGamma$ is said to be {\em orientable} 
if it has an admissible orientation.
$\mGamma$ is said to be {\em modular} 
if every triple of vertices 
has a (not necessarily unique) median.
\begin{Thm}[{\cite{Kar98a}}]\label{thm:hard}
If $\mGamma$ is not orientable or not modular, 
then {\bf 0-Ext}$[\mGamma]$ is NP-hard.
\end{Thm}
In fact, a frame is precisely
an orientable modular graph 
with the hereditary property that 
every isometric subgraph 
is modular; see \cite{Bandelt88}.
A median graph is an orientable modular graph but 
the converse is not true.
Moreover,
a median graph is not necessarily a frame, 
and a frame is not necessarily a median graph.
In \cite{Kar04a}, 
Karzanov proved a tractability theorem extending 
Theorem~\ref{thm:median}.
He conjectured that {\bf 0-Ext}$[\mGamma]$ is tractable 
for a certain {\em proper} subclass 
of orientable modular graphs 
including frames and median graphs.
He also conjectured that {\bf 0-Ext}$[\mGamma]$ is 
NP-hard for any graph $\mGamma$ not in this class.

The main result of this paper is 
the tractability theorem for 
{\em all} orientable modular graphs.
Thus the class of tractable graphs is larger than his expectation.
\begin{Thm}\label{thm:main}
If $\mGamma$ is orientable modular, 
then {\bf 0-Ext}$[\mGamma]$ is solvable in polynomial time.
\end{Thm}
Combining this result with Theorem~\ref{thm:hard},
we obtain a complete classification of the graphs $\mGamma$
for which {\bf 0-Ext}$[\mGamma]$ is 
solvable in polynomial time.

\paragraph{Overview.}
In proving Theorem~\ref{thm:main},
we employ an axiomatic approach 
to optimization in orientable modular graphs.
This approach is inspired by the theory 
of discrete convex analysis developed by Murota 
and his collaborators (including Fujishige, Shioura, and Tamura);
see \cite{FujishigeMurota, Murota98, MurotaShioura, MurotaTamura, MurotaBook} 
and also \cite[Chapter VII]{FujiBook}.
Discrete convex analysis
is a theory of convex functions 
on integer lattice $\ZZ^n$, 
with the goal of providing a unified framework
for polynomially solvable 
combinatorial optimization problems 
including network flows, matroids, and 
submodular functions.
The theory that we are going to develop here 
is, in a sense, {\em a theory of
discrete convex functions on orientable 
modular graphs}, 
with the goal of providing 
a unified framework for polynomially 
solvable 0-extension problems
and related multiflow problems.
We believe that our theory 
establishes a new link between previously unrelated fields,
broadens the scope of discrete convex analysis, 
and opens a new perspective and new research directions.

Let us start with a simple observation 
to illustrate our basic idea.
Consider a path $P_m$ 
of length $m$, 
and consider {\bf 0-Ext}$[P_m]$, where $P_m$ is 
trivially an orientable modular graph.
Then {\bf 0-Ext}$[P_m]$ for input $V,c$
can be regarded as an optimization problem
on the integer lattice $\ZZ^n$ as follows.
Suppose that  $V_{P_m} = \{1,2,3,\ldots, m\}$, and 
$s$ and $s+1$ are adjacent for $s = 1,2,\ldots, m-1$.
Then $d_{P_m}(s,t) = |s - t|$, and {\bf 0-Ext}$[P_m]$ 
is equivalent to the minimization of the function
\begin{equation}\label{eqn:typical}
\sum_{1\leq s \leq m} \sum_{1 < j < n} c(s j) | s - \rho_j | 
+  \sum_{1 \leq i < j \leq n} c(ij) | \rho_i - \rho_j |
\end{equation}
over all $(\rho_1, \rho_2,\ldots, \rho_n) \in [0,m]^n \cap \ZZ^n$.
This function 
is a simple instance of 
{\em L$^\natural$-convex functions},  
one of the fundamental classes of discrete convex functions.
We do not give a formal 
definition of L$^\natural$-convex functions here.
The only important facts for us are 
the following properties of L$^\natural$-convex functions in optimization:
\begin{itemize}
\item[(a)] Local optimality implies global optimality.  
\item[(b)] The local optimality can be checked by 
{\em submodular function minimization}.
\item[(c)] An efficient descent algorithm can be designed 
based on successive application of submodular function minimization.
\end{itemize}
As is well-known, submodular functions can be minimized 
in polynomial time~\cite{GLS, IFF, Schrijver}.
Actually
the function (\ref{eqn:typical}) 
can be minimized by successive application of
minimum-cut computation~\cite{Kolen, PicardRatliff}, 
a special case of submodular function minimization.

Motivated by this observation, 
we regard {\bf 0-Ext}$[\mGamma]$ 
as a minimization of a function 
defined on the vertex set 
of a product of $\mGamma$, 
which is also orientable modular. 
We will introduce a class of functions, 
called {\em L-convex functions},
on an orientable modular graph.
We show that our L-convex function satisfies
analogues of (a), (b) and (c) above, 
and also that a multifacility location function, 
the objective function of {\bf 0-Ext}$[\mGamma]$, 
is an L-convex function, in our sense, 
on the product of $\mGamma$. 
Theorem~\ref{thm:main} 
is a consequence of these properties.

Let us briefly mention how to 
define L-convex functions, 
which constitutes the main body of this paper.
Our definition is based 
on the {\em Lov\'asz extension}~\cite{Lovasz83}, 
a well-known concept in 
submodular function theory~\cite{FujiBook}, 
and a kind of construction of polyhedral complexes, 
due to Karzanov~\cite{Kar98a} and Chepoi~\cite{Chepoi00},
from a class of modular graphs. 
Let $\mGamma$ be an orientable modular graph
with admissible orientation $o$.
We call a pair $(\mGamma, o)$ a {\em modular complex}. 
It turns out that $(\mGamma, o)$ can be 
viewed as a structure glued together from modular lattices, and
gives rise to a simplicial complex as follows.
Consider a cube subgraph $B$ of $\mGamma$.
The digraph $\vec B$ oriented by $o$ 
coincides with the Hasse diagram of 
a Boolean lattice.
Consider the simplicial 
complex ${\mit\Delta}(\mGamma,o)$ 
whose simplices are sets of vertices 
forming a chain of the Boolean 
lattice corresponding to 
some cube subgraph of $\mGamma$; 
see Figure~\ref{fig:modular_complex}.
\begin{figure} 
\begin{center}
\includegraphics[scale=0.65]{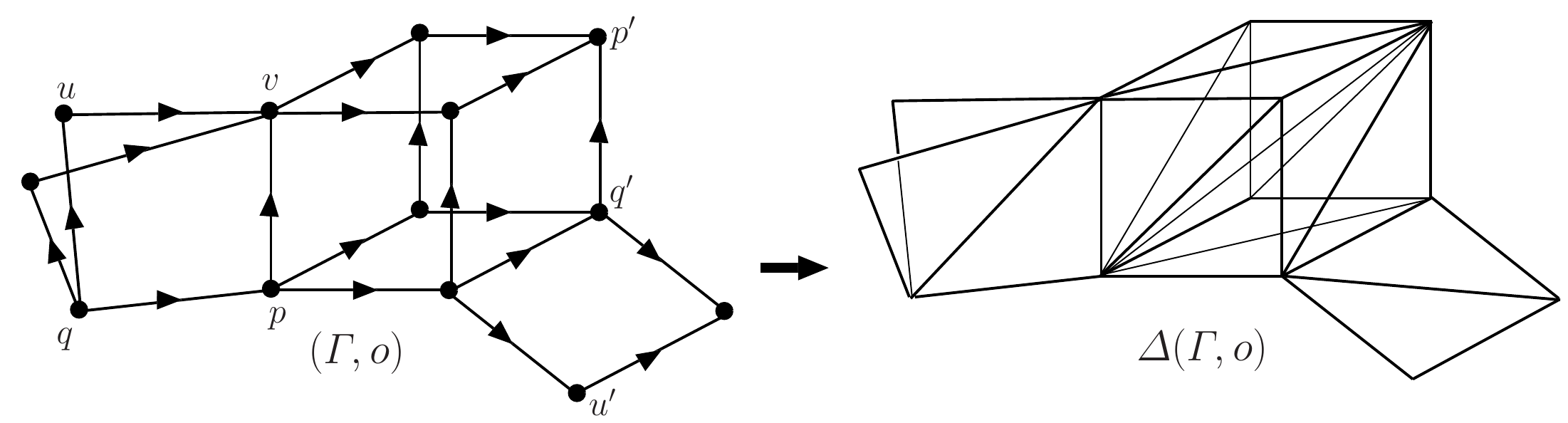} 
\caption{A construction of ${\mit \Delta}(\mGamma, o)$}  
\label{fig:modular_complex}         
\end{center}
\end{figure}
Each (abstract) simplex is naturally regarded 
as a simplex in the Euclidean space.
${\mit\Delta}(\mGamma,o)$ is naturally regarded as 
a metrized simplicial complex.
Then any function $g:V_{\mGamma} \to \RR$ 
is extended to $\overline{g}: {\mit\Delta}(\mGamma,o) \to \RR$
by interpolating $g$ on each simplex linearly;
this is an analogue of the Lov\'asz extension. 
The simplicial 
complex ${\mit\Delta}(\mGamma,o)$ 
enables us to consider 
the {\em neighborhood} ${\cal L}_p^*$
around each vertex $p \in V_{\mGamma}$, 
as well as the local 
behavior of $\overline g$ in ${\cal L}_{p}^*$.
\begin{figure} 
\begin{center} 
\includegraphics[scale=0.7]{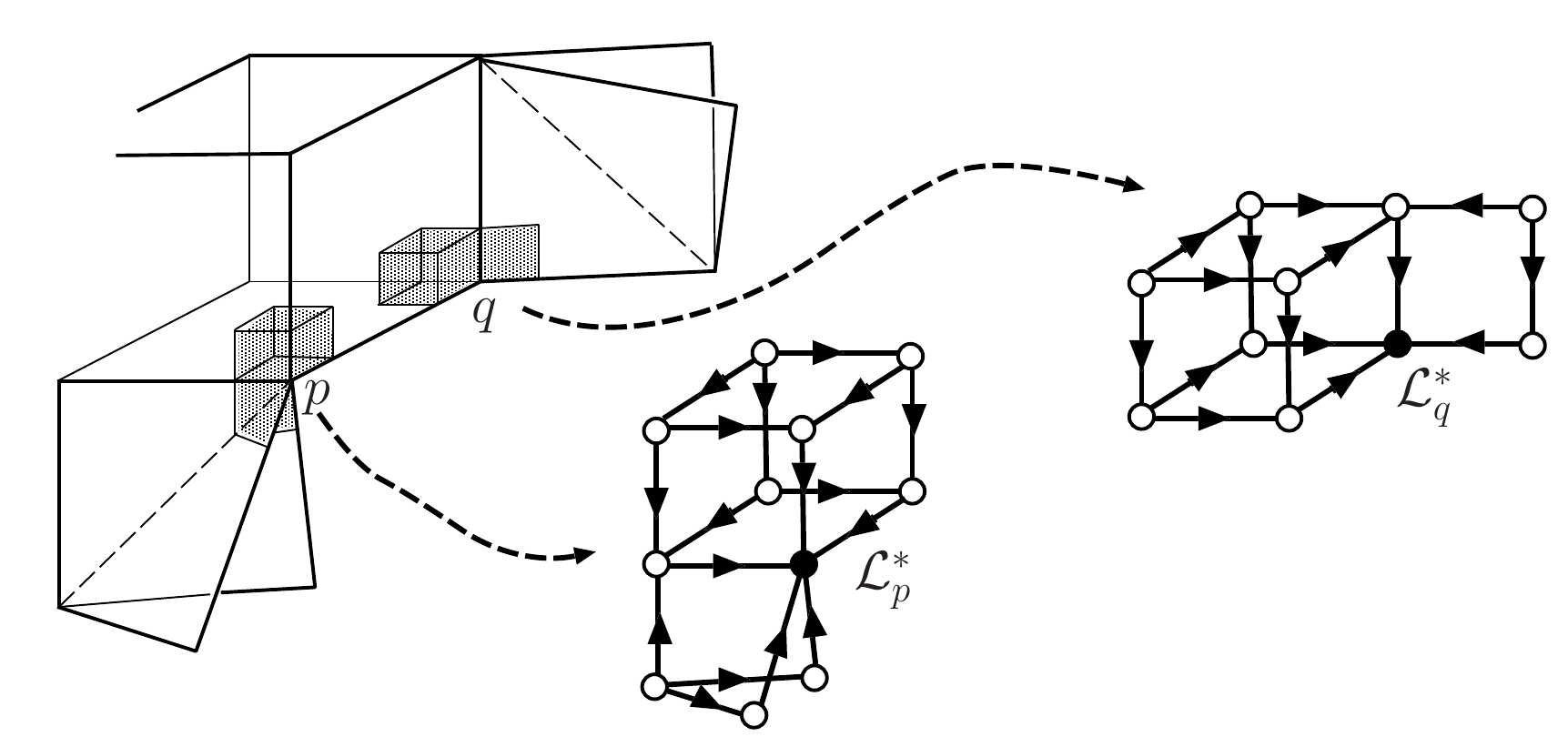}              
\caption{Neighborhood semilattices}  
\label{fig:neighborhood}         
\end{center}
\end{figure}
As in Figure~\ref{fig:neighborhood}, 
neighborhood ${\cal L}^*_p$ 
can be described as a partially ordered set 
with  the unique minimal element $p$.
Then, by restricting $\overline g$ to ${\cal L}^*_p$,
we obtain a function on ${\cal L}^*_p$ 
associated with each vertex $p$.
In fact, the poset ${\cal L}^*_p$ 
is a {\em modular semilattice}, 
a semilattice analogue of a modular lattice introduced by 
Bandelt, van de Vel, and Verheul~\cite{BVV}.  
We first define {\em submodular functions on modular semilattices}, 
and next define {\em L-convex functions on modular complex $(\mGamma,o)$} 
as functions $g$ on $V_{\mGamma}$ such that $\overline g$ is submodular 
on neighborhood semilattice ${\cal L}^*_p$ for each vertex $p$.

Then the multifacility location function, the objective of {\bf 0-Ext}$[\mGamma]$ (see (\ref{eqn:multifac})), 
is indeed an L-convex function on the $n$-fold product of $\mGamma$, and
the optimal solution of {\bf 0-Ext}$[\mGamma]$ can be obtained by 
successive application of submodular function
minimization on the product of $n$ modular semilattices.
Thus our problem reduces to 
the problem of minimizing submodular function $f$ on 
the product of modular 
semilattices ${\cal L}_1, {\cal L}_2,\ldots,{\cal L}_n$,
where the input of the problem is
${\cal L}_1,{\cal L}_2,\ldots,{\cal L}_n$, 
and an evaluating oracle of $f$.
We do not know 
whether this problem in general is tractable in the oracle model,  
but the submodular functions arising 
from {\bf 0-Ext}$[\mGamma]$ take a special form;
they are {\em the sum of submodular functions with 
arity $2$}. 
Here the {\em arity} of a function $f$ is the number of variables of $f$.
Namely, if a function $f$ on 
${\cal L} = {\cal L}_1 \times {\cal L}_2 \times \cdots \times {\cal L}_n$ 
is represented as 
\[
f(x) = h(x_{i_1},x_{i_2},\ldots, x_{i_k}) \quad (x = (x_1,x_2,\ldots,x_n) \in {\cal L})
\]
for some function $h$ on ${\cal L}_{i_1} \times {\cal L}_{i_2} \times \cdots \times {\cal L}_{i_k}$
with $i_1 < i_2 < \cdots < i_k$, 
then the arity of $f$ is (at most) $k$. 
See (\ref{eqn:multifac});
our objective function is a weighted sum of distance functions, which have arity $2$.
This type of optimization problem with bounded arity 
is well-studied in the literature 
of {\em valued CSP} ({\em valued constraint satisfaction problem})~\cite{VCSP1, KolmogorovZivny12, VCSP2, ZivnyBook}.
Valued CSP deals with minimization of a sum of functions $f_i$ $(i=1,2,\ldots,m)$, 
where the arity $k_i$ of each $f_i$ is a part of the input; 
namely the input consists of all values of all functions $f_i$.
Valued CSP admits an integer programming formulation, and 
its natural LP relaxation is called the {\em basic LP-relaxation}.
Recently, Thapper and \v{Z}ivn\'y~\cite{ThapperZivny12FOCS} discovered
a surprising criterion for 
the basic LP-relaxation of valued CSP to exactly solve 
the original valued CSP instance.
They proved that if the class of valued CSP 
(the class of input objective functions)
has a certain nice {\em fractional polymorphism} 
(a certain set of linear inequalities which any input function satisfies), 
then the basic LP-relaxation is exact.
We prove that the class of submodular functions 
on modular semilattice 
admits such a fractional polymorphism. 
Then the sum of submodular functions 
with bounded arity 
can be minimized in polynomial time.
Consequently
we can solve {\bf 0-Ext}$[\mGamma]$ in polynomial time.

We believe that our classes of functions 
deserve to be called submodular and L-convex.
Indeed, they include not only 
(ordinary) submodular/L-convex functions 
but also other submodular/L-convex-type functions.
Examples are
{\em bisubmodular functions}~\cite{CK88, Nakamura88, Qi88}
(see \cite[Section 3.5]{FujiBook}), 
{\em multimatroid rank functions} by Bouchet~\cite{Bouchet97}, 
{\em submodular functions on trees} by Kolmogorov~\cite{Kolmogorov},  
{\em $k$-submodular functions} by 
Huber and Kolmogorov~\cite{HuberKolmogorov12} (also see~\cite{FT13}), 
and {\em skew-bisubmodular functions} by 
Huber, Krokhin and Powell~\cite{HuberKrokhinPowell13} 
(also see~\cite{FTY13, HuberKrokhinPowell13, HuberKrokhin13}).
Moreover, combinatorial dual problems arising from
a large class of (well-behaved) 
multicommodity flow problems, discussed 
in~\cite{HH09JCTB, HH11folder, HHbounded, HHMPA, Kar89, Kar98a, Kar98b}, 
fall into submodular/L-convex function 
minimization in our sense.
This can be understood as a multiflow analogue 
of a fundamental fact in network flow theory: 
the minimum cut problem, the dual of maxflow problem, 
is a submodular function minimization.
The detailed discussion on these topics
will be given in a separate paper~\cite{HHprepar}; 
some of the results were announced by~\cite{HH_HJ}.

\paragraph{Organization.}
In Section~\ref{sec:preliminaries}, 
we first explain basic notions of valued CSP and 
the Thapper-\v{Z}ivn\'{y} criterion (Theorem~\ref{thm:TZ}) 
on the exactness of the basic LP relaxation.
We then describe basic facts on modular graphs 
and modular lattices. 
In Section~\ref{sec:submo}, 
we develop a theory of submodular functions on modular semilattices.
We show that our submodular function satisfies the Thapper-\v{Z}ivn\'{y} criterion, 
and that a sum of submodular functions with bounded arity
can be minimized in polynomial time.
In Section~\ref{sec:L-convex},
we first explore several structural properties of orientable modular graphs.
Based on the above mentioned idea,
we define L-convex functions, 
and prove that our L-convex functions indeed have properties 
analogous to (a), (b) and (c) above.
In Section~\ref{sec:0extension}, 
we formulate {\bf 0-Ext}$[\mGamma]$ 
as an optimization problem on a modular complex.
We show that a multifacility location function, 
the objective function of {\bf 0-Ext}$[\mGamma]$,
is indeed an L-convex function, and 
we prove Theorem~\ref{thm:main}.
Our framework is applicable to 
a certain weighted version of {\bf 0-Ext}$[\mGamma]$.
As a corollary,  
we give a generalization of 
Theorem~\ref{thm:main} to general metrics, 
which completes classification of metrics $\mu$ 
for which the 0-extension problem 
on $\mu$ is polynomial time solvable (Theorem~\ref{thm:main_mu}).
In the last section (Section~\ref{sec:concluding}), 
we discuss a connection to a dichotomy theorem of finite-valued CSP 
obtained by Thapper and \v{Z}ivn\'y~\cite{ThapperZivny13STOC} 
after the first submission of this paper.
In fact, the complexity dichotomy (of form ``either P or NP-hard") of {\bf 0-Ext}, 
established in this paper,  can be viewed as a special case of 
their dichotomy theorem of finite-valued CSP.

\paragraph{Notation.}
Let $\ZZ,\QQ$, and $\RR$ denote 
the sets of integers, rationals, and reals, respectively.
Let $\overline{\RR} := \RR \cup \{\infty\}$ and $\overline{\QQ} := \QQ \cup \{\infty\}$, 
where $\infty$ is an infinity element and is treated as: $\infty \cdot 0 = 0$, 
$x < \infty$ $(x \in \RR)$, $\infty + x = \infty$ $(x \in \overline{\RR})$,  
$x \cdot \infty = \infty$ $(a \in \RR: a > 0)$.
Let $\ZZ_{+},\QQ_{+}$, and $\RR_{+}$ denote 
the sets of nonnegative integers, nonnegative rationals, 
and nonnegative reals, respectively.
For a function $f: X \to \overline{\RR}$ on a set $X$, 
let $\dom f$ denote the set of elements $x \in X$ with $f(x) \neq \infty$.

For a graph $\mGamma$, 
the vertex set and the edge set 
are denoted by $V_{\mGamma}$ and $E_{\mGamma}$, 
respectively.
For a vertex subset $X$, 
$\mGamma[X]$ denotes the subgraph of $\mGamma$ induced by $X$.
For a nonnegative edge-length $h: E_{\mGamma} \to \RR_{+}$, 
$d_{\mGamma,h}$ denotes the shortest path metric on $V_{\mGamma}$ 
with respect to the edge-length $h$.
When $h(e) = 1$ for every edge $e$,
$d_{\mGamma, h}$ is denoted by $d_{\mGamma}$.
A path is represented by
a chain $(p_1,p_2,\ldots,p_n)$ 
of vertices with $p_ip_{i+1} \in E_{\mGamma}$.
The {\em Cartesian product} 
$\mGamma \times \mGamma'$ of graphs $\mGamma$ and $\mGamma'$
is the graph with vertex set $V_\mGamma \times V_{\mGamma'}$ 
and edge set given as:
$(p,p')$ and $(q,q')$ are connected by an edge 
if and only if 
$p = q$ and $p'q' \in E_{\mGamma'}$ or $p' = q'$ and $pq \in E_\mGamma$.
The $n$-fold Cartesian product $\mGamma \times \mGamma \times \cdots \times \mGamma$ 
of $\mGamma$ is denoted by $\mGamma^n$.
In this paper, 
graphs and posets (partially ordered sets)  are supposed to be finite.

\section{Preliminaries}\label{sec:preliminaries}
In this section, we give preliminary
arguments for valued CSP, and modular graphs and modular (semi)lattices. 
Our references are \cite{KTZ13, ZivnyBook} for valued CSP and 
\cite{BandeltChepoi, BVV, Birkhoff, Chepoi00, VanDeVel} 
for modular graphs and lattices. 
A further discussion on valued CSP is given in Section~\ref{sec:concluding}.

\subsection{Valued CSP and fractional polymorphism}\label{subsec:VCSP}
Let $D_1,D_2,\ldots,D_n$ be finite sets, 
and let $D :=D_1 \times D_2 \times \cdots \times D_n$.
A {\em constraint} on $D$ is a function 
$f:D_{i_1} \times D_{i_2} \times \cdots \times D_{i_k} \to \overline \RR$
for some $i_1 < i_2 < \cdots < i_k$, 
where $I_f := \{i_1,i_2,\ldots,i_k\}$ is called the {\em scope} of $f$, 
and $k_f := k$ is called the {\em arity} of $f$.
Let $D_{I_f} := D_{i_1} \times D_{i_2} \times \cdots \times D_{i_f}$, and
for $x = (x_1,x_2,\ldots,x_n) \in D$, 
let $x_{I_f} := (x_{i_1},x_{i_2},\ldots, x_{i_k}) \in D_{I_f}$.
The {\em valued CSP} ({\em valued constraint satisfaction problem}) is:
\begin{description}
\item[VCSP:] Given a set ${\cal F}$ of constraints on $D$,
\item[] \hspace{1cm} minimize $\displaystyle \sum_{f \in {\cal F}} f(x_{I_f})$ over all $x = (x_1,x_2,\ldots,x_n) \in D$.
\end{description}
The input of VCSP is the set of all values of all constraints in ${\cal F}$, 
and hence its size is estimated by $O(|{\cal F}| N^K B)$, 
where $N := \max_{1 \leq i \leq n} |D_i|$, $K := \max_{f \in {\cal F}} k_f$, 
and $B$ is the bit size to represent constraints in ${\cal F}$.
By a {\em constraint language} 
we mean a (possibly infinite) set  $\mLambda$ of constraints.
A constraint in  $\mLambda$ is called a {\em $\mLambda$-constraint}.  
Let {\bf VCSP}$[\mLambda]$ denote the subclass of {\bf VCSP}
such that the input is restricted to a set of $\mLambda$-constraints.
%

The minimum 0-extension problem {\bf 0-Ext}$[\mGamma]$ is formulated 
as an instance of {\bf VCSP}.
Let $D_i := V_{\mGamma}$ for $i=1,2,\ldots,n$.
Define constraints $g_i: D_i \to \RR$ and $f_{ij}: D_i \times D_j \to \RR$ by
\begin{eqnarray}\label{eqn:g_jf_ij}
g_i (\rho_i) &:= & \sum_{s \in V_{\mGamma}} c(si) d_{\mGamma}(s, \rho_i) \quad (\rho_i \in D_i), \\
f_{ij}(\rho_i, \rho_j) & := & c(ij) d_{\mGamma}(\rho_i, \rho_j) \quad ((\rho_i,\rho_j) \in D_i \times D_j). \nonumber
\end{eqnarray}
Define the input ${\cal F}$ of {\bf VCSP} by
\[
{\cal F} := \{ g_i \mid 1 \leq i \leq n\} \cup \{ f_{ij} \mid 1 \leq i < j \leq n\}.
\]
Notice that the size of ${\cal F}$ is polynomial in $n$, $\mGamma$, and 
the bit size representing $c$. 
Hence {\bf 0-Ext}$[\mGamma]$ is a particular subclass of {\bf VCSP}.

{\bf VCSP} admits
the following integer programming formulation:
\begin{eqnarray}\label{eqn:IP}
\mbox{Min.} && \sum_{f \in {\cal F}} \sum_{y \in \dom f} f(y) \lambda_{f,y}  \\
\mbox{s.t.} && \sum_{y \in \dom f: y_i = a} \lambda_{f,y} = \mu_{i, a} \quad (f \in {\cal F}, i \in I_f, a \in D_i) \nonumber \\
   && \sum_{a \in D_i} \mu_{i,a} = 1 \quad (1 \leq i \leq n), \nonumber \\
   && \lambda_{f,y} \in \{0,1\} \quad ( f \in {\cal F}, y \in \dom f),\nonumber \\
   && \mu_{i,a} \in \{0,1\} \quad (1 \leq i \leq n, a \in D_i). \nonumber
\end{eqnarray} 
Indeed,  for each $i$ there uniquely exists $a_i \in D_i$ with $\mu_{i,a_i} = 1$.
Also for $f \in {\cal F}$ there uniquely exists $y \in \dom f$ such that 
$\lambda_{f,y} = 1$ and $y_i = a_i$ for $i \in I_f$.
Define $x = (x_1,x_2,\ldots,x_n)$ by $x_i := a_i$. 
Then $\lambda_{f,y} = 1$ if and only if $x_{I_f} = y$.
Therefore we obtain a solution  $x = (x_1,x_2,\ldots,x_n)$ of {\bf VCSP} with the same objective value.
Conversely, for a solution $x = (x_1,x_2,\ldots,x_n)$ of {\bf VCSP}, 
define $\mu_{i,a} := 1$ if $x_i = a$, and $\lambda_{f,y} := 1$ 
if $x_{I_f} = y$. The other variables are defined as zero.
Then we obtain a solution of (\ref{eqn:IP}) with the same objective value.

Observe that there are $O(|{\cal F}| N^K + n N)$ variables 
and $O(|{\cal F}| K N + n)$ constraints.
Therefore the size of this IP is bounded by a polynomial of the input size.
The {\em basic LP relaxation} (BLP) is the linear problem obtained 
by relaxing the 0-1 constraints $\lambda_{f,y} \in \{0,1\}$ and $\mu_{i,a} \in \{0,1\}$
into $\lambda_{f,y} \geq 0$ and  $\mu_{i,a} \geq 0$, respectively.
In particular BLP can be solved in (strongly) polynomial time.

Recently Thapper and \v{Z}ivn\'y~\cite{ThapperZivny12FOCS} 
discovered a surprisingly powerful criterion for which BLP solves {\bf VCSP}.
To describe their result, let us introduce some notions.
For a constraint language $\mLambda$, 
BLP is said be {\em exact} for ${\mLambda}$ if for every input ${\cal F} \subseteq {\mLambda}$,
the optimal value of BLP coincides with the optimal value of {\bf VCSP}$[\mLambda]$.
An {\em operation} on $D_i$ is a function $D_i \times D_i \to D_i$.
A (separable) operation $\vartheta$ on $D$
is a function $D \times D \to D$
such that $\vartheta$ is represented as
\[
\vartheta(x,y) = (\vartheta_1(x_1,y_1), \vartheta_2(x_2,y_2),\ldots, \vartheta_n(x_n,y_n))
\quad (x,y \in D)
\]
for some operations $\vartheta_i$ on $D_i$ for $i=1,2,\ldots,n$.
A {\em fractional operation} $\omega$ is a function 
from the set of all operations to $\RR_+$ 
such that the total sum $\sum \omega(\vartheta)$ over all operations $\vartheta$ is $1$.
We denote a fractional operation $\omega$ 
by the form of a {\em formal convex combination} 
$\sum \omega(\vartheta) \vartheta$ of operations $\vartheta$.
The {\em support} of $\omega$ is the set of operations $\vartheta$ 
with $\omega(\vartheta) > 0$. 
For a constraint language $\mLambda$, 
a {\em fractional polymorphism}
is a fractional operation $\sum_{\vartheta} \omega({\vartheta}) \vartheta$ on $D$ 
such that it satisfies
\begin{equation}
\frac{f(x) + f(y)}{2} \geq  \sum_{\vartheta} \omega(\vartheta) 
f( \vartheta(x,y) ) \quad (f \in \mLambda, x,y \in D_{I_f}), 
\end{equation}
where $\vartheta$ 
is regarded an operation on $D_{I_f}$
by $(\vartheta(x,y))_{i} := \vartheta_{i}(x_{i},y_{i})$ for $i \in I_f$.
For example, if $D_i$ is a lattice for each $i$, 
then $\frac{1}{2}\wedge + \frac{1}{2}\vee$ is nothing but a fractional polymorphism 
for submodular functions, i.e., 
functions $f$ satisfying $f(p) + f(q) \geq f(p \wedge q) + f(p \vee q)$ for $p,q \in D$.

\begin{Thm}[{Special case of \cite[Theorem 5.1 ]{ThapperZivny12FOCS}}]{\label{thm:TZ}}
If a constraint language $\mLambda$ admits
a fractional polymorphism $\omega$
such that the support of $\omega$ 
contains a semilattice operation, 
then BLP is exact for $\mLambda$, and hence {\bf VCSP}$[\mLambda]$ can be solved in polynomial time.
\end{Thm}
Here a {\em semilattice operation} is an operation $\vartheta$ 
satisfying $\vartheta(a,a) = a$, $\vartheta(a,b) = \vartheta(b,a)$, 
and $\vartheta( \vartheta(a,b),c) = \vartheta(a,\vartheta(b,c))$ for $a,b,c \in D$.
Although the feasible region of BLP is not necessarily an integral polytope,
we can check whether there exists an optimal solution $x$ with $x_i = a \in D_i$ 
by comparing the optimal values of BLP for the input ${\cal F}$ and for ${\cal F}_{i,a}$, 
which is the set of cost functions
obtained by fixing variable $x_i$ to $a$ for each cost function on ${\cal F}$.
Necessarily BLP is exact for ${\cal F}_{i,a}$ if there is an optimal solution $x$ with $x_i = a$.
Hence, after $n$ fixing procedures, we obtain an optimal solution $x$.

\begin{Rem}
In the setting in~\cite{ThapperZivny12FOCS}, 
$D_i$ is the same set $\tilde D$ for all $i$.
Our setting reduces to this case by 
taking the disjoint union of $D_i$ as $\tilde D$, and
extending each cost function $f : D_{I_f} \to \overline{\RR}$ 
to $f: {\tilde D}^{k_f} \to \overline{\RR}$ 
by $f(x_{i_1}, x_{i_2},\ldots, x_{i_k}) :=  \infty$ 
for $(x_{i_1},x_{i_2}, \ldots, x_{i_k}) 
\not \in D_{i_1} \times D_{i_2} \times \cdots \times D_{i_K}$.
Without such a reduction, 
their proof also works for our setting 
in a straightforward way.
\end{Rem}

\subsection{Modular metric spaces and modular graphs}
For a metric space $(X,d)$, 
the {\em (metric) interval} $I(x,y)$ 
of $x,y \in X$ is defined as
\[
I(x,y) := \{z \in X \mid d(x,z) + d(z,y) = d(x,y)\}.
\]
For two subsets $A,B$, $d(A,B)$ denotes 
the infimum of distances between $A$ and $B$, i.e.,
\[
d(A,B) = \inf_{x \in A, y \in B} d(x,y).
\]
For $x_1,x_2,x_3 \in X$, 
an element $m$ in 
$I(x_1,x_2) \cap I(x_2,x_3) \cap I(x_3,x_1)$
is called {\em a median} of $x_1$, $x_2$, and $x_3$.
%
%
A metric space $(X,d)$ is said to be {\em modular} 
if every triple of elements in $X$ has a median.
In particular, a graph $\mGamma$ is modular 
if and only if the shortest path metric space 
$(V_{\mGamma}, d_{\mGamma})$ is modular.
We will often use 
the following characterization of modular graphs. 
\begin{Lem}[{\cite[Proposition 1.7]{BVV}; see \cite[Proposition 6.2.6, Chapter I]{VanDeVel}}]\label{lem:quadrangle}
A connected graph $\mit\Gamma$ is modular if and only if
\begin{itemize}
\item[{\rm (1)}] $\mit\Gamma$ is bipartite, and
\item[{\rm (2)}] for vertices $p,q$ and 
neighbors $p_1,p_2$ of $p$ with $d_{\mGamma}(p,q) 
= 1 + d_{\mGamma}(p_1,q) = 1 + d_{\mGamma}(p_2,q)$,
there exists a common neighbor $p^*$ of $p_1,p_2$ with
$d_{\mGamma}(p,q) = 2 + d_{\mGamma}(p^*,q)$.
\end{itemize}
\end{Lem}
The condition (2) is 
called the {\em quadrangle condition}~\cite{BandeltChepoi, Chepoi00} 
(or the {\em semimodularity condition} in \cite{BVV, VanDeVel}).
\begin{Lem}\label{lem:acyclic}
For a modular graph, 
every admissible orientation is acyclic.
\end{Lem}
\begin{proof}
Suppose indirectly that the statement is false.
Take a vertex $p$ belonging to a directed cycle, 
and take a directed cycle $C$ containing $p$ 
with $\sum_{u \in V_C} d_{\mGamma}(p,u)$ minimum.
The length $k$ of $C$ is at least four 
(by simpleness and bipartiteness).
By the definition of admissible orientation,
$k=4$ is impossible.
Hence $k > 4$. Take a vertex $q$ in $C$ 
with $d_{\mGamma}(p,q)$ maximum.
Take two neighbors $q',q''$ of $q$ in $C$.
Then $d_{\mGamma}(q,p) = d_{\mGamma}(q',p) + 1 = d_{\mGamma}(q'',p) + 1$ 
(by the maximality of $q$ and the bipartiteness of $\mGamma$).
By the quadrangle condition, 
there is a common neighbor $q^*$ of $q',q''$ 
with $d_{\mGamma}(p,q^*) = d_{\mGamma}(p,q) - 2$.
Here the cycle $C'$ obtained from $C$
by replacing $q$ by $q^*$ is a directed cycle, 
since the orientation is admissible. 
Then we have 
$\sum_{u' \in V_{C'}}d_{\mGamma}(p,u') < \sum_{u \in V_C} d_{\mGamma}(p,u)$. 
This contradicts the minimality of $C$.
\end{proof}

\subsubsection{Orbits and orbit-invariant functions}~\label{subsec:orbits}
Let $\mGamma$ be a modular graph.
Edges $e$ and $e'$ are said to be 
{\em projective} 
if 
there is a sequence $(e = e_0, e_1, e_2, \ldots, e_m = e')$ of edges
such that $e_i$ and $e_{i+1}$ 
belong to a common 4-cycle
and share no common vertex.
We will use
the following criterion for two edges to belong to a common orbit.
\begin{Lem}\label{lem:commonorbit}
Let $\mGamma$ be a modular graph.
For edges $pq$ and $p'q'$, suppose that
$d_\mGamma(p,p') = d_\mGamma (q,q')$ and 
$d_{\mGamma}(p,q') = d_{\mGamma}(p,p') + 1 = d_{\mGamma}(p',q)$. 
\begin{itemize}
\item[{\rm (1)}] $pq$ and $p'q'$ are projective.
\item[{\rm (2)}] In addition, if
$\mGamma$ has an admissible orientation $o$,
then $p \searrow_o q$ implies $p' \searrow_o q'$.
\end{itemize}
\end{Lem}
\begin{proof}
We use the induction on $k := d_{\mGamma}(p,p') = d_{\mGamma}(q,q')$.
The case of $k=1$ is obvious.
Take a neighbor $p^*$ of $p$ 
with $d_{\mGamma}(p^*,p') = d_{\mGamma}(p,p') - 1 = k-1$.
Then $d_{\mGamma}(p^*,q') = k$.
By the quadrangle condition for $p,q,p^*,q'$, 
there is a common neighbor $q^{*}$ of $q,p^*$ with 
$d_{\mGamma}(q^{*},q') = k-1$. Also $d_{\mGamma}(q^*,p') = k$.
Obviously $pq$ and $p^*q^*$ are projective, 
and $p \searrow_o q$ implies $p^* \searrow_o q^*$.
Apply the induction for $p^*q^*$ and $p'q'$.
\end{proof}
An {\em orbit} is an equivalence class of the projectivity relation.
The (disjoint) union of several orbits is 
called an {\em orbit-union}.
For an orbit-union $U$, $\mGamma/{U}$ 
is the graph obtained 
by contracting all edges {\em not} in $U$
and by identifying multiple edges. 
The vertex in $\mGamma/U$ corresponding to $p \in V_{\mGamma}$ 
is denoted by $p/U$.
The graph $\mGamma/U$ is also modular, and
any shortest path in $\mGamma$ induces a shortest path in $\mGamma/U$ as follows.
\begin{Lem}[\cite{Bandelt85}, also see \cite{Kar04a}]\label{lem:orbits}
Let $\mGamma$ be a modular graph, and
$U$ an orbit-union.
\begin{itemize} 
\item[{\rm (1)}] $\mGamma /U$ is a modular graph.
\item[{\rm (2)}] For every $p,q \in V_{\mGamma}$,
every shortest $(p,q)$-path $P$, 
and every $(p,q)$-path $P'$, 
we have $|P \cap U| \leq |P' \cap U|$. 
\item[{\rm (3)}] 
For every $p,q \in V_{\mGamma}$ and every shortest $(p,q)$-path $P$, the image $P/U$ of $P$ 
is a shortest $(p/U,q/U)$-path in $\mGamma/U$.
\end{itemize}
\end{Lem}\noindent
In particular, for any partition ${\cal U}$ of $E_{\mGamma}$ into orbit-unions, we have
\begin{equation*}
d_{\mGamma}(p,q) = \sum_{U \in {\cal U}} d_{\mGamma /U}(p/U,q/U)  = \sum_{Q:\mbox{\footnotesize{orbit}}} d_{\mGamma /Q}(p/Q,q/Q) \quad (p,q \in V_\mGamma).
\end{equation*}

A function $h$ on edge set $E_{\mGamma}$
is called {\em orbit-invariant}
if $h(e) = h(e')$ provided $e$ and $e'$ belong to 
the same orbit.
For an orbit $Q$, let $h_Q$ denote the value of $h$ on $Q$.
An orbit-invariant function $h$ is said to be {\em nonnegative} 
if $h(e) \geq 0$ for $e \in E_{\mGamma}$, 
and is said to be {\em positive} 
if $h(e) > 0$ for $e \in E_{\mGamma}$.
For a constant $c \geq 0$,
if $h(e) = c$ for all edges $e$, 
then $h$ is simply denoted by $c$; 
in particular $d_{\mGamma} = d_{\mGamma, 1}$.
By taking the value of $h$ of the preimage, 
we can define a function on the edge set of $\mGamma/U$ 
for any orbit-union $U$,
which is also orbit-invariant 
in $\mGamma/U$ and is denoted by $h$.  
By Lemma~\ref{lem:orbits}~(2), 
the shortest path structures 
of $(V_{\mGamma}, d_{\mGamma})$ and $(V_{\mGamma}, d_{\mGamma,h})$
are the same in the following sense:
\begin{Lem}\label{lem:shortest}
If an orbit-invariant function $h$ is nonnegative, then 
$(1)$ implies $(2)$, where
\begin{itemize} 
\item[{\rm (1)}] $P$ is a shortest $(p,q)$-path with respect to $1$,
\item[{\rm (2)}] $P$ is a shortest $(p,q)$-path with respect to $h$.
\end{itemize}
If $h$ is positive, then the converse also holds.
\end{Lem}
As a consequence of Lemmas~\ref{lem:orbits} and \ref{lem:shortest}, 
for any partition ${\cal U}$ of $E_{\mGamma}$ into orbit-unions, we have
\begin{equation}\label{eqn:into}
d_{\mGamma,h}(p,q)  =  \sum_{U \in {\cal U}} d_{\mGamma/U, h} (p/U,q/U) 
=
\sum_{Q:\mbox{\footnotesize{orbit}}} h_Q d_{\mGamma/Q, 1} (p/Q,q/Q). 
\end{equation}

\subsubsection{Convex sets and gated sets}\label{subsec:gated}
Let $(X,d)$ be a metric space.
A subset $Y \subseteq X$ is called {\em convex} 
if $I(p,q) \subseteq Y$ for every $p,q \in Y$.
A subset $Y \subseteq X$ is called {\em gated} 
if for every $p \in X$ there is $p^{*} \in Y$, called a {\em gate} of $p$ at $Y$, 
such that $d(p,q) = d(p,p^{*}) + d(p^{*},q)$ holds for every $q \in Y$.
One can easily see that 
gate $p^{*}$ is uniquely determined for each $p$~\cite[p.~112]{DS87}.
Therefore we obtain a map $\pr_Y: X \to Y$ by 
defining $\pr_Y(p)$ to be the gate of $p$ at $Y$.
\begin{Thm}[{\cite{DS87}}]\label{thm:gatedset}
Let $A$ and $A'$ 
be gated subsets of $(X,d)$ and 
let $B := \pr_{A}(A')$ and $B' := \pr_{A'}(A)$.
\begin{itemize}
\item[{\rm (1)}] $\pr_{A}$ and $\pr_{A'}$ 
induce isometries, inverse to each other, 
between $B'$ and $B$.
\item[{\rm (2)}] For $p \in A$ and $p' \in A'$, 
the following conditions are equivalent:
\begin{itemize}
\item[{\rm (i)}] $d(p,p') = d(A,A')$. 
\item[{\rm (ii)}] $p = \pr_{A}(p')$ and $p' = \pr_{A'}(p)$.
\end{itemize}
\item[{\rm (3)}] $B$ and $B'$ are gated, 
and $\pr_{B} = \pr_{A} \circ \pr_{A'}$ and $\pr_{B'} = \pr_{A'} \circ \pr_{A}$.
\end{itemize}
\end{Thm}
As remarked in \cite{DS87}, 
every gated set is convex 
(see the proof of Lemma~\ref{lem:2convexity} below). 
The converse is not true in general, 
but is true for modular graphs.
The following useful characterization 
of convex (gated) sets in a modular graph
is due to Chepoi~\cite{Chepoi89}.
Here, for a graph $\mGamma$, 
a subset $Y$ of vertices 
is said to be convex (resp. gated)
if $Y$ is convex (resp. gated) in $(V_{\mGamma}, d_{\mGamma})$.  
\begin{Lem}[\cite{Chepoi89}]\label{lem:2convexity}
Let $\mit\Gamma$ be a modular graph. 
For $Y \subseteq V_\mGamma$, 
the following conditions are equivalent:
\begin{itemize} 
\item[{\rm (1)}] $Y$ is convex.
\item[{\rm (2)}] $Y$ is gated.
\item[{\rm (3)}] $\mGamma [Y]$ 
is connected and
$I(p,q) \subseteq Y$ holds
for every $p,q \in Y$ with $d_{\mit\Gamma}(p,q) = 2$.
\end{itemize}
\end{Lem}
We give a proof for the convenience of readers 
as the original paper is in Russian.  
\begin{proof}
$d_{\mGamma}$ is denoted by $d$.
(1) $\Rightarrow$ (3) is obvious.
We show (3) $\Rightarrow$ (1).
Take $p,q \in Y$, and take $a \in I(p,q)$. 
We are going to show $a \in Y$.
Since ${\mit\Gamma}[Y]$ is connected, 
we can take a path $P = (p = p_0,p_1,\ldots, p_k = q)$ 
with $p_i \in Y$. 
Take such a path $P$ with
$\kappa_P := \sum_{i=0}^{k} d(a,p_i)$ minimum.
If $d(a,p_{i-1}) < d(a, p_i) > d(a,p_{i+1})$
for some $i$, then, 
by the quadrangle condition in Lemma~\ref{lem:quadrangle},
there is a common neighbor $p^*$ of $p_{i-1}, p_{i+1}$ with 
$d(a,p^*) = d(a,p_i) - 2$.
Since $I(p_{i-1},p_{i+1}) \subseteq Y$ by (3), 
$p^*$ belongs to $Y$.
Then we can replace $p_i$ by $p^*$ in $P$ 
to obtain another path $P'$ connecting $p,q$ 
with $\kappa_{P'}= \kappa_{P} - 2$;
a contradiction to the minimality.
Therefore there is no index $j$ with 
$d(a,p_{j-1}) < d(a, p_j) > d(a,p_{j+1})$.
Thus there is a unique 
index $i$ with $d(a,p_i)$ minimum.
Then we have $d(p,p_i) + d(p_i,a) = d(p,a)$ and
$d(q,p_i) + d(p_i,a) = d(q,a)$.
By $a \in I(p,q)$, 
we have $d(p,q) = d(p,a) + d(a,q)= d(p,p_i) + d(p_i,q) + 2d(p_i,a) \geq d(p,q) + 2d(p_i,a)$.
Hence we must have $d(p_i,a) = 0$, implying $a = p_i \in Y$.

We show (2) $\Rightarrow$ (1).
As already mentioned, 
any gated set is convex.
Indeed, suppose that
$Y$ is gated.
Take $p,q \in Y$, and take $a \in I(p,q)$.
Consider the gate $a^*$ of $a$ in $Y$.
Then $d(p,a) = d(p,a^*) + d(a^*,a)$ 
and $d(q,a) = d(q,a^*) + d(a^*,a)$.
Since $a \in I(p,q)$, we have
$d(p,q) = d(p,a) + d(a,q) = d(p,a^*) + d(a^*,q) + 2 d(a^*,a) 
\geq d(p,q) + 2d(a^*,a)$,
implying $d(a^*,a) = 0$ and $a = a^* \in Y$.
Thus we get (2) $\Rightarrow$ (1).

Finally we show (1) $\Rightarrow$ (2).
Suppose that $Y$ is convex. 
Let $p$ be an arbitrary vertex.
Let $p^*$ be a point in $Y$ satisfying $d(p,Y) = d(p,p^*)$. 
We show that $p^{*}$ is a gate of $p$ at $Y$.
Take arbitrary $q \in Y$. 
Consider a median $m$ of $p,q,p^{*}$. 
By convexity, $m$ belongs to $Y$, and also $m \in I(p^*,p)$.
By definition of $p^*$,  it must hold $p^* = m$.
Thus $d(p,q) = d(p,p^*) + d(p^*,q)$ holds for every $q \in Y$.
This means that $p^{*}$ is the gate of $p$, 
and therefore $Y$ is gated.
\end{proof}

\subsection{Modular lattices and modular semilattices}\label{subsec:lattice}
Let ${\cal L}$ be a partially ordered set (poset)
with partial order $\preceq$. 
For $a,b \in {\cal L}$, 
the (unique) minimum common upper bound, if it exists, 
is denoted by $a \vee b$, 
and the (unique) maximum common lower bound, if it exists, 
is denoted by $a \wedge b$.
${\cal L}$ is said to be a {\em lattice} if
both $a \vee b$ and $a \wedge b$ exist for every $a,b \in {\cal L}$, 
and said to be a {\em (meet-)semilattice} if $a \wedge b$ exists 
for every $a,b \in {\cal L}$.
In a semilattice, 
if $a$ and $b$ have a common upper bound, 
then $a \vee b$ exists.
Such $(a,b)$ is said to be {\em bounded}.
By the expression``$a \vee b \in {\cal L}$" 
we mean that $a \vee b$ exists.
A pair $(a,b)$ is said to be {\em comparable} 
if $a \preceq b$ or $b \preceq a$, and {\em incomparable} otherwise.
We say ``{\em $b$ covers $a$}'' 
if $a \prec b$ and there is no 
$c \in {\cal L}$ with $a \prec c \prec b$, 
where $a \prec b$ means $a \preceq b$ and $a \neq b$.
The maximum element (universal upper bound) 
and the minimum element (universal lower bound), if they exist, 
are denoted by ${\bf 1}$ and
${\bf 0}$, respectively. 
For $a \preceq b$, 
the interval $\{ c \in {\cal L} \mid a \preceq c \preceq b\}$
is denoted by $[a,b]$.
A {\em chain} from $a$ to $b$ 
is a sequence $(a = u_0,u_1, u_2,\ldots, u_k =b)$ 
with $u_{i-1} \prec u_{i}$ for $i=1,2,\ldots,k$; 
the number $k$ is the length of the chain.
The length $r[a,b]$ 
of the interval $[a,b]$ is defined as
the maximum length of a chain from $a$ to $b$.
The {\rm rank} $r(a)$ of an element $a$ 
is defined by $r(a) = r[{\bf 0}, a]$.
An {\em atom} is an element of rank $1$.
The {\em covering graph} of a poset ${\cal L}$ is 
the underlying undirected graph of 
the Hasse diagram of ${\cal L}$.

A lattice ${\cal L}$ is called {\em modular} if 
$a \vee (b \wedge c) = (a \vee b) \wedge c$ 
for every $a,b,c \in {\cal L}$ with $a \preceq c$.
Modular lattices are also characterized 
by the modular equality of the rank function.
\begin{Lem}[{See \cite[Chapter III, Corollary 1]{Birkhoff}}]\label{lem:modular} 
A lattice ${\cal L}$ is modular if and only if
\[
r(a) + r(b) = r(a \vee b) + r(a \wedge b) \quad (a,b \in {\cal L}).
\]
\end{Lem}
A lattice ${\cal L}$ is called {\em complemented} 
if for every $p \in {\cal L}$ 
there is an element $q$, called a {\em complement} of $p$, such that $p \vee q = {\bf 1}$ 
and $p \wedge q = {\bf 0}$, 
and {\em relatively complemented}
if $[a,b]$ is complemented for every $a,b \in {\cal L}$ 
with $a \preceq b$.
\begin{Thm}[{See \cite[Chapter IV, Theorem 4.1]{Birkhoff}}]
\label{thm:complemented}
Let ${\cal L}$ be a modular lattice. 
The following conditions are equivalent:
\begin{itemize}
\item[{\rm (1)}] ${\cal L}$ is complemented.
\item[{\rm (2)}] ${\cal L}$ is relatively complemented.
\item[{\rm (3)}] Every element is the join of atoms.
\item[{\rm (4)}] ${\bf 1}$ is the join of atoms.
\end{itemize}
\end{Thm}
\paragraph{Modular semilattice.}
The modularity concept 
has been extended for semilattices by 
Bandelt, van de Vel, and Verheul~\cite{BCK00}. 
A semilattice ${\cal L}$
is said to be {\em modular} if $[{\bf 0},p]$ 
is a modular lattice for every $p \in {\cal L}$, 
and $a \vee b \vee c \in {\cal L}$ provided 
$a \vee b, b \vee c, c \vee a \in {\cal L}$.
A modular semilattice is said to be {\em complemented} 
if $[{\bf 0}, p]$ is a complemented modular lattice 
for every $p \in {\cal L}$.

It is known that
a lattice is modular if and only 
if its covering graph is modular; see \cite[Proposition 6.2.1]{VanDeVel}.
A modular semilattice is characterized by an analogous property as follows.
\begin{Thm}[{\cite[Theorem 5.4]{BVV}}]\label{thm:BVV}
A semilattice is modular if and only 
if its covering graph is modular.
\end{Thm}
The Hasse diagram of ${\cal L}$ is 
admissibly oriented since every 4-cycle is a form of $(p, p \wedge q, q, p \wedge q)$.
\begin{Cor}
The covering graph of a modular semilattice 
is orientable modular.
\end{Cor}

Let ${\cal L}$ be a modular semilattice and 
let $\mGamma$ be the covering graph of ${\cal L}$, which is orientable modular.
An immediate consequence of the Jordan-Dedekind chain condition 
for modular lattices is: 
\begin{myitem}
For $p,q \in {\cal L}$ with $p \preceq q$, 
we have $I(p,q) = [p,q]$ and $d_{\mGamma}(p,q) = r[p,q]$. \label{eqn:I(pq)=[pq]}
\end{myitem}\noindent
A (positive) {\em valuation} of ${\cal L}$ is
is a function on ${\cal L}$ satisfying
\begin{eqnarray}
v(q) - v(p) > 0 &&  (p,q \in {\cal L}: p \prec q), \label{eqn:valuation1}\\
v(p) + v(q) = v(p \wedge q) + v(p \vee q) && (p,q \in {\cal L}: \mbox{bounded}). \label{eqn:valuation2}
\end{eqnarray}
This is a natural extension of a valuation of a modular lattice;
see \cite[Chapter III, 50]{Birkhoff}
(we follow the terminology in the third edition of this book).
In particular, the rank function $r$ is a valuation.
For $p,q$ with $p \preceq q$, let $v[p,q]$ denote $v(q) - v(p)$.
Valuations and orbit-invariant functions are related in the following way.
\begin{Lem}\label{lem:valuation}
\begin{itemize}
\item[{\rm (1)}] For a valuation $v$ on ${\cal L}$, 
the edge-length $h$ on $\mGamma$ defined by
\[
h(pq) := v(q) - v(p) \quad (p,q \in {\cal L}: \mbox{$q$ covers $p$})
\]
is a positive orbit-invariant function, and satisfies
\[
v[p,q] = d_{\mGamma,h}(p,q) \quad (p,q \in {\cal L}: p \preceq q).
\]
\item[{\rm (2)}] For a positive orbit-invariant function $h$ on $\mGamma$, a function $v$ on ${\cal L}$ defined by
\[
v(p) := d_{\mGamma,h}({\bf 0}, p) \quad (p \in {\cal L})
\]
is a valuation.
\end{itemize}
\end{Lem}
\begin{proof}
(1). The positivity of $h$ follows from (\ref{eqn:valuation1}).
The orbit invariance of $h$ follows from (\ref{eqn:valuation2}) and the observation 
that every 4-cycle of $\mGamma$ 
is the form of $(p, p \wedge q, q, p \vee q)$, 
where $p \vee q$ covers $p$ and $q$, and $p \wedge q$ is covered by $p$ and $q$.
We show the latter part by induction on $r[p,q]$; the case $r[p,q] = 1$ is obvious.
Take $p' \in [p,q] = I(p,q)$ such that $p'$ covers $p$.
By induction, we have $d_{\mGamma, h}(p',q) = v[p',q]$.
By Lemma~\ref{lem:shortest} and (\ref{eqn:I(pq)=[pq]}), we have 
$d_{\mGamma,h}(p,q) =  h(pp') + d_{\mGamma, h}(p',q)  = v[p,q]$.

(2). By Lemma~\ref{lem:shortest} and (\ref{eqn:I(pq)=[pq]}), 
if $q$ covers $p$, then $v(q) - v(p) = h(pq) > 0$, implying (\ref{eqn:valuation1}).
For a bounded pair $(p,q)$, take maximal chains 
$(p \wedge q = p_0,p_1,\ldots,p_k = p)$ and $(p \wedge q = q_0,q_1,\ldots,q_l = q)$.
Let $a_{i,j} := p_{i} \vee q_{j}$. 
By modularity, we see that $a_{i+1,j+1}$ covers $a_{i+1,j}$ and $a_{i,j+1}$, 
and $a_{i,j}$ is covered by $a_{i+1,j}$ and $a_{i,j+1}$; 
in particular $a_{i+1,j+1} = a_{i+1,j} \vee a_{i,j+1}$ and 
$a_{i,j} = a_{i+1,j} \wedge a_{i,j+1}$.
Therefore
$v(p) + v(q) - v(p \wedge q) - v(p \vee q)
= \sum_{i,j} (v(a_{i+1,j}) + v(a_{i,j+1}) - v(a_{i+1,j+1}) - v(a_{i,j})) 
= \sum_{i,j} (h(a_{i+1,j}a_{i,j}) - h(a_{i,j+1}a_{i+1,j+1}))$.
By the orbit invariance of $h$, all summands are zero, implying (\ref{eqn:valuation2}).
\end{proof}

Consider the case where ${\cal L}$ 
is the product ${\cal L}_1 \times {\cal L}_2$ of two modular semilattices ${\cal L}_1, {\cal L}_2$.
For a valuation $v$ on ${\cal L}$, define $v_i:{\cal L}_i \to \RR$ $(i=1,2)$ by
 \begin{equation}\label{eqn:valuation_summand}
 v_1(p_1) := v(p_1, {\bf 0}) \quad (p_1 \in {\cal L}_1), 
 \quad  v_2(p_2) := v({\bf 0},p_2) \quad (p_2 \in {\cal L}_2).
 \end{equation}
Then $v_i$ is a valuation on ${\cal L}_i$ for $i=1,2$, and satisfies
 \begin{equation}\label{eqn:valuation_summand1}
 v(p) = v_1(p_1) + v_2(p_2) - v({\bf 0}) \quad (p = (p_1,p_2) \in {\cal L}).
 \end{equation}
Conversely, for a valuation $v_i$ on ${\cal L}_i$ $(i=1,2)$, 
define $v:{\cal L} \to \RR$ by
\begin{equation}\label{eqn:valuation_sum}
v(p) := v_1(p_1) + v_2(p_2) \quad (p = (p_1,p_2) \in {\cal L}). 
\end{equation}
Then $v$ is a valuation on ${\cal L}$.

In the sequel,
a modular semilattice ${\cal L}$ is supposed 
to be endowed with some valuation $v$.
If ${\cal L}$ is the product of modular semilattices ${\cal L}_i$, 
then the valuation of each ${\cal L}_i$ is defined by (\ref{eqn:valuation_summand}), 
and is also denoted by $v$.
For modular semilattices ${\cal L}_1$ and ${\cal L}_2$, 
the valuation $v$ of ${\cal L}_1 \times {\cal L}_2$ 
is defined to be the sum of valuations of ${\cal L}_1$ and ${\cal L}_2$ 
according to (\ref{eqn:valuation_sum}).
Also a modular semilattice ${\cal L}$ is regarded as a metric space 
by the shortest path metric of its covering graph $\mGamma$
with respect to a positive orbit-invariant function in Lemma~\ref{lem:valuation}~(1).
The corresponding metric function is denoted by $d = d_{{\cal L}}$.
We give basic properties of metric intervals of ${\cal L}$.
\begin{Lem}~\label{lem:I(p,q)}
For $p,q \in {\cal L}$, we have the following.
\begin{itemize}
\item[{\rm (1)}] $d(p,q) = v[p \wedge q, p] + v[p \wedge q, q]$.
\item[{\rm (2)}] $
I(p,q) = \{ a \vee b \mid a \in [p \wedge q,p], b \in [p \wedge q, q]: \mbox{$a \vee b$ exists} \}$.
\item[{\rm (3)}] If $c = a \vee b$ for 
$a \in [p \wedge q, p], b \in [p \wedge q, q]$,
then $a = p \wedge c$ and $b = q \wedge c$.
\item[{\rm (4)}] For $u,u' \in I(p,q)$, it holds $u \wedge u' =
 (u \wedge u' \wedge p) \vee (u \wedge u' \wedge q)$; in particular $u \wedge u' \in I(p,q)$.
\end{itemize}
\end{Lem}
The properties (1), (2), and (3) appeared (implicitly) in \cite{BVV}.
\begin{proof} 
By Lemmas~\ref{lem:shortest} and \ref{lem:valuation}, 
we can assume that valuation $v$ is equal to the rank function $r$.

(3). Necessarily $a \preceq p \wedge c$ and
$b \preceq q \wedge c$, 
implying $a \vee b \preceq (p \wedge c) \vee (q \wedge c) \preceq c = a \vee b$.
Hence $(p \wedge c) \vee (q \wedge c) = c$.
Also $(p \wedge c) \wedge (q \wedge c) = p \wedge q$ 
(since $p \wedge q \preceq c$).
By the modularity equality, we have
$r(a) + r(b) = r(c) + r(p \wedge q) = r(p \wedge c) + r(q \wedge c)$,
which implies $r[a, p \wedge c] = r[b, q \wedge c] = 0$.
Thus $p \wedge c = a$ 
and $q \wedge c = a$ must hold.

(1). We use the induction on $d(p,q)$.
Take a neighbor $q'$ of $q$ in $I(p,q)$. 
By induction $d(p,q') = r[p \wedge q',p] + r[p \wedge q',q']$,  
and either (i) $q$ covers $q'$ or (ii) $q'$ covers $q$.
In the first case (i), we must have $p \wedge q \preceq q'$. 
Suppose not. Then $(p \wedge q) \vee q' = q$, and $(p \wedge q) \wedge q' = p \wedge q'$.
The modularity equality yields
$r[p \wedge q,q] = r[p \wedge q',q']$, which 
means that there is a $(p,q)$-path
passing through $p \wedge q$
with the length shorter than $d(p,q')$, 
contradicting $d(p,q) = d(p,q') + 1$.
It follows from $p \wedge q \preceq q'$ that $p \wedge q = p \wedge q'$, 
and the claim follows.
In the second case (ii) where $q'$ covers $q$, 
$p \wedge q'$ covers $p \wedge q$; since 
otherwise $p \wedge q' = p \wedge q$ which leads to a contradiction  
$d(p,q') > d(p,q)$.
By the modularity equality 
$r[p \wedge q',q'] = r[p \wedge q,q]$ and
the claim follows.

(2).  By (1), $p \wedge q \in I(p,q)$. 
By the modularity equality 
we have $(\supseteq)$.
We show the reverse inclusion. Take $u \in I(p,q)$.
Let $a := u \wedge p$ and $b := u \wedge q$.
By (1), $d(p,q) = d(p,u) + d(u,q) = d(p,a) + d(a,u) + d(u,b) + d(b,q)$.
Necessarily $u \in I(a,b)$, and $d(a,b) = r[a,u] + r[b,u]$.
Since $d(a,b) = r[a \wedge b, a] + r[a \wedge b, b] 
= r[a, a \vee b] + r[b, a \vee b] = d(a,b) - 2 r[a \vee b, u]$,
we have $r[a \vee b, u] = 0$, implying $a \vee b = u$.
Also $d(p,q) = d(p,a) + d(a,b) + d(b,q)$ must hold.
Hence $d(p,q) = d(p,a) + d(a,a \wedge b) + d(a \wedge b,b)+ d(b,q) 
= r[a \wedge b, p] + r[a \wedge b,q] = d(p,q) + 2 r[a \wedge b, p \wedge q]$.
 This implies $a \wedge b = p \wedge q$, and $a \in [p \wedge q,p]$ and $b \in [p \wedge q, p]$.

(4). First we note that $I(p,q)$ is an isometric subspace (with respect to $r$).
Indeed, 
for $w,w' \in I(p,q)$, by Theorem~\ref{thm:BVV}, there is a median $m$ of $p,w,w'$.
In particular $m \in I(p,q)$. So $I(w,m) \cup I(m,w') \subseteq I(p,q)$.
This means that $w$ and $w'$ are joined by a path in $I(p,q)$ of length $d(w,w')$.
By (2) and (3),
if $w$ covers $w'$, 
then $w \wedge p$ covers $w' \wedge p$ and $w \wedge q = w' \wedge q$, 
or $w \wedge q$ covers $w' \wedge q$ and $w \wedge p = w' \wedge p$. 
Thus a shortest path $P$ between $u$ and $u'$ in $I(p,q)$ induces a path $P'$
between $a := u \wedge p$ and $a' := u' \wedge p$ 
(by map $w \mapsto w \wedge p$)
and a path $P''$ 
between $b := u \wedge q$ and $b' := u' \wedge q$ (by map $w \mapsto w \wedge q$).
The length of $P$ is the sum of lengths of $P'$ and of $P''$.
This implies that 
\begin{eqnarray*}
d(u,u') &\geq & d(a,a') + d(b, b') = r(a) + r(a') - 2 r(a \wedge a') + r(b) + r(b') - 2 r (b \wedge b') \\
&= & r(a \vee b) + r(a' \vee b') - 2 r((a \wedge a') \vee (b \wedge b')) \\
& = & d(u,u') + 2 r[(a \wedge a') \vee (b \wedge b'), u \wedge u'],
\end{eqnarray*}
where the second equality follows from the modularity equality with 
$a \wedge b = a' \wedge b' = p \wedge q$ and the third follows from (1).
Hence $u \wedge u' = (a \wedge a') \vee (b \wedge b')$, as required.
\end{proof}

A subset $X$ of ${\cal L}$ is called a {\em subsemilattice} if $p \wedge q \in X$ 
for any $p,q \in X$, and is called {\em convex} if $X$ is a convex set in $\mGamma$.
For an edge-set $U$, 
define ${\cal L}|U \subseteq {\cal L}$ by
\begin{equation}\label{eqn:L|U}
{\cal L}|U := \{ p \in {\cal L} \mid 
\mbox{any shortest path from ${\bf 0}$ to $p$ belongs to $U$} \}.
\end{equation}
\begin{Lem}\label{lem:interval}
\begin{itemize}
\item[{\rm (1)}] Any convex set in ${\cal L}$
is a modular subsemilattice of ${\cal L}$.
\item[{\rm (2)}] Suppose that ${\cal L}$ is a lattice. 
Then a subset $C$ is convex 
if and only if $C = [a,b]$ 
for some $a,b \in {\cal L}$ with $a \preceq b$.
\item[{\rm (3)}] Suppose that ${\cal L}$ is complemented.
For an orbit-union $U$,  ${\cal L}|U$ is convex, 
and is a complemented modular subsemilattice of ${\cal L}$.
For $p \in {\cal L}$, define $p|U \in {\cal L}|U$ by
\[
p|U := \mbox{the gate of $p$ at ${\cal L}|U$}.
\]
Then $p|U \preceq p$, and any shortest path between $p$ and $p|U$ does not meet $U$.
\end{itemize}
\end{Lem}
\begin{proof}
(1) follows from Lemma~\ref{lem:I(p,q)}.
The if part of (2) also follows from Lemma~\ref{lem:I(p,q)}. 
To see the only if part of (2), 
consider $a:=\bigwedge_{u \in C} u$ 
and $b := \bigvee_{u \in C} u$. Then $C \subseteq [a,b]$.
From $a,b \in C$, we have $[a,b]  = I(a,b) \subseteq C$.

(3). For $p,q \in {\cal L}|U$,
there is a shortest path from ${\bf 0}$ to $p$ (or $q$) 
passing through $p \wedge q$. 
This means $[p \wedge q,p], [p \wedge q, q] \subseteq {\cal L}|U$.
By Lemma~\ref{lem:I(p,q)} and Lemma~\ref{lem:orbits}~(2), 
it holds that $I(p,q) \subseteq {\cal L}|U$. Hence ${\cal L}|U$ is convex, 
and is a modular subsemilattice by (1).
Since $[p,q] = I(p,q) \subseteq {\cal L}|U$ for $p,q \in {\cal L}|U$ with $p \preceq q$, 
every interval of  ${\cal L}|U$ is complemented.%

By ${\bf 0} \in {\cal L}|U$ and the definition of gates, 
we have $p|U \in I({\bf 0},p) = [{\bf 0}, p]$. Hence $p|U \preceq p$.
Suppose that there is a shortest path from $p|U$ to $p$ 
having an edge $st$ in $U$.
Suppose that $s$ is covered by $t$. 
By the relative complementarity of $[p|U, p]$, there is $s'$ 
such that $s' \vee s = t$ and $s \wedge s' = p|U$.
Then $s'$ covers $p|U$. So $d(p|U, t) = d(s',t) + 1$ and
$d(p,p|U) = d(p,s') + 1$ hold.
In particular, edges $st$ and $(p|U)s'$ are projective (Lemma~\ref{lem:commonorbit}).
Thus $(p|U)s' \in U$, and $s' \in {\cal L}|U$.
By definition of gates,
we have $d(p,s') = d(p,p|U) + 1$, contradicting $d(p,p|U) = d(p,s') + 1$.
\end{proof}

\section{Submodular function on modular semilattice}\label{sec:submo}
In this section, we develop a theory of submodular functions 
on modular semilattices.
A modular semilattice ${\cal L}$ is not necessarily a lattice. 
Join $p \vee q$ of elements $p, q$ may or may not exist. 
Interestingly, we can define a certain kind of a join, 
called a {\em fractional join}, 
which is a formal convex combination of elements of a set 
${\cal E}(p, q) \subseteq {\cal L}$ determined by $(p, q)$:
\[
\sum_{u \in {\cal E}(p, q)} c(u; p,q) u.
\]
If $p, q$ have the join $p \vee q$, 
then the fractional join is equal to $1(p \vee q)$.
In Section~\ref{subsec:frac_join}, 
the set ${\cal E}(p,q)$ and the coefficient $c(u;p,q)$ are introduced.
%
%
Then, in Section~\ref{subsec:submodular}, 
a function $f:{\cal L} \to \overline{\RR}$ is defined to be 
a {\em submodular function} if it satisfies
\[
f(p) + f(q) \geq f(p \wedge q) + \sum_{u \in {\cal E}(p,q)} c(u;p,q) f(u) \quad (p,q \in {\cal L}).
\]
The main properties of our submodular functions are:
\begin{itemize}
\item The distance function $d = d_{{\cal L}}$ on ${\cal L}$ 
is submodular on ${\cal L} \times {\cal L}$ (Theorem~\ref{thm:d_is_submodular}).
\item Submodular functions admit a fractional polymorphism containing semilattice operation $\wedge$, and hence {\bf VCSP}$[\mLambda]$ for submodular language $\mLambda$ 
can be solved in polynomial time by the basic LP relaxation
(Theorem~\ref{thm:minimized}).
\end{itemize}
For readability, less obvious theorems will be proved in Section~\ref{subsec:submo_proof}.

Although our framework for submodularity 
was motivated by its application to 0-extension problems,
it turned out that several other submodular-type functions, mentioned in the introduction, 
fall into our framework; 
see \cite{HH_HJ, HHprepar} for detail. 

\subsection{Fractional join}\label{subsec:frac_join}
Let ${\cal L}$ be a modular semilattice, where its valuation  is denoted by $v$.
We begin by sketching
the construction of the fractional join of $p,q$; see Figure~\ref{fig:ConvI(p,q)}. 
By valuation $v$ and the expression in Lemma~\ref{lem:I(p,q)}~(3), 
the metric interval $I(p,q)$ is naturally mapped to the plane $\RR^2$.
Consider the convex hull $\Conv I(p,q)$ of the image of $I(p,q)$.
Then the fractional join is a formal sum of elements $u \in I(p,q)$ mapped to 
maximal extreme points of $\Conv I(p,q)$; 
the set of such elements is called the {\em $(p,q)$-envelope}.
The coefficient of $u$ is determined by the normal cone 
$C(u;p,q)$ at (the image of) $u$.
\begin{figure} 
\begin{center}
\includegraphics[scale=0.55]{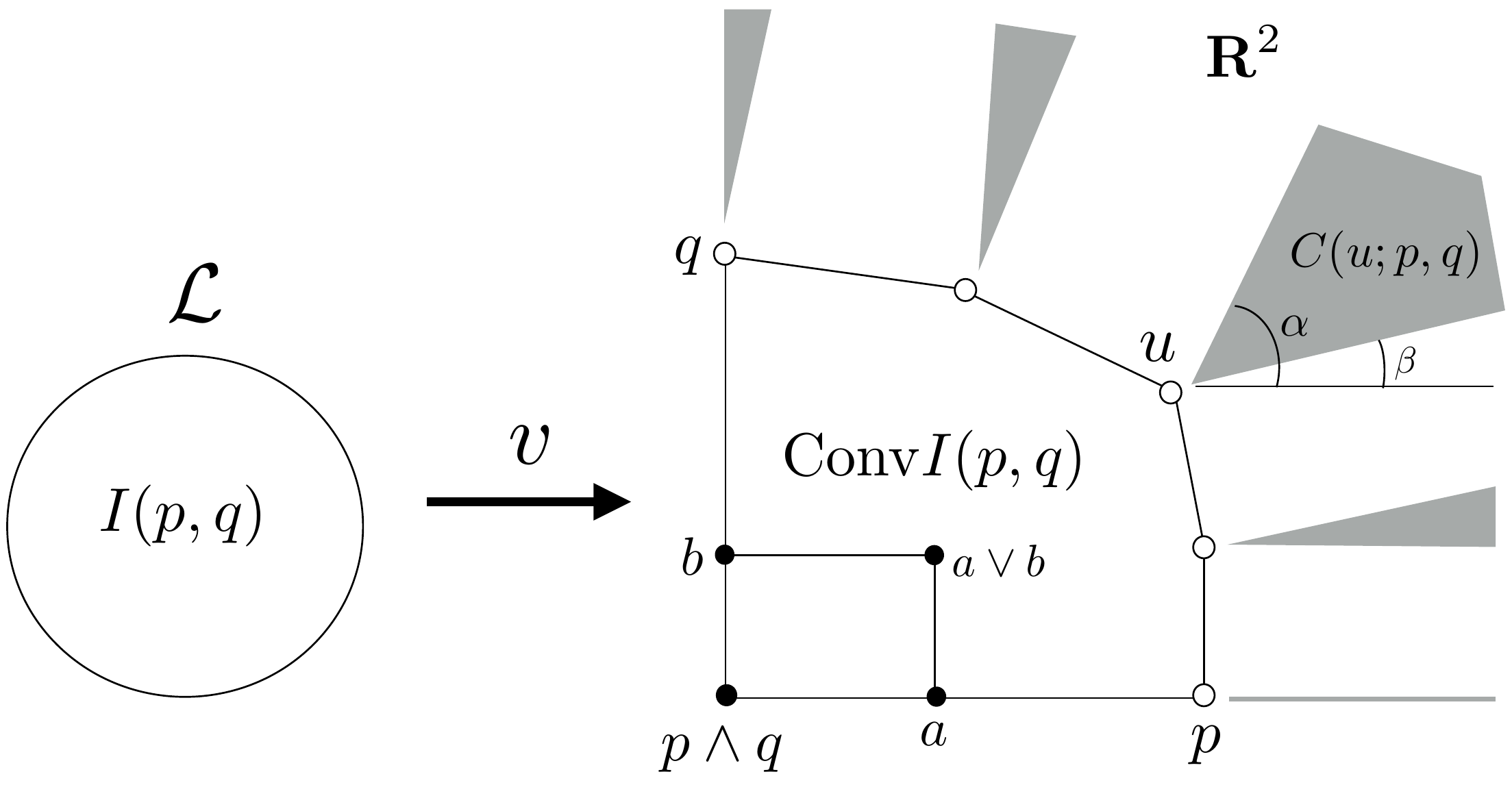}                
\caption{Construction of fractional join}  
\label{fig:ConvI(p,q)}
\end{center}
\end{figure}

\paragraph{$(p,q)$-envelope.}
First we introduce the concept of the $(p,q)$-envelope.
Let $(p,q)$ be a pair of elements in ${\cal L}$.
Define vector $v (u; p,q)$ in $\RR^2_+$ by 
\begin{equation}
v(u; p,q) := (v[p \wedge q, u \wedge p], v[p \wedge q, u \wedge q]).
\end{equation}
Let $\Conv I(p,q)$ denote the convex hull 
of $\{ v(u; p,q) \mid u \in I(p,q)\}$ in $\RR^2_+$.
The polygon $\Conv I(p,q)$ contains
$v(p \wedge q; p,q) = (0,0)$, 
$v(p;p,q) = (v[p \wedge q,p], 0)$, 
and $v(q;p,q) = (0, v[p \wedge q,q])$ as extreme points, 
and contains horizontal segment $[v(p \wedge q; p,q), v(p;p,q)]$ 
and vertical segment $[v(p \wedge q; p,q), v(q;p,q)]$
as edges. Also $\Conv I(p,q)$ is contained 
in the rectangle of four vertices 
\[
(0,0), (v[p \wedge q,p], 0),  (v[p \wedge q,p], 0),  (v[p \wedge q,p], v[p \wedge q,q]).
\]

The {\em $(p,q)$-envelope} ${\cal E}({p,q})$ is 
the set of elements $u \in I(p,q)$
such that $v(u;p,q)$ is 
a maximal extreme point of $\Conv I(p,q)$, where
a {\em maximal} extreme point is an extreme point $z$ in $\Conv I(p,q)$ 
such that for every positive vector $\epsilon$
it holds $z + \epsilon \not \in \Conv I(p,q)$.
Observe that  ${\cal E}({p,q})$ always contains $p$ and $q$.
\begin{Lem}\label{lem:unique}
The map $u \mapsto v (u; p,q)$ is injective on ${\cal E}(p,q)$. 
\end{Lem}
This lemma will be proved in Section~\ref{subsub:unique}.
Hence the map $u \mapsto v (u; p,q)$ is a bijection 
between ${\cal E}({p,q})$ 
and the set of maximal extreme points of $\Conv I(p,q)$.

\paragraph{Valuation of convex cones in $\RR^2_+$.}
To define the coefficient, we consider a valuation of convex cones in $\RR^2_+$.
Every closed convex cone $C (\neq \{0\})$ in $\RR^2_{+}$ 
is uniquely represented as
\[
C = \{ (x,y) \in \RR^{2}_{+} \mid 
- x \sin \alpha + y \cos \alpha \leq 0,\ 
- x \sin \beta  + y \cos \beta  \geq 0 \} 
\]
for some $0 \leq \beta \leq \alpha \leq \pi/2$.
Define $[C]$ by
\begin{equation}\label{eqn:nu(C)}
[C] :=  \frac{\sin \alpha}{ \cos \alpha + \sin \alpha } 
- \frac{\sin \beta }{ \cos \beta + \sin \beta }.
\end{equation}
For convention, we let
$
[\{0\}] := 0. 
$
The following property is easily verified, 
where $C,C'$ are closed convex cones in $\RR^2_+$:
\begin{myitem1}
\item[(1)] $[C] \geq 0$, and $[C] > 0$ if and only if $C$ is full dimensional.
\item[(2)] $[C] + [C'] = [C \cap C'] + [C \cup C']$ 
for $C \cap C' \neq \emptyset$.
\item[(3)] $[\RR^2_+] = 1$. \label{eqn:valuation}
\end{myitem1}\noindent
We are now ready to define the fractional join.
\paragraph{Fractional join.}
%
For $u \in I(p,q)$, 
let $C(u;p,q)$ denote the set of 
nonnegative vectors $w \in \RR^2_{+}$
with $\langle w, v (u;p,q) \rangle  
= \max_{u' \in  I(p,q)} \langle w, v (u';p,q) \rangle$, 
where $\langle , \rangle$ is the standard inner product.
Namely $C(u;p,q)$ is the intersection of $\RR^2_+$ and 
the normal cone of $\Conv I(p,q)$ at extreme point $v (u;p,q)$.
In particular, $C(u;p,q)$ forms 
a closed convex cone in $\RR^2_{+}$.
The {\em fractional join} of $(p,q)$ is the formal convex combination 
of $u \in {\cal E}(p,q)$ with coefficient $[C(u;p,q)]$:
\[
\sum_{u \in {\cal E}(p,q)} [C(u;p,q)] u.
\]
Obviously it holds $[C(u;p,q)] = 0$ for $u \in I(p,q) \setminus {\cal E}(p,q)$  
since $C(u;p,q) = \{0\}$. So the fractional join is also equal to 
$\sum_{u \in I(p,q)} [C(u;p,q)] u$.
Note that the set ${\cal E}(p,q)$ and the coefficient $[C(u_i;p,q)]$
depend on valuation $v$.
Since the set of cones $C(u;p,q)$ $(u \in {\cal E}(p,q))$ forms 
the intersection of $\RR^2_+$ and
the normal fan of $\Conv I(p,q)$, we have
\begin{myitem1}
\item[(1)] $\RR^2_{+}  =  \bigcup_{u \in {\cal E}(p,q)} C(u; p,q)$. 
\item[(2)] For distinct $u, u' \in {\cal E}(p,q)$, 
the intersection $C(u;p,q) \cap C(u'; p,q)$ has no interior point, 
and hence $[C(u;p,q) \cap C(u'; p,q)] = 0$.
\item[(3)] $\bigcup_{u \in {\cal E}(p,q)} [C(u; p,q)] = 1$. \label{eqn:normal_fan}
\end{myitem1}\noindent
Therefore the fractional join of $p,q$ 
is a formal convex combination of elements in ${\cal E}(p,q)$.
An explicit formula of $[C(u;p,q)]$ is given as follows.
\begin{Lem}\label{lem:[C]}
Suppose that ${\cal E}(p,q) = \{p = u_0,u_1,u_2,\ldots,u_m = q\}$,  
and $v(u_{i};p,q)$ and $v(u_{i+1};p,q)$ are adjacent extreme points in 
$\Conv I(p,q)$. Then we have 
\[
[C(u_i;p,q)] = \delta_i - \delta_{i-1} \quad (i=0,1,2,\ldots,m),
\]
where $\delta_i$ is defined by $\delta_{-1} := 0$, $\delta_m := 1$, and
\[
\delta_i := \frac{v[u_i \wedge u_{i+1}, u_i]}{v[u_i \wedge u_{i+1}, u_i] + v[u_i \wedge u_{i+1}, u_{i+1}]} \quad (i=1,2,\ldots,m-1).
\]
\end{Lem}
This lemma will be proved in Section~\ref{subsub:unique}.
We next define the {\em fractional join operation}. 
Let ${\cal O}({\cal L})$ denote the set 
of all binary operations ${\vartheta}: {\cal L} \times {\cal L} \to {\cal L}$.
Regard ${\cal O}({\cal L})$ as a poset by the order: $\vartheta \preceq \vartheta'$ 
if $\vartheta(p,q) \preceq \vartheta'(p,q)$ for all $p,q$.
Then ${\cal O}({\cal L})$ is isomorphic to the product 
${\cal L}^{{\cal L} \times {\cal L}}$ of ${\cal L}$ by the correspondence:
\begin{equation}\label{eqn:correspondence}
{\cal O}({\cal L}) \ni \vartheta \longleftrightarrow  (\vartheta(p,q): p,q \in {\cal L}) \in {\cal L}^{{\cal L} \times {\cal L}}.
\end{equation}
In particular ${\cal O}({\cal L})$ is also a modular semilattice, where 
the meet $\wedge$ is
given by $(\vartheta \wedge \vartheta' )(p,q) := \vartheta (p,q) \wedge \vartheta' (p,q)$.
The valuation of ${\cal O}({\cal L})$ is given according to (\ref{eqn:valuation_sum}), and is also denoted by $v$.
Let $L, R$ be the {\em projection operations} defined by
\begin{eqnarray*}
L(p,q) := p, \quad R(p,q) := q \quad (p,q \in {\cal L}).
\end{eqnarray*}
Let $I({\cal L}) := I(L,R)$, and
let ${\cal E}({\cal L})$ 
denote the $(L,R)$-envelope ${\cal E}({L,R})$ in ${\cal O}({\cal L})$.
An operation in ${\cal E}({\cal L})$ is called {\em extremal}.
For an operation $\vartheta$ in $I({\cal L})$, 
the cone $C(\vartheta; L,R)$ is denoted simply by $C({\vartheta})$.
The {\em fractional join operation} is the formal sum of extremal operations $\vartheta$ 
with coefficient $[C(\vartheta)]$:
\[
\sum_{\vartheta \in {\cal E}({\cal L})} [C(\vartheta)] \vartheta.
\]
The fractional join operation is nothing but the fractional join of $L,R$ in ${\cal O}({\cal L})$, and
indeed gives fractional joins in ${\cal L}$.
\begin{Prop}\label{prop:frac_join}
$\displaystyle
\sum_{u \in {\cal E}(p,q)} [C(u;p,q)] u = \sum_{\vartheta \in {\cal E}({\cal L})} [C(\vartheta)] \vartheta (p,q) \quad (p,q \in {\cal L}).
$
\end{Prop}
The proof is given in Section~\ref{subsub:frac_join}.
Consider the case where
${\cal L}$ is the product of modular semilattices ${\cal L}_i$ for $i=1,2,\ldots,n$.
For operations $\vartheta_i$ on ${\cal L}_i$ $(i=1,2,\ldots,n)$, 
the {\em componentwise extension} $(\vartheta_1,\vartheta_2,\ldots,\vartheta_n)$ 
is an operation 
on ${\cal L}$ defined by 
\begin{eqnarray*}
(\vartheta_1,\vartheta_2,\ldots,\vartheta_n)(p,q) &= &
(\vartheta_1(p_1,q_1),\vartheta_2(p_2,q_2),\ldots,\vartheta_n (p_n,q_n)) \\
&& \quad \quad \quad (p,q \in {\cal L} = {\cal L}_1 \times {\cal L}_2 \times \cdots \times {\cal L}_n).
\end{eqnarray*}
\begin{Prop}\label{prop:frac_join_prod} 
\[
\sum_{\vartheta \in {\cal E}({\cal L})} [C(\vartheta)] \vartheta =
\sum_{\vartheta_1,\vartheta_2,\ldots,\vartheta_n}  [C(\vartheta_1) \cap C(\vartheta_2) \cap \cdots \cap C(\vartheta_n)] (\vartheta_1,\vartheta_2,\ldots, \vartheta_n),
\]
where $\vartheta_i$ is taken over all extremal operations in ${\cal L}_i$ $(i=1,2,\ldots,n)$.
Moreover, 
if ${\cal L}_i = {\cal L}_j$ and $\vartheta_i \neq \vartheta_j$ for some $i,j$, 
then  $[C(\vartheta_1) \cap C(\vartheta_2) \cap \cdots \cap C(\vartheta_n)] = 0$.
\end{Prop}
The proof is given in Section~\ref{subsub:frac_join_prod}.
In particular, any extremal operation $\vartheta$ in ${\cal L}$ is 
the componentwise extension of extremal operations $\theta_i$ in ${\cal L}_i$ 
for $i=1,2,\ldots,n$.

\subsection{Submodular function}\label{subsec:submodular}
Let ${\cal L}$ be a modular semilattice with valuation $v$.
A function $f: {\cal L} \to \overline{\RR}$ 
is called {\em submodular} on ${\cal L}$ (with respect to $v$)
if it satisfies 
\begin{equation}\label{eqn:gsubmo}
f(p) + f(q) \geq f(p \wedge q) + \sum_{u \in {\cal E}(p,q)} [C(u;p,q)] f(u) \quad ( p, q \in {\cal L}).
\end{equation}
By Proposition~\ref{prop:frac_join},  
a submodular function may also be characterized as a function $f$ satisfying
\begin{equation}\label{eqn:polymor}
f(p) + f(q) \geq f(p \wedge q) + \sum_{\vartheta \in {\cal E}({\cal L})} 
[C(\vartheta)] f(\vartheta(p, q)) 
\quad (p,q \in {\cal L}).
\end{equation}
In the case where a pair $(p,q)$ is bounded, 
the join $(p,q)$ exists, $\Conv I(p,q)$ is a square
of vertices $(0, 0), (0, v[p \wedge q, q]), (v[p \wedge q,p], 0), 
(v[p \wedge q,p], v[p \wedge q, q])$, ${\cal E}(p,q) = \{p, p \vee q, q\}$, and 
the fractional join is equal to $0 p + 1 (p \vee q)  + 0 q = p \vee q$.
See Figure~\ref{fig:bounded_antipodal}.
\begin{figure} 
\begin{center}
\includegraphics[scale=0.35]{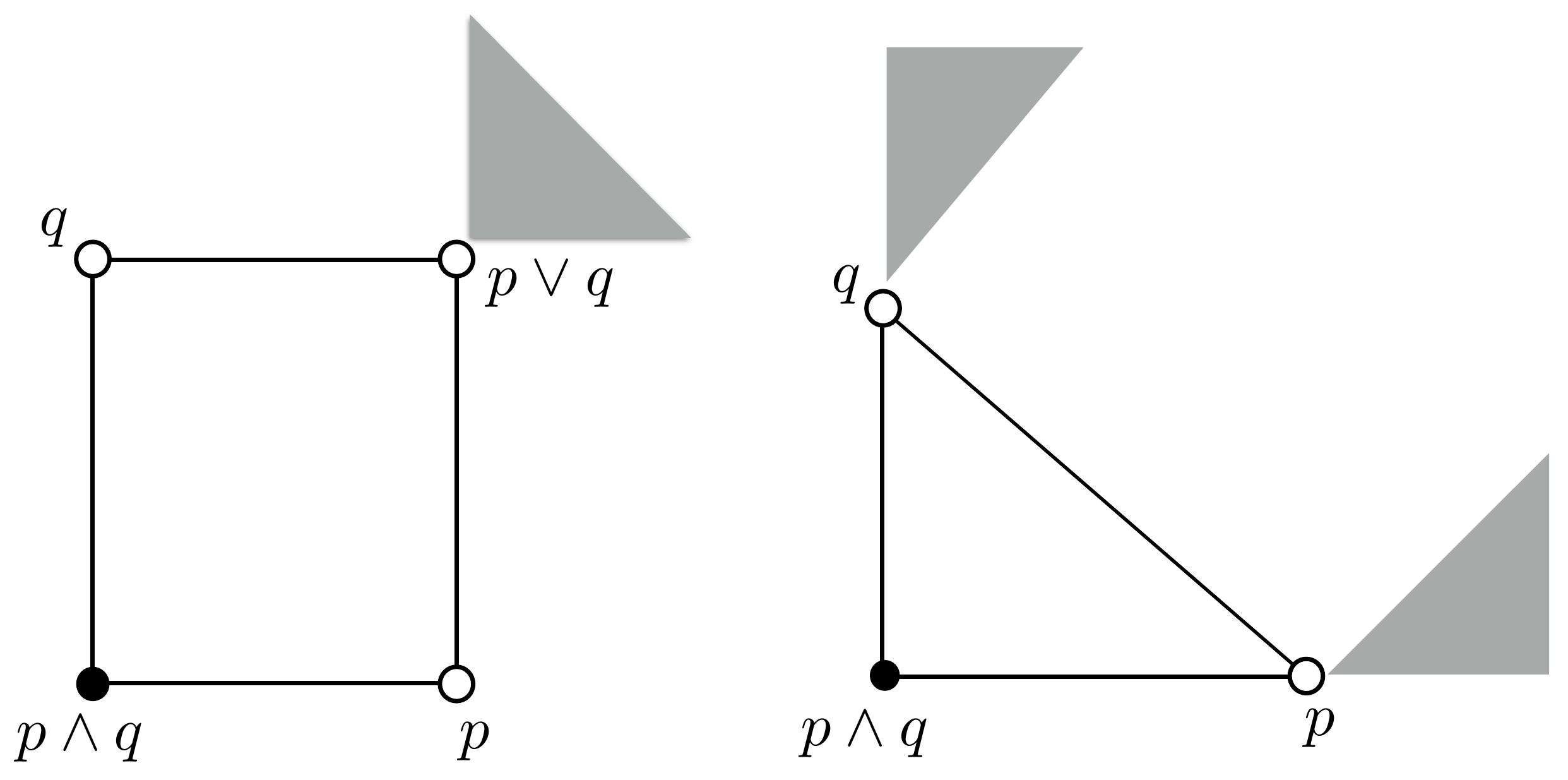}                
\caption{$\Conv I(p,q)$ for a bounded pair (left) and an antipodal pair (right)}  
\label{fig:bounded_antipodal}
\end{center}
\end{figure}
Hence the corresponding inequality in (\ref{eqn:gsubmo}) is equal to  
the usual submodularity inequality:
\begin{equation}\label{eqn:submodular}
f(p) + f(q) \geq f(p \wedge q) + f(p \vee q).
\end{equation}
Another extremal case is the case of ${\cal E}(p,q) = \{p,q\}$,  
which implies that $\Conv I(p,q)$ is the triangle with vertices 
$(0,0), (v[p \wedge q,p],0), (0, v[p \wedge q,q])$.
Such a pair $(p,q)$ is called  {\em antipodal}. 
By definition, a pair $(p,q)$ is antipodal if and only if 
every bounded pair $(a,b)$ with $p \succeq a \succeq p \wedge q \preceq b \preceq q$
satisfies
\begin{equation}\label{eqn:antipodal}
v[a,p]v[b,q] \geq v[p \wedge q, a]v[p \wedge q, b].
\end{equation}
This condition rephrases that 
the point $v(a \vee b; p,q)$  is lower than 
the line through the points $v(p;p,q)$ and $v(q;p,q)$.
In this case, the fractional join of $(p,q)$ 
is equal to 
\[
\frac{v[p \wedge q, p]}{v[p \wedge q,p] + v[p \wedge q, q]} p + 
\frac{v[p \wedge q, q]}{v[p \wedge q,p] + v[p \wedge q, q]} q.
\]
The inequality in (\ref{eqn:gsubmo}) corresponds to
\begin{equation}\label{eqn:^-convexity}
v[p \wedge q,q] f(p) +  v[p \wedge q,p] f(q) \geq 
(v[p \wedge q,p] + v[p \wedge q,q])f(p \wedge q).
\end{equation} 
We call this inequality the {\em $\wedge$-convexity inequality}.
Then inequalities (\ref{eqn:submodular}) and (\ref{eqn:^-convexity}) 
suffice to characterize the submodularity: 
\begin{Thm}\label{thm:submo}
$f:{\cal L} \to \overline{\RR}$ is submodular if and only if it satisfies 
\begin{itemize}
\item [{\rm (1)}] ${\cal E}(p,q) \subseteq \dom f$ for $p,q \in \dom f$,
\item [{\rm (2)}] the submodularity inequality for every bounded pair $(p,q)$, and
\item [{\rm (3)}] the $\wedge$-convexity inequality for every antipodal pair $(p,q)$.
\end{itemize}
\end{Thm}
The proof is given in Section~\ref{subsub:unique}.
One of the main properties of our submodularity is the following:
\begin{Thm}\label{thm:d_is_submodular}
Let ${\cal L}$ be a modular semilattice.
The distance function $d_{{\cal L}}$ on ${\cal L}$ is submodular on ${\cal L} \times {\cal L}$.
\end{Thm}
We will prove this theorem and a more general version (Theorem~\ref{thm:d_is_L-convex}) 
in Section~\ref{subsub:d_is_L-convex}.
We note some basic properties of our submodular functions concerning 
addition and restriction.
\begin{Lem}\label{lem:properties_submo}
Let ${\cal L}, {\cal L}', {\cal M}$ and ${\cal M}'$ be modular semilattices, and let
$f$ and $f'$ be submodular functions on ${\cal L}$.
\begin{itemize}
\item[{\rm (1)}] For $b \in \RR$ and $c,c' \in \RR_+$,  
$b + c f + c' f'$ is a submodular function on ${\cal L}$.
\item[{\rm (2)}] A function $\tilde f$ defined by 
\[
\tilde f(p,p') := f(p) \quad ( (p,p') \in {\cal L} \times {\cal L}')
\]
is a submodular function on ${\cal L} \times {\cal L}'$.
\item[{\rm (3)}] Suppose ${\cal L} = {\cal M} \times {\cal M}'$. 
For any $p' \in {\cal M}'$, a function $f_{p'}$ defined by 
\[
f_{p'}(p)  :=  f(p,p') \quad (p \in {\cal M})
\] 
is a submodular function on ${\cal M}$.
\item[{\rm (4)}] For a convex set ${\cal N}$ of ${\cal L}$, 
the restriction of $f$ to ${\cal N}$ is a submodular function on~${\cal N}$ 
(regarded as a modular semilattice). 
\end{itemize}
\end{Lem}
\begin{proof}
(1) follows from the facts that the submodularity is closed under nonnegative sum, and
that any constant function is submodular (by (\ref{eqn:normal_fan})~(3)).

(2) follows from (\ref{eqn:polymor}), Proposition~\ref{prop:frac_join_prod}, and
\begin{eqnarray*}
&& \tilde f(p,p') + \tilde f(q,q') = f(p) + f(q)  \geq f(p \wedge q) + \sum_{\vartheta \in {\cal E}({\cal L})} [C(\vartheta)] f(\vartheta(p,q)) \\
&& = \tilde f(p \wedge q, p' \wedge q') + \sum_{\vartheta \in {\cal E}({\cal L})} 
\sum_{\vartheta' \in {\cal E}({\cal L}')} [C(\vartheta) \cap C(\vartheta')] \tilde f(\vartheta(p,q), \vartheta'(p',q')),
\end{eqnarray*}
where we use $[C(\vartheta)] = \sum_{\vartheta' \in {\cal E}({\cal L}')}[C(\vartheta) \cap C(\vartheta')]$ 
(obtained from (\ref{eqn:normal_fan})).

(3). Notice that the $(p,p)$-envelope is $\{p\}$, 
and any extremal operation $\vartheta$ is idempotent, i.e., $\vartheta(p,p) = p$.
Thus, by Proposition~\ref{prop:frac_join_prod}, we have
\begin{eqnarray*}
&& f_{p'}(p) + f_{p'}(q) = f(p,p') + f(q,p') \\  
&& \geq f(p \wedge q, p' \wedge p')  
+ \sum_{\vartheta \in {\cal E}({\cal M})} \sum_{\vartheta' \in {\cal E}({\cal M'})} [C(\vartheta) \cap C(\vartheta')] f(\vartheta(p,q), \vartheta'(p',p')) \\
&& = f(p \wedge q, p') + \sum_{\vartheta \in {\cal E}({\cal M})} \sum_{\vartheta' \in {\cal E}({\cal M'})} [C(\vartheta) \cap C(\vartheta')]  f(\vartheta(p,q), p') \\
&& = f_{p'} (p \wedge q) + \sum_{\vartheta \in {\cal E}({\cal M})} [C(\vartheta)] f_{p'}(\vartheta(p,q)). 
\end{eqnarray*}
(4) follows from the fact that for $p,q \in {\cal N}$
the metric interval $I(p,q)$ is the same on ${\cal L}$ and on ${\cal N}$.
\end{proof}

We finally give a useful criterion of the submodularity.
A bounded pair $(p,q)$ in ${\cal L}$ is said to be {\em $2$-bounded}
if $p \vee q$ covers $p$ and $q$ 
(in which case both $p$ and $q$ cover $p \wedge q$).
\begin{Prop}\label{prop:2section}
Suppose that ${\cal L}$ is the product of two modular semilattices ${\cal L}_1 \times {\cal L}_2$.
$f:{\cal L} \to \RR$ is submodular if and only if it satisfies 
\begin{itemize}
\item [{\rm (1)}] the submodularity inequality for every $2$-bounded pair $p,q$, and
\item [{\rm (2)}] the $\wedge$-convexity inequality 
for every pair $(p,q) = ((p_1,p_2), (q_1,q_2))$ 
such that $p_1 = q_1$ and $(p_2,q_2)$ is antipodal in ${\cal L}_2$ or
$p_2 = q_2$ and $(p_1,q_1)$ is antipodal in ${\cal L}_1$.
\end{itemize}
\end{Prop}
The proof is given in Section~\ref{subsub:2section}.
Notice that this criterion does not work when $f$ has infinite values.

\paragraph{Minimizing a sum of submodular functions 
with bounded arity.}

Here we consider the problem of minimizing 
submodular function $f$ on the product of modular semilattices ${\cal L}_1, {\cal L}_2, \ldots, {\cal L}_n$, 
where the input of the problem is 
${\cal L}_1, {\cal L}_2, \ldots, {\cal L}_n$ 
and an evaluating oracle of $f$.
In the case where 
each ${\cal L}_i$ is a lattice of rank $1$, 
this problem is 
the submodular set function minimization in the ordinary sense, 
and can be solved in polynomial time~\cite{GLS, IFF, Schrijver}.
However, we do not know whether this problem in general is 
polynomial time solvable or not.
One notable result in this direction, 
due to Kuivinen~\cite{Kuivinen},
is that 
if each ${\cal L}_i$ is a complemented 
modular lattice of rank $2$ (a {\em diamond lattice}), 
then this problem has a {\em good characterization.}

So we restrict our investigation to the problem of minimizing a sum of 
submodular functions with bounded arity, 
i.e., valued CSP for submodular functions.
See Section~\ref{subsec:VCSP} for notions in valued CSP.
Let ${\cal L} := {\cal L}_1 \times {\cal L}_2 \times \cdots \times {\cal L}_n$.
A constraint $f$ on ${\cal L}$ is called  {\em submodular} if 
$f$ is a submodular function on ${\cal L}_{I_f}$. 
The {\em submodular language} ${\cal S}_{\cal L}$ is 
the set of all submodular constraints on ${\cal L}$.
\begin{Thm}\label{thm:minimized}
Let ${\cal L}$ be the product of 
modular semilattices ${\cal L}_1, {\cal L}_2, \ldots, {\cal L}_n$.
Then {\bf VCSP}$[{\cal S}_{\cal L}]$ can be solved in polynomial time. 
\end{Thm}
Indeed, the submodular language ${\cal S}_{\cal L}$ satisfies 
the Thapper-\v{Z}ivn\'y criterion (Theorem~\ref{thm:TZ}).
Define a fractional operation $\omega$ on ${\cal L}$ by
\begin{equation}\label{eqn:polymorphism_submo}
\omega = \frac{1}{2} \wedge + \frac{1}{2} \sum_{\vartheta \in {\cal E}({\cal L})} [C(\vartheta)] \vartheta.
\end{equation}
By Proposition~\ref{prop:frac_join_prod}, 
an extremal operation is the componentwise 
extension of operations in ${\cal L}_i$, and is separable.
Also, by (\ref{eqn:normal_fan})~(3),
the total sum of coefficients $[C(\vartheta)]$ is equal to $1$.
Therefore $\omega$ is a fractional polymorphism 
for the submodular language ${\cal S}_{\cal L}$.
Obviously $\wedge$ is a semilattice operation.
Hence Theorem~\ref{thm:minimized} follows from Theorem~\ref{thm:TZ}.
\begin{figure} 
\begin{center}
\includegraphics[scale=0.45]{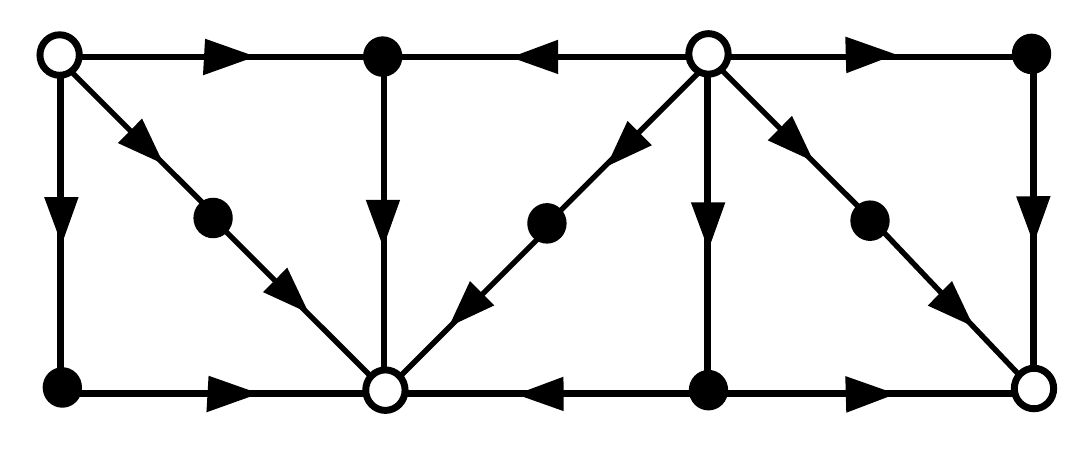}                
\caption{A nonsemilattice orientable modular graph}  
\label{fig:nonsemilattice}
\end{center}
\end{figure}
\begin{Rem}\label{rem:nonsemilattice}
Suppose that $\mGamma$ is the covering graph of a modular semilattice ${\cal L}$.
By Theorem~\ref{thm:d_is_submodular} and Lemma~\ref{lem:properties_submo}, 
$g_j$ and $f_{ij}$ (defined in (\ref{eqn:g_jf_ij})) are submodular constraints on ${\cal L}^n$.
Then {\bf 0-Ext}$[\mGamma]$ is a subclass of {\bf VCSP}$[{\cal S}_{{\cal L}^n}]$.
Therefore, by Theorem~\ref{thm:minimized}, 
{\bf 0-Ext}$[\mGamma]$ can be solved by the basic LP relaxation.
This observation, however, does not get us to the main result (Theorem~\ref{thm:main})
since there is an orientable modular graph that cannot be 
represented as the covering graph of a modular semilattice.
See the graph of Figure~\ref{fig:nonsemilattice},  where there are exactly 
two admissible orientations: one is the reverse of the other, 
and both orientations have two sinks and two sources.
Nevertheless, the basic LP is expected to solve  {\bf 0-Ext}$[\mGamma]$ 
for an arbitrary orientable modular graph $\mGamma$; see Section~\ref{sec:concluding}.
\end{Rem}

\subsection{Proofs}\label{subsec:submo_proof}

\subsubsection{Proof of Lemmas~\ref{lem:unique} and \ref{lem:[C]} 
and Theorem~\ref{thm:submo}}\label{subsub:unique}
Let $(p,q)$ be a pair of elements in modular semilattice ${\cal L}$.
First we prove
a general version of Lemma~\ref{lem:unique}.
\begin{Lem}\label{lem:extreme}
For $s,s' \in {\cal E}(p,q)$, 
if $v(s \wedge p) \geq v(s' \wedge p)$ 
and $v(s \wedge q) \leq v(s' \wedge q)$, then
$s \wedge p \succeq s' \wedge p$ 
and $s \wedge q \preceq s' \wedge q$. 
In particular, $v(s; p,q) = v(s'; p,q)$ implies $s = s'$.
\end{Lem}
\begin{proof}
We can assume that $v(p \wedge q) = 0$ by adding a constant to $v$
(for notational simplicity).
It suffices to consider the case where 
$v(s; p,q) = v(s';p,q)$ or
$v(s; p,q)$ and $v(s';p,q)$ are adjacent extreme points in $\Conv I(p,q)$. 
Let $(a,b) := (s \wedge p, s \wedge q)$ and $(a',b') := (s' \wedge p, s' \wedge q)$.
All pairs $(a,  a')$, $(a,  b \wedge b')$,  
and $(a', b \wedge b')$ from among the
triple $(a, a', b \wedge b')$
are bounded.
By definition of modular semilattices, 
their join
$\eta := (a \vee a') \vee (b \wedge b')$ 
exists and belongs to $I(p,q)$ (see Lemma~\ref{lem:I(p,q)}). 
Similarly the join $\xi := (a \wedge a' ) \vee (b \vee b')$ 
of triple $(a \wedge a', b, b')$
exists and belongs to $I(p,q)$.
Therefore $v(\eta;p,q) =  (v(a \vee a'),  v(b \wedge b'))$ and 
$v(\xi;p,q) = ( v(a \wedge a'),  v(b \vee b') )$.
By modularity equality (\ref{eqn:valuation2}) for $v$, we have
\begin{equation*}
v(s; p,q) + v(s'; p,q) = (v(a) + v(a'), v(b) + v(b'))  = v(\eta; p,q) + v(\xi; p,q).
\end{equation*}
Then both $v(\eta;p,q)$ and $v(\xi;p,q)$ must belong to 
$[v(u;p,q), v(u';p,q)]$ 
since it is an edge (or an extreme point) of $\Conv I(p,q)$.
Necessarily $v(\eta;p,q)  = v(s;p,q)$
and $v(\xi;p,q)  = v(s';p,q)$.
This means that $a \vee a' = a$ and
$b \vee b' = b'$.
Hence the claim follows.
\end{proof}

\begin{Lem}\label{lem:antipodal_u_i-u_j}
For $s,t \in {\cal E}(p,q)$ with $t \wedge p \preceq s \wedge p$ 
(and $t \wedge q \succeq s \wedge q$), the following hold:
\begin{itemize}
\item[{\rm (1)}] $d(p,q) = d(p,s) + d(s,t) + d(t,q)$; in particular $I(s,t) \subseteq I(p,q)$.
\item[{\rm (2)}] $v(u;p,q) = v(u;s,t) + v(s \wedge t; p,q)$ for $u \in I(s,t) \subseteq I(p,q)$.
\item[{\rm (3)}] If $v(s;p,q)$ and $v(t;p,q)$ are adjacent extreme points, 
then $(s,t)$ is antipodal.
\end{itemize}
\end{Lem}
\begin{proof}
(1). By Lemma~\ref{lem:I(p,q)}~(4),  we have
$(s \wedge t) \vee (q \wedge t) = (s \wedge t \wedge p) \vee (s \wedge t \wedge q) \vee (q \wedge t) = (t \wedge p) \vee (t \wedge q) = t$.
This means $t \in I(s,q)$. Hence $d(p,q) = d(p,s) + d(s,q) = d(p,s) + d(s,t) + d(t,q)$.

(2). By $u \wedge s = (s \wedge t) \vee (u \wedge p)$ 
(Lemma~\ref{lem:I(p,q)}~(4) for $I(p,t)$), 
we have
$v[t \wedge p, p \wedge u] = v[s \wedge t, u \wedge s]$, and
\[
v[p \wedge q, p \wedge u] =  v[p \wedge q, t \wedge p] + v[t \wedge p, p \wedge u] 
 =  v[p \wedge q, s \wedge t \wedge p] + v[s \wedge t, u \wedge s].
\]
Similarly (for $I(s,q)$), we have $v[p \wedge q, q \wedge u] = v[p \wedge q, s \wedge t \wedge q] + v[s \wedge t, u \wedge t]$.

(3). If $(s,t)$ is not antipodal, then 
there is $u \in I(s,t)$ such that 
$v(u;s,t)$ goes beyond the line segment between $v(s;s,t)$ and $v(t;s,t)$.
Then, by (2), $v(u;p,q)$ is in the outside of $\Conv I(p,q)$. This is a contradiction.
\end{proof}
Suppose that
${\cal E}({p,q}) = \{p = u_0,u_1,\ldots,u_m = q\}$, 
and  $v(u_i;p,q)$ and $v(u_{i+1};p,q)$ are adjacent extreme points. 
Let $\theta_i$ be
the angle of the line normal 
to the line segment connecting 
$v(u_{i-1};p,q)$ and $v(u_{i}; p,q)$.
By Lemma~\ref{lem:antipodal_u_i-u_j}~(2), $\theta_i$ is equal to
the angle of the line normal 
to the line segment connecting 
$v(u_{i-1};u_{i},u_{i-1})$ and $v(u_{i}; u_{i},u_{i-1})$.
Therefore
\begin{equation}
\frac{\sin \theta_i}{\sin \theta_i + \cos \theta_i}
= \frac{v[u_i \wedge u_{i-1}, u_{i-1}]}{v[u_i \wedge u_{i-1},u_{i-1}] 
+ v[u_i \wedge u_{i-1},u_i]} = \delta_{i-1}.
\end{equation}
Therefore we obtain the formula of Lemma~\ref{lem:[C]}.

Next we prove Theorem~\ref{thm:submo}.
It suffices to prove the if part. 
Let $(p,q)$ be a pair of (incomparable) elements in ${\cal L}$.
Suppose that ${\cal E}({p,q}) = \{p = u_0,u_1,\ldots,u_m = q\}$ is given as above.
Let $p_i := u_i \wedge p$ and $q_i := u_i \wedge q$ for $i=0, 1,2,\ldots,m$.
By Lemma~\ref{lem:extreme}, it holds
$p_{i}  \succeq p_{j}$ and
$q_{i}   \preceq q_{j}$ for $i \leq j$.
\begin{figure} 
\begin{center}
\includegraphics[scale=0.6]{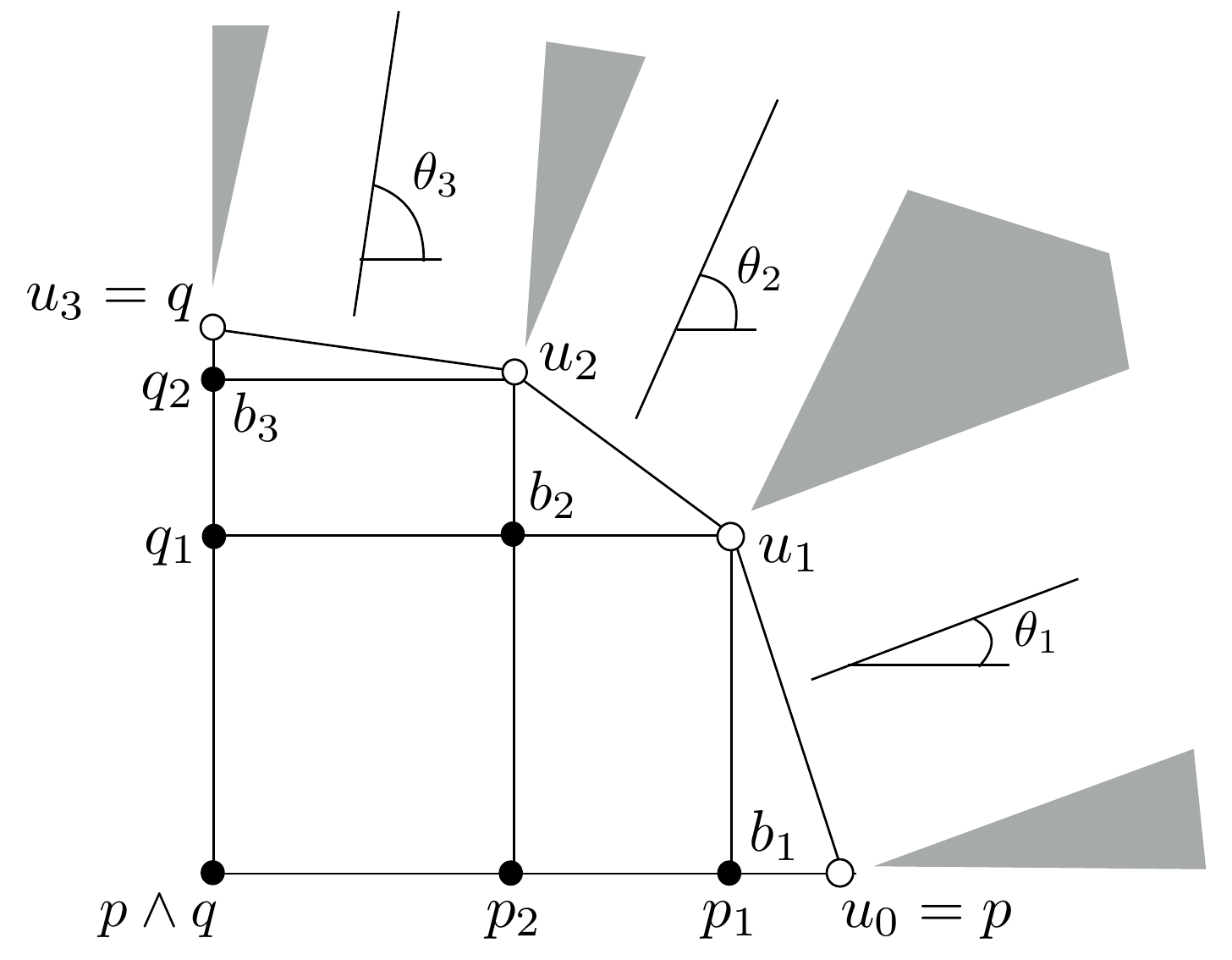}                
\caption{$\Conv I(p,q)$ and ${\cal E}(p,q)$}  
\label{fig:envelope}
\end{center}
\end{figure}
Let $b_i := u_{i-1} \wedge u_i$ for $i=1,2,\ldots,m$; 
see Figure~\ref{fig:envelope}.
Then
we have
\begin{equation}\label{eqn:wedge}
b_i \wedge q_i = u_{i-1} \wedge u_i \wedge u_i \wedge q = q_{i-1}.
\end{equation}
Also, by Lemma~\ref{lem:I(p,q)}~(3) and (4), we have
\begin{eqnarray}\label{eqn:vee}
b_i \vee q_i & = & (u_{i-1} \wedge u_i) \vee q_i  =   (u_{i-1} \wedge u_i \wedge p) \vee (u_{i-1} \wedge u_i \wedge q) \vee q_i  \\
& = &   (p_{i-1} \wedge p_i ) \vee (q_{i-1} \wedge q_{i}) \vee q_i  =  p_i \vee q_i = u_i.\nonumber
\end{eqnarray}

Let $f$ be a function on ${\cal L}$ satisfying 
the conditions (1), (2), and (3) in Theorem~\ref{thm:submo}.
We show that $f$ satisfies the inequality (\ref{eqn:gsubmo}) for $p,q$.
We may assume that $p,q \in \dom f$. 
By condition (1), ${\cal E}(p,q) \subseteq \dom f$. 
Namely $u_i \in \dom f$. By (3), we have $b_i =  u_i \wedge u_{i-1} \in \dom f$.
By (2), we have 
$b_i \wedge b_{i-1} = u_{i} \wedge u_{i-1} \wedge u_{i-2} = u_i \wedge u_{i-2}  \in \dom f$. 
Consequently all $p_i, q_i, b_i$ belong to  $\dom f$. 
By (\ref{eqn:wedge}), (\ref{eqn:vee}), and the condition (2) (submodularity), we have
\begin{equation}\label{eqn:adding}
f(b_{i}) + f(q_{i}) \geq f(q_{i-1}) + f(u_{i}) \quad (i = 1,2,\ldots,m).
\end{equation}
By adding (\ref{eqn:adding}) for $i=1,2,\ldots, m$ 
and $f(p) = f(u_0)$ (recall $(p,q) = (u_0,q_m)$), we obtain
\begin{equation}\label{eqn:ba}
f(p) + f(b_1) + f(b_2) + \cdots + f(b_{m}) + f(q) 
\geq f(p \wedge q) + f(u_0) + f(u_1) + \cdots + f(u_{m}).
\end{equation}
By Lemma~\ref{lem:antipodal_u_i-u_j}~(3), 
pair $(u_{i},u_{i+1})$ is antipodal. 
By condition (3), $f$ satisfies the $\wedge$-convexity inequality (\ref{eqn:^-convexity}) for $(u_{i},u_{i+1})$, 
which is rewritten as 
\begin{equation}
f(u_{i}) \geq f(b_{i+1}) + 
\delta_i f(u_{i}) 
- \delta_i f(u_{i+1}).
\end{equation}
Substituting (\theequation) to (\ref{eqn:ba}) for $i=0,1,2,\ldots,m-1$, we obtain
\[
f(p) + f(q)  \geq  f(p \wedge q) + \sum_{i=0}^m (\delta_i - \delta_{i-1})f(u_i) 
 =  f(p \wedge q) + \sum_{i=0}^m [C(u_i;p,q)] f(u_i).
\]

\subsubsection{Proof of Proposition~\ref{prop:frac_join}}\label{subsub:frac_join}
We start with preliminary arguments.
Suppose that ${\cal L}$  
is the product ${\cal L}_1 \times {\cal L}_2$ of 
modular semilattices ${\cal L}_1$ and ${\cal L}_2$.
The valuations of ${\cal L}_1$ and ${\cal L}_2$ are given as (\ref{eqn:valuation_summand}). 
They are also denoted by $v$.  
Let $(p,q) = ((p_1,p_2), (q_1,q_2))$ be a pair of elements in ${\cal L}$.
By (\ref{eqn:valuation_summand1}), we have
\begin{equation}
v(u; p,q) = v(u_1; p_1,q_1) +  v(u_2; p_2,q_2) \quad (u \in I(p,q)).
\end{equation}
By this equation together with $I(p,q) = I(p_1,q_1) \times I(p_2,q_2)$,  we have
\begin{equation}\label{eqn:Minkowski}
\Conv I(p,q) = \Conv I(p_1,q_1) + \Conv I(p_2,q_2),
\end{equation}
where the sum means the Minkowski sum.

In general, if a polytope $P$ is the Minkowski sum 
of two polytopes $Q$ and $Q'$, 
then every extreme point of $P$ is uniquely represented as 
the sum of extreme points of $Q$ and of $Q'$.
By this fact and the injectivity of $v(\cdot;p,q)$ (Lemma~\ref{lem:unique}) on ${\cal E}(p,q)$, 
we obtain:
\begin{myitem}
For $u = (u_1,u_2) \in {\cal E}(p,q)$, 
there uniquely exist maximal extreme points $x_i \in \Conv I(p_i, q_i)$ for $i=1,2$ 
such that $x_i = v(u_i; p_i,q_i)$ for $i=1,2$, and
$v(u; p,q) = x_1 + x_2$. In particular, $u_i$ belongs to ${\cal E}(p_i,q_i)$ for $i=1,2$. 
\label{eqn:x1+x2}
\end{myitem}\noindent
Moreover, $v(u;p,q)$ 
maximizes $\langle c, x \rangle$ over $x \in \Conv I(p,q)$
if and only if $v(u_i;p_i,q_i)$ maximizes
 $\langle c, x \rangle$ over $x \in \Conv I(p_i,q_i)$ for $i=1,2$.
Therefore we have
\begin{equation}\label{eqn:C1^C2}
C(u; p,q) = C(u_1;p_1,q_1) \cap C(u_2; p_2,q_2).
\end{equation}

We are ready to prove Proposition~\ref{prop:frac_join}.
Regard ${\cal O}({\cal L})$ as 
a modular semilattice ${\cal L}^{{\cal L} \times {\cal L}}$.
Then $I({\cal L})$ 
is the product of $I (L(p,q), R(p,q)) = I(p,q)$ over all $(p,q) \in {\cal L} \times {\cal L}$.
By (\ref{eqn:Minkowski}) we have
\[
\Conv I({\cal L}) = \sum_{(p,q) \in {\cal L} \times {\cal L}} \Conv I(p,q).
\] 
By (\ref{eqn:x1+x2}), for every extremal operation $\vartheta$, 
each $\vartheta(p,q)$ belongs to ${\cal E}(p,q)$. 
Also, by (\ref{eqn:C1^C2}), we have
\begin{equation}
C(\vartheta)  =  \bigcap_{(p,q) \in {\cal L} \times {\cal L}} C(\vartheta(p,q) ; p,q) \quad 
(\vartheta \in {\cal E}({\cal L})).
\end{equation}
From (\ref{eqn:normal_fan}), 
$C(\vartheta) \subseteq C(u; p,q)$ if and only 
if $\vartheta(p,q) = u$, and
\begin{equation}
C(u;p,q)   =  \bigcup_{\vartheta \in {\cal E}({\cal L}): \vartheta(p,q) = u} 
C(\vartheta) \quad (p,q \in {\cal L}, u \in {\cal E}(p,q)),
\end{equation}
where any two of the cones in the union have no common interior points.
Therefore 
\[
\sum_{\vartheta \in {\cal E}({\cal L})} [C(\vartheta)] \vartheta(p,q) = 
\sum_{u \in {\cal E}(p,q)} \left( \sum_{\vartheta \in {\cal E}({\cal L}): 
\vartheta(p,q) = u} [C(\vartheta)]  \right) u 
 =  \sum_{u \in {\cal E}(p,q)} [C(u; p,q)] u.
\]
This proves Proposition~\ref{prop:frac_join}.

\subsubsection{Proof of Proposition~\ref{prop:frac_join_prod}}\label{subsub:frac_join_prod}

It suffices to consider the case where ${\cal L} = {\cal L}_1 \times {\cal L}_2$.
If $\vartheta \in I({\cal L})$ is the componentwise extension of 
$\vartheta_i \in I({\cal L}_i)$ for $i=1,2$, then 
\begin{eqnarray} \label{eqn:v(theta;L,R)}
v(\vartheta;L,R) &= & \sum_{(p,q) \in {\cal L} \times {\cal L}} v(\vartheta(p,q); p,q)  \\ 
& = & \sum_{(p,q) \in {\cal L} \times {\cal L}} v(\vartheta_1(p_1,q_1); p_1,q_1 ) +  v(\vartheta_2 (p_2,q_2); p_2,q_2) \nonumber \\
& = & |{\cal L}_2|^2 v (\vartheta_1; L,R) + |{\cal L}_1|^2 v(\vartheta_2; L,R). \nonumber
\end{eqnarray}
Therefore it holds 
\begin{equation}\label{eqn:ConvL1+ConvL2}
\Conv I ({\cal L}) \supseteq 
| {\cal L}_2 |^2 \Conv I ({\cal L}_1) + | {\cal L}_1 |^2 \Conv I ({\cal L}_2).
 \end{equation}
We are going to show that every extremal operation in ${\cal L}$ is 
the componentwise extension of operations in ${\cal L}_i$ for $i=1,2$, and 
the equality holds in (\ref{eqn:ConvL1+ConvL2}).
Take an extremal operation $\vartheta$ on ${\cal L}$.
There are $\vartheta_i :{\cal L} \times {\cal L} \to {\cal L}_i$ for $i=1,2$ such that
$\vartheta(p,q) = (\vartheta_1(p,q), \vartheta_2(p,q))$ for $(p,q) \in {\cal L} \times {\cal L}$.
By Proposition~\ref{prop:frac_join}, $\vartheta(p,q)$ belongs to ${\cal E}(p,q)$.
By (\ref{eqn:x1+x2}), 
$\vartheta_1(p,q)$ and $\vartheta_2(p,q)$ belong to ${\cal E}(p_1,q_1)$ and ${\cal E}(p_2,q_2)$, respectively.
For $(p,q),(p',q') \in {\cal L} \times {\cal L}$ with $(p_1,q_1) = (p'_1,q'_1)$,
suppose (indirectly) that $\vartheta_1(p,q) \neq \vartheta_1(p',q')$.
Let $\vartheta'$ be the operation in $I({\cal L})$ obtained from $\vartheta$ 
by replacing $\vartheta_1(p',q')$ with $\vartheta_1(p,q)$, 
and let $\vartheta''$ be the operation in $I({\cal L})$ obtained from $\vartheta$  
by replacing $\vartheta_1(p,q)$ with $\vartheta_1(p',q')$. Then we have
\begin{eqnarray*}
&& v(\vartheta(p,q); p,q) + v(\vartheta(p',q'); p',q') = \sum_{i=1,2} v(\vartheta_i(p,q); p_i,q_i) +  v(\vartheta_i(p',q'); p'_i,q'_i) \\
&& = \frac{1}{2} \left\{ (v(\vartheta'(p,q); p,q) + v(\vartheta'(p',q'); p',q')) + (v(\vartheta''(p,q); p,q) + v(\vartheta''(p',q'); p',q')) \right\}.
\end{eqnarray*}
By Lemma~\ref{lem:unique}, $v(\vartheta_1(p,q);p_1,q_1)$ and $v(\vartheta_1(p',q');p_1,q_1)$
are distinct, and
consequently $v(\vartheta;L,R)$ is the midpoint of segment 
between distinct points $v(\vartheta';L,R)$ and $v(\vartheta'';L,R)$, 
contradicting the fact that $\vartheta$ is extremal.
Therefore $\vartheta_1(p,q) = \vartheta_1(p',q')$ must hold.
This means that $\vartheta_1(p,q)$ does not depend on 
the second component of each of $p,q$. 
So we can regard $\vartheta_1 \in I({\cal L}_1)$.
Similarly $\vartheta_2 \in I({\cal L}_2)$, and
$\vartheta$ is equal to 
the componentwise extension of $\vartheta_1$ and $\vartheta_2$.
Hence the equality holds in (\ref{eqn:ConvL1+ConvL2}), 
both $\vartheta_1$ and $\vartheta_2$ must be extremal, and
$
[C(\vartheta)] = [C(\vartheta_1) \cap C(\vartheta_2)]$.
Thus we have
\[
\sum_{\vartheta \in {\cal E}({\cal L})} [C(\vartheta)] \vartheta
= \sum_{\vartheta_1 \in {\cal E}({\cal L}_1), \vartheta_2 \in {\cal E}({\cal L}_2)} [C(\vartheta_1) \cap C(\vartheta_2) ] (\vartheta_1,\vartheta_2).
\]

Suppose that ${\cal L}_1 = {\cal L}_2$ holds. 
Suppose that $\vartheta_1$ and $\vartheta_2$ are different.
We see from (\ref{eqn:v(theta;L,R)}) that
$v( (\vartheta_1,\vartheta_2); L,R )$ is the midpoint 
of the segment between distinct points $v((\vartheta_1,\vartheta_1);L,R)$ and 
$v((\vartheta_2,\vartheta_2);L,R)$.
Hence $(\vartheta_1,\vartheta_2)$ is never extremal.
The proof of Proposition~\ref{prop:frac_join_prod} is now complete.

\subsubsection{Proof of Proposition~\ref{prop:2section}}\label{subsub:2section}
We use the characterization of Theorem~\ref{thm:submo}.
So the only if part is obvious. 
We prove the if part.
We first show that the submodularity inequality for an arbitrary bounded pair
is implied by submodularity inequalities for $2$-bounded pairs.
For a bounded pair $(p,q)$, take maximal chains 
$(p \wedge q = p_0,p_1,\ldots,p_k = p)$ and $(p \wedge q = q_0,q_1,\ldots,q_l = q)$.
Let $a_{i,j} := p_{i} \vee q_{j}$. Then
$f(p) + f(q) - f(p \wedge q) - f(p \vee q)
= \sum_{i,j} (f(a_{i+1,j}) + f(a_{i,j+1}) - f(a_{i+1,j+1}) - f(a_{i,j})) \geq 0$.
Here we use the fact seen from modularity
that $(a_{i+1,j},a_{i,j+1})$ is a $2$-bounded pair 
with $a_{i+1,j+1} = a_{i+1,j} \vee a_{i,j+1}$ 
and $a_{i,j} = a_{i+1,j} \wedge a_{i,j+1}$.

Next we show the $\wedge$-convexity inequality.
Take an (incomparable) 
antipodal pair $(p,q) = ((p_1,p_2), (q_1,q_2))$ in ${\cal L}$.
Then ${\cal E}({p,q}) = \{p, q\}$.
Then $\Conv I(p,q)$ is a triangle. 
By (\ref{eqn:Minkowski}), it holds
$\Conv I(p, q) = \Conv I(p_1,q_1) + \Conv I(p_2,q_2)$.
Therefore both $\Conv I(p_1,q_1)$ and $\Conv I(p_2,q_2)$
are triangles congruent to a dilation of $\Conv I(p,q)$.
Hence $(p_i,q_i)$ is antipodal in ${\cal L}_i$, and 
$C(p_i; p_i, q_i) = C(p; p,q)$ and $C(q_i; p_i, q_i) = C(q; p,q)$ for $i=1,2$.
In particular, both $((q_1,q_2), (q_1,p_2))$ and $((q_1,p_2), (p_1,p_2))$
are antipodal pairs in ${\cal L} = {\cal L}_1 \times {\cal L}_2$.
Letting $C_p := C(p; p,q) = C(p_i; p_i, q_i)$ and $C_q := C(q; p,q) = C(q_i; p_i, q_i)$, we have
\begin{eqnarray*}
&& (1 - [C_q]) f(q_1,q_2) + (1 - [C_p]) f(q_1,p_2) \geq f(q_1, p_2 \wedge q_2), \\
&& (1 - [C_q]) f(q_1,p_2) + (1 - [C_p]) f(p_1,p_2) \geq f(p_1 \wedge q_1, p_2). 
\end{eqnarray*}
Also, 
by submodularity inequality (shown above), we have
\begin{equation*}
f(q_1, p_2 \wedge q_2) + f(p_1 \wedge q_1, p_2) \geq f(p_1 \wedge q_1,p_2 \wedge q_2) + f(q_1,p_2).
\end{equation*}
From the three inequalities, we obtain
\begin{eqnarray*}
&& (1 - [C_q]) f(q_1,q_2) + (2 - [C_p] - [C_q]) f(q_1,p_2) + (1 - [C_p])f(p_1,p_2)  \\
&& \quad \quad \geq f(p_1 \wedge q_1,p_2 \wedge q_2) + f(q_1,p_2).
\end{eqnarray*}
By using $[C_p] + [C_q] =1$, 
we obtain the $\wedge$-convexity inequality for $(p,q)$.

\section{L-convex function on modular complex}\label{sec:L-convex}
A {\em modular complex} $\pmb{\mGamma}$
is a triple $(\mGamma,o,h)$ of an orientable modular graph $\mGamma$, its admissible orientation $o$, 
and its positive orbit-invariant function $h$.
The goal of this section is to introduce 
a class of discrete convex functions, called {\em L-convex functions}, on $\pmb{\mGamma}$, 
and show that L-convex functions have several nice properties 
for optimization, analogous to L$^\natural$-convex functions in discrete convex analysis.

The main properties of our L-convex functions are:
\begin{itemize}
\item The distance function $d_{{\mGamma},h}$ is an L-convex function 
on $\pmb{\mGamma} \times \pmb{\mGamma}$ (Theorem~\ref{thm:d_is_L-convex}).
\item In the minimization of an L-convex function,
checking optimality and finding a descent direction
can be done by submodular function minimization 
on modular semilattices (Theorem~\ref{thm:L-optimality}).
\end{itemize}
In Section~\ref{subsec:complex}, 
we explore several structural properties of modular complexes.
In particular, a modular complex can be regarded as 
a structure obtained by gluing modular semilattices (Theorem~\ref{thm:lattice}),  
and admits a subdivision operation (Theorem~\ref{thm:2subdivision}).
This operation produces a fine modular complex $\pmb{\mGamma}^*$ into which 
the original modular complex $\pmb{\mGamma}$ is embedded, 
and also enables us to define the {\em neighborhood semilattice} ${\cal L}_p^*$ 
around each vertex $p$, which is also a modular semilattice.
Based on this investigation as well as the idea mentioned in the introduction, 
in Section~\ref{subsec:L-convex}, 
we introduce L-convex functions on $\pmb{\mGamma}$, and present their properties.
Again less obvious theorems will be proved in 
Section~\ref{subsec:L-convex_proof}.
A further geometric study on orientable modular graphs is given in \cite{CCHO}.

\subsection{Modular complex}\label{subsec:complex}
Let $\pmb{\mGamma} = (\mGamma,o,h)$ be a modular complex, 
where a modular complex is denoted by the bold style $\pmb{\mGamma}$ 
of the underlying graph $\mGamma$.

\paragraph{Boolean pairs.}
%
Consider a cube subgraph $B$ of ${\mGamma}$, 
and consider the digraph $\vec{B}$ of $B$ 
oriented by $o$.
One can easily see from the definition of 
an admissible orientation
that $\vec{B}$ 
is isomorphic to the Hasse diagram of a Boolean lattice.
Hence $\vec{B}$ determines
the maximum element and the minimum element of 
the corresponding Boolean lattice.
A pair $(p,q)$ of vertices is 
called {\em $o$-Boolean}, or simply, {\em Boolean} 
if $p$ and $q$ are the minimum element and  the maximum element, 
respectively,
of the Boolean lattice associated 
with some cube subgraph of ${\mGamma}$.
By convention,
$(p,p)$ is defined to be Boolean.
The set of Boolean pairs is denoted by ${\cal B}(\pmb{\mGamma})$.
In Figure~\ref{fig:modular_complex} in the introduction, for example,
$(p',p)$, $(v, q)$, $(p', v)$ are Boolean, and $(q, p')$ is not Boolean.

Recall that any admissible orientation is acyclic 
(Lemma~\ref{lem:acyclic}).
Let $\preceq (=\ \preceq_o)$ be 
the transitive closure of relation
$\swarrow (=\ \swarrow_o)$ on $V_{\mGamma}$. 
Then $V_{\mGamma}$ is regarded as
a partially ordered set according to this relation; so $p \swarrow q$ implies $p \prec q$.
For any Boolean pair $(p,q)$, 
necessarily $p \preceq q$ holds.
\begin{Prop}\label{prop:modularlattice}
Let $\pmb{\mGamma}$ be a modular complex. 
For $p,q \in V_{\mGamma}$ with $p \preceq q$,  we have the following:
\begin{itemize}
\item[{\rm (1)}] $[p,q]$ is a modular lattice, 
is convex in $\mGamma$, and is equal to $I(p,q)$.
\item[{\rm (2)}] $(p,q)$ is Boolean if and only if $[p,q]$ is a complemented modular lattice.
\end{itemize}
In particular we can check whether a given pair is Boolean in time polynomial in $|V_{\mGamma}|$.
\end{Prop}
We prove this proposition in Section~\ref{subsub:proof_of_lattice}.
We define the relation $\sqsubseteq ( =\ \sqsubseteq_o)$ as: 
$p \sqsubseteq q$ if $(p,q)$ is a Boolean pair.
This relation $\sqsubseteq$ coarsens $\preceq$, 
and is not transitive in general. 
Since a complemented modular lattice is relatively complemented (Theorem~\ref{thm:complemented}), 
we have:
\begin{myitem}
If $p \sqsubseteq q$ and $p \preceq u \preceq v \preceq q$,
then $p \sqsubseteq u \sqsubseteq v \sqsubseteq q$. \label{eqn:v}
\end{myitem}\noindent
For a vertex $p$, 
define subsets ${\cal L}^+_{p}(\pmb{\mGamma})$ and ${\cal L}^-_{p}(\pmb{\mGamma})$ 
of vertices by
\begin{equation}
{\cal L}^+_{p}(\pmb{\mGamma}) :=  \{ q \in V_{\mGamma} \mid p \sqsubseteq q \},  \quad
{\cal L}^-_{p}(\pmb{\mGamma}) :=  \{ q \in V_{\mGamma} \mid q \sqsubseteq p \}.
\end{equation}
In the sequel, ${\cal L}^+_{p}(\pmb{\mGamma})$ and ${\cal L}^-_{p}(\pmb{\mGamma})$ 
are often denoted by ${\cal L}^+_p$ and ${\cal L}^-_p$, respectively.
Regard ${\cal L}^+_p$ as a poset by the partial order $\preceq$, 
and regard ${\cal L}^-_p$ as a poset by the reverse of $\preceq$.
\begin{Thm}\label{thm:lattice} 
Let $\pmb{\mGamma}$ be a modular complex.
For every vertex $p$, both ${\cal L}^+_p$ and ${\cal L}^-_p$ are complemented modular semilattices, 
and convex in $\mGamma$.
\end{Thm}
Theorem~\ref{thm:lattice} will be proved in Section~\ref{subsub:proof_of_lattice}.
Therefore
$\pmb{\mGamma}$ is a structure obtained by gluing modular lattices and semilattices.
Moreover $\pmb{\mGamma}$ gives rise to 
a simplicial complex ${\mit\Delta}(\pmb{\mGamma})$ as follows.
For each Boolean pair $(p,q)$ and 
each ascending path 
$(p= p_0, p_1, p_2,\ldots, p_k = q)$ from $p$ to $q$, 
fill a $k$-dimensional simplex  
as in Figure~\ref{fig:modular_complex}.
Then we obtain a simplicial complex ${\mit\Delta}(\pmb{\mGamma})$, and 
we can define an analogue of Lov\'asz extension for 
any function on $V_{\mGamma}$.
We however do not use 
this complex ${\mit\Delta}(\pmb{\mGamma})$ 
in the sequel, 
although our argument is based on
this geometric view.
Instead of dealing with ${\mit\Delta}(\pmb{\mGamma})$,
we use a graph-theoretic operation, 
the $2$-subdivision $\pmb{\mGamma}^*$ of $\pmb{\mGamma}$, 
which comes from the barycentric subdivision 
of ${\mit\Delta}(\pmb{\mGamma})$.

\paragraph{2-subdivision and neighborhood semilattices.}
The $2$-subdivision $\pmb{\mGamma}^*$ of $\pmb{\mGamma}$
is constructed as follows.
A Boolean pair $(p,q) \in {\cal B}(\pmb{\mGamma})$ is denoted by $q / p$.
The {\em $2$-subdivision} ${\mGamma}^*$ of ${\mGamma}$ 
is a simple undirected graph on 
the set ${\cal B}(\pmb{\mGamma})$ of all Boolean pairs
with edges given as:
$q/p$ and $q'/p'$ are 
adjacent if and only if 
$p = p'$ and $qq' \in E_{\mit\Gamma}$ 
or $q = q'$ and $pp' \in E_{\mit\Gamma}$.
The orientation $o^*$ for $\mGamma^*$ is given as:
$q/p \swarrow_{o^*} q'/p'$ if $p = p'$ and $q \swarrow_{o} q'$ 
or if $q = q'$ and $p' \swarrow_{o} p$.
See Figure~\ref{fig:orientation}.
\begin{figure} 
\begin{center}
\includegraphics[scale=0.8]{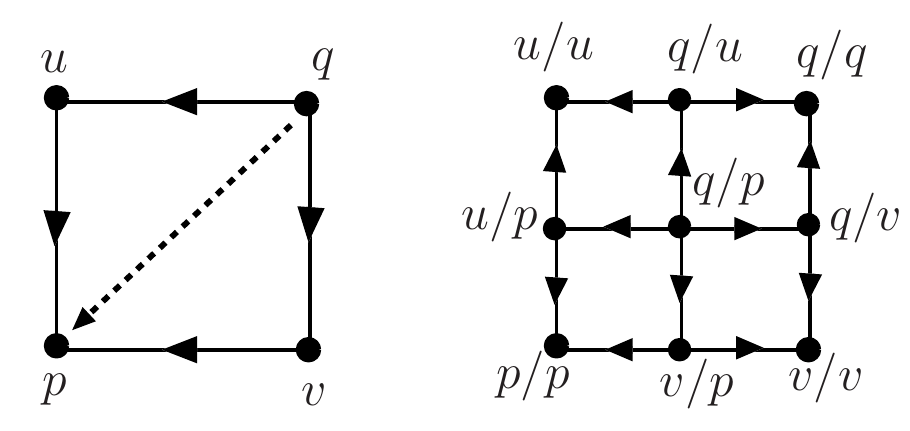}                
\caption{Orientations $o$ and $o^*$}  
\label{fig:orientation}
\end{center}
\end{figure}
In fact, $\mGamma^*$ does not depend on the choice of an admissible orientation; see \cite{CCHO}.

An edge joining $q/p$ and $q'/p$ (resp. $q/p$ and $q/p'$) 
is denoted by $qq'/p$ (resp. $q/pp'$).
A function $h^*$ on $E_{\mGamma^*}$ is defined as
$h^*(qq'/p) := h(qq')/2$ and $h^*(q/pp') := h(pp')/2$.
Let $\pmb{\mGamma}^* := (\mGamma^*, o^*, h^*)$, which is called the {\em $2$-subdivision} of 
$\pmb{\mGamma}$.
\begin{Thm}\label{thm:2subdivision}
For a modular complex $\pmb{\mit\Gamma}$, 
the $2$-subdivision ${\pmb{\mit\Gamma}}^*$ is also a modular complex.
\end{Thm}
This theorem will be proved in Section~\ref{subsub:proof_2subdivision}.
Figure~\ref{fig:neighborhood1} illustrates 
the $2$-subdivision of $\pmb{\mGamma}$ in 
Figure~\ref{fig:modular_complex}.
\begin{figure} 
\begin{center}
\includegraphics[scale=0.85]{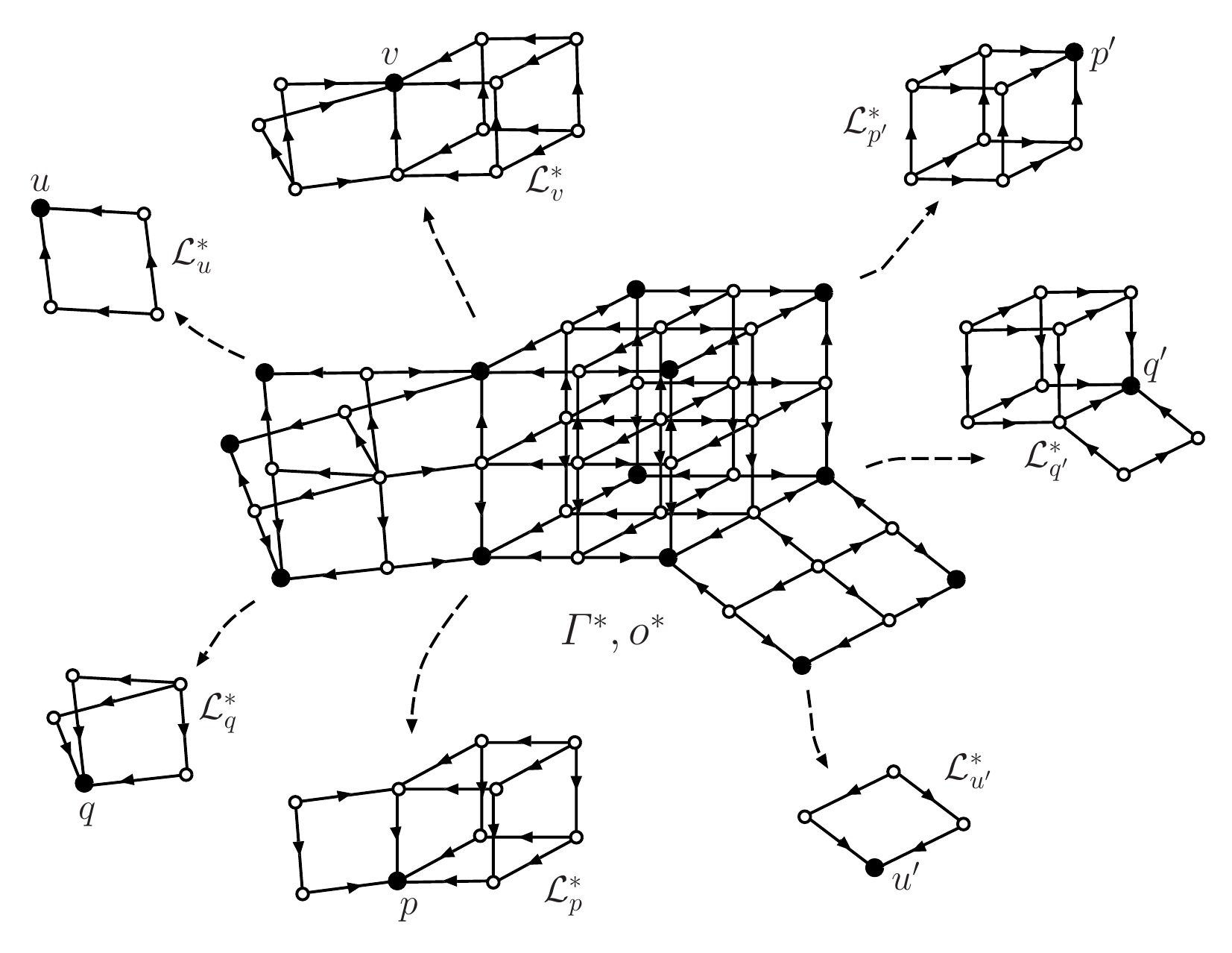}     
\caption{Construction of $\pmb{\mGamma}^*$ and neighborhood semilattices}  
\label{fig:neighborhood1}         
\end{center}
\end{figure}
By embedding $p \mapsto p/p$, 
we can regard $V_{\mGamma} \subseteq V_{\mGamma^*}$.
The admissible orientation $o^*$ is oriented so 
that the vertices in $V_{\mGamma}$ are all sinks.
The partial order $\preceq_{o^*}$ on $V_{\mGamma^*}$ induced by $o^*$ 
is denoted by $\preceq_{*}$, and $\sqsubseteq_{o^*}$ is also denoted by $\sqsubseteq_{*}$.
In fact, one can show that two relations  $\preceq_{*}$ and $\sqsubseteq_{*}$ are the same.
Here we only note the following obvious relation:
\begin{equation}\label{eqn:p'pqq'}
q/p \preceq_* q'/p'   \Longleftrightarrow p' \preceq p \preceq q \preceq q' \quad (q/p,q'/p' \in {\cal B}(\pmb{\mGamma})).
\end{equation}

For each vertex $p \in V_{\mGamma}$, 
define the {\em neighborhood semilattice} 
${\cal L}^*_p := {\cal L}_{p/p}^+(\pmb{\mGamma}^*)$.
By Theorems~\ref{thm:lattice} and \ref{thm:2subdivision}, we obtain:
\begin{Prop}
Let $\pmb{\mGamma}$ be a modular complex. 
For each vertex $p$,
the neighborhood semilattice ${\cal L}^*_p$ is a complemented modular semilattice 
with the minimum element $p$.
\end{Prop}
See Figure~\ref{fig:neighborhood1}.
Neighborhood semilattice ${\cal L}^*_p$ 
has more information about the local property of $p$ 
than ${\cal L}^+_p$ and ${\cal L}^-_p$ have.

\paragraph{Valuation of local semilattices.}
A positive orbit-invariant 
function $h$ naturally gives valuation $v_p$ on ${\cal L}_p^s$ for $s \in \{-,+\}$, and 
valuation $v_p^*$ on ${\cal L}_p^*$ by 
\begin{eqnarray}\label{eqn:valuations}
v_p(q) & :=& d_{\mGamma,h}(q,p) \quad (q \in {\cal L}_p^{s}, s \in \{-,+\}), \\
v^*_p(v/u) &:=& d_{\mGamma^*,h^*}(v/u,p/p) \quad (v/u \in {\cal L}_p^*). \nonumber
\end{eqnarray}
See Lemma~\ref{lem:valuation}.
In the sequel, semilattices ${\cal L}_p^+$, ${\cal L}_p^-$, 
and ${\cal L}_p^*$ are supposed to be endowed with these valuations.

\paragraph{Embedding of $\pmb{\mGamma}$ into $\pmb{\mGamma}^*$.}
The distances on $\mGamma$ and $\mGamma^*$ are related as follows.
\begin{Prop}\label{prop:1/2} 
Let $\pmb{\mGamma} = (\mGamma,o,h)$ be a modular complex, and $\pmb{\mGamma}^*$ the $2$-subdivision of $\pmb{\mGamma}$. Then we have
\begin{equation}\label{eqn:isometric_embedding}
d_{\mGamma^*,h^*}(q/p,q'/p') = \frac{d_{\mGamma,h}(p,p') + d_{\mGamma,h}(q,q')}{2} 
\quad (q/p, q'/p' \in {\cal B}(\pmb{\mGamma}) = V_{\mGamma^*}).
\end{equation}
In particular, $(V_\mGamma, d_{\mGamma,h})$ is isometrically embedded into 
$(V_{\mGamma^*}, d_{\mGamma^*, h^*})$
by $p \mapsto p/p$. 
\end{Prop}
This proposition will be proved in Section~\ref{subsub:proof_2subdivision}.

\paragraph{Product of modular complexes.}
Suppose that 
we are given two modular complexes 
$\pmb{\mGamma} = (\mGamma,o,h)$ and $\pmb{\mGamma}' =(\mGamma',o',h')$.
Then the Cartesian product $\mGamma \times \mGamma'$ 
is also modular.
Furthermore, define 
the orientation $o \times o'$ of $\mGamma \times \mGamma'$
as: 
$(p,p') \swarrow_{o \times o'} (p,q')$ if 
$p' \swarrow_{o'} q'$
and $(p,p') \swarrow_{o \times o'} (q,p')$ 
if $p \swarrow_{o} q$.
Then $o \times o'$ is an admissible orientation.
Similarly define $h \times h'$ 
by $(h \times h')((p,p')(q,p')) := h(pq)$ 
and $(h \times h')((p,p')(p,q')) := h'(p'q')$, 
which is orbit-invariant in $\mGamma \times \mGamma'$.
Thus we obtain a new modular complex 
$\pmb{\mGamma} \times \pmb{\mGamma}' := (\mGamma \times \mGamma', o \times o', h \times h')$, 
which is called the {\em product} 
of $\pmb{\mGamma}$ and $\pmb{\mGamma}'$. 
\begin{Lem}
$(p,p') \sqsubseteq_{o \times o'} (q,q')$ 
if and only if $p \sqsubseteq_{o} q$ and $p' \sqsubseteq_{o'} q'$.
\end{Lem}
\begin{proof}
Since $[(p,p'),(q,q')] \simeq [p,q] \times [p',q']$,
$[(p,p'),(q,q')]$ is complemented modular if and only if 
both $[p,q]$ and $[p',q']$ are complemented modular.
Thus, by Proposition~\ref{prop:modularlattice}, we have the claim.  
\end{proof}
In particular the correspondence 
${\cal B}(\pmb{\mGamma} \times \pmb{\mGamma}') \ni 
(q,q')/(p,p') \mapsto (q/p,q'/p') \in 
{\cal B}(\pmb{\mGamma}) \times {\cal B}(\pmb{\mGamma}')$
is bijective, and we can regard 
\[
{\cal B}(\pmb{\mGamma} \times \pmb{\mGamma}') 
= {\cal B}(\pmb{\mGamma}) \times {\cal B}(\pmb{\mGamma}').
\]
Under this correspondence, the product operation 
and the 2-subdivision operation commute in the following sense.
\begin{Lem}\label{lem:product}
\begin{itemize}
\item[{\rm (1)}] ${\cal L}^s_{(p,p')}(\pmb{\mGamma} \times \pmb{\mGamma}') 
= {\cal L}^s_p(\pmb{\mGamma}) \times {\cal L}^s_{p'}( \pmb{\mGamma}')$ 
for $s \in \{-,+\}$.
\item[{\rm (2)}] $(\pmb{\mGamma} \times \pmb{\mGamma}')^* = \pmb{\mGamma}^* \times {\pmb{\mGamma}'}^*$.
\item[{\rm (3)}] ${\cal L}^*_{(p,p')}(\pmb{\mGamma} \times \pmb{\mGamma}') 
= {\cal L}^*_p(\pmb{\mGamma}) \times {\cal L}^*_{p'}( \pmb{\mGamma}')$.
\end{itemize}
\end{Lem}

\begin{proof}
(1) follows from the previous lemma.
(2) follows from the fact that 
$(q,q')/(p,p')$ and $(v,v')/(u,u')$ have an edge in 
$(\mGamma \times \mGamma')^{*}$ if and only if
$d_{\mGamma}(q,v) + d_{\mGamma'}(q',v') + d_{\mGamma}(p,u) + d_{\mGamma'}(p',u') = 1$, 
which is equivalent to the condition that
$(q/p, q'/p')$ and $(v/u, v'/u')$ have an edge in 
$\mGamma^{*} \times \mGamma'^{*}$.
(3) follows from (2).
\end{proof}

\subsection{L-convex function on modular complex}\label{subsec:L-convex}
We are ready to introduce 
the concept of an L-convex function on 
a modular complex $\pmb{\mGamma} =  (\mGamma,o,h)$.
Consider the $2$-subdivision $\pmb{\mGamma}^{*}$ of $\pmb{\mGamma}$.
For a function $g: V_{\mGamma} \to \RR$, 
define $\overline{g}: V_{\mGamma^*} \to \RR$ by
\begin{equation}\label{eqn:Lovasz}
\overline{g}(q/p) := \frac{g(p) + g(q)}{2}  \quad 
(q/p \in {\cal B}(\pmb{\mGamma}) = V_{\mGamma^*}).
\end{equation}
This is the restriction of 
the Lov\'asz extension of $g$; 
see the introduction for the Lov\'asz extension.
By restricting $\overline{g}$ to 
neighborhood semilattices,
we obtain functions on (complemented) modular semilattices 
${\cal L}^*_p$ for each vertex $p$.

An {\em L-convex function} on $\pmb{\mGamma}$ 
is a function $g: V_{\mGamma} \to \RR$ 
such that for each vertex $p \in V_{\mGamma}$
the restriction of $\overline{g}$ to ${\cal L}^*_p$ is submodular on ${\cal L}^*_{p}$ 
(with respect to the valuation $v^*_p$).
Corresponding to Theorem~\ref{thm:d_is_submodular}, 
the distance function $d_{\mGamma,h}$ is L-convex on $\pmb{\mGamma} \times \pmb{\mGamma}$,
which is one of the most important properties for our application to 0-extension problem.
\begin{Thm}\label{thm:d_is_L-convex}
For a modular complex $\pmb{\mGamma} = (\mGamma,o,h)$, 
the distance function $d_{\mGamma,h}$ is an L-convex function 
on $\pmb{\mGamma} \times \pmb{\mGamma}$.
\end{Thm}
The proof of this theorem is given in Section~\ref{subsub:d_is_L-convex}.
Corresponding to Lemma~\ref{lem:properties_submo}, we obtain:
\begin{Lem}\label{lem:properties_L-convex}
Let $\pmb{\mGamma}, \pmb{\mGamma}', \pmb{\mLambda}$, and  $\pmb{\mLambda}'$
be modular complexes, and let $g$ and $g'$ be L-convex functions on $\pmb{\mGamma}$.
\begin{itemize}
\item[{\rm (1)}] For $b \in \RR$ and $c,c' \in \RR_+$, 
$b + c g + c' g'$ is an L-convex function on 
$\pmb{\mGamma}$. 
\item[{\rm (2)}] A function $\tilde g$ defined by
\[
\tilde g(p,p') := g(p) \quad ((p,p') \in V_{\mGamma \times \mGamma'})
\]
is an L-convex function on ${\pmb{\mGamma}} \times {\pmb{\mGamma}}'$.
\item[{\rm (3)}] Suppose that ${\pmb{\mGamma}} = {\pmb{\mLambda}} \times {\pmb{\mLambda}}'$.
For $p' \in V_{\mLambda'}$, a function $g_{p'}$ 
defined by 
\[
g_{p'}(p) := g(p,p') \quad (p \in V_{\mLambda})
\]
is an L-convex function on $\pmb{\mLambda}$.
\end{itemize}
\end{Lem}
\begin{proof}
(1) follows from Lemma~\ref{lem:properties_submo}~(1).

(2). $\overline{(\tilde g)}(q/p, q'/p') 
= \{ \tilde g(p,p') + \tilde g(q,q') \}/2 = \{g(p) + g(q)\}/2 = \overline{g}(q/p)$.
Therefore, by Lemma~\ref{lem:properties_submo}~(2), 
function $(q/p, q'/p') \mapsto \overline{g}(q/p)$ is submodular on 
${\cal L}^*_{p}(\pmb{\mGamma}) \times {\cal L}^*_{p'}(\pmb{\mGamma}')
={\cal L}^*_{(p,p')}(\pmb{\mGamma} \times \pmb{\mGamma}')$.

(3). $\overline{g_{p'}}(q/p) = \{ g(p,p') + g(q, p') \}/2 = \overline{g} (q/p, p'/p')$.
Since $\overline{g}$ is submodular on 
${\cal L}_{p}^*(\pmb{\mLambda}) \times {\cal L}_{p'}^*(\pmb{\mLambda}')$,
by Lemma~\ref{lem:properties_submo}~(3),  
$\overline{g_{p'}}$ is submodular on ${\cal L}_{p}^*(\pmb{\mLambda})$.
\end{proof}
The restrictions of an L-convex function to ${\cal L}_p^{+}$ and to ${\cal L}_p^{-}$ are submodular.
\begin{Lem}\label{lem:submo_L_p}
An L-convex function $g$ on $\pmb{\mGamma}$
is submodular on ${\cal L}^+_p$  and on ${\cal L}_p^-$ for each vertex $p$.
\end{Lem}
\begin{proof}
Consider the set ${\cal L}_p^{*+}$ of elements in ${\cal L}_p^*$ of the form $q/p$. 
By Proposition~\ref{prop:1/2},
any vertex in any shortest path between $q/p$ and $q'/p$ is of the form $u/p$.
Hence ${\cal L}_p^{*+}$ is convex in $\mGamma$ and 
in ${\cal L}_p^{*}$. 
By Lemma~\ref{lem:properties_submo}~(4), 
$\overline{g}$ is submodular on ${\cal L}_p^{*+}$. 
Obviously ${\cal L}_p^{*+}$ is isomorphic to ${\cal L}_p^+$ by $q/p \mapsto q$.
By using relation $g(q) = 2 \overline{g}(q/p) - g(p)$  
$(q \in {\cal L}^+_p)$, 
we see the submodularity of $g$ on ${\cal L}^+_p$.
\end{proof}

\paragraph{L-optimality criterion.}

Consider minimization of L-convex functions 
on a modular complex $\pmb{\mGamma}$.
There is an optimality criterion that extends the L-optimality 
criterion of L$^{\natural}$-convex function in discrete convex analysis;
see \cite[Theorem 7.14]{MurotaBook}.
\begin{Thm}[{L-optimality criterion}]\label{thm:L-optimality}
Let $g$ be an L-convex function 
on a modular complex $\pmb{\mGamma}$.
For a vertex $p \in V_\mGamma$, 
the following conditions are equivalent:
\begin{itemize}
\item[{\rm (1)}] $g(p) \leq g(q)$ holds for every $q \in V_{\mGamma}$.
\item[{\rm (2)}] $g(p) \leq g(q)$ holds for every $q \in V_{\mGamma}$ 
with $p \sqsubseteq q$ or $q \sqsubseteq p$. That is
\[
g(p) = \min \{g(q) \mid q \in {\cal L}_p^+\} = \min \{ g(q) \mid  q \in {\cal L}_p^-\}.
\]
\end{itemize}   
\end{Thm}
We prove this theorem in Section~\ref{subsub:L-optimality}.
The condition (2) implies that
$g$ can be minimized 
by tracing the 1-skeleton graph 
of ${\mit\Delta}(\pmb{\mGamma})$.
By Lemma~\ref{lem:submo_L_p}, checking the condition ($2$) reduces to 
the submodular function minimization on modular semilattices,
analogous to the case of L$^{\natural}$-convex function in
discrete convex analysis~\cite[Section 10.3]{MurotaBook}.
Suppose that $\pmb{\mGamma}$ is the product
of modular complexes $\pmb{\mGamma}_i$ 
for $i=1,2,\ldots,n$.
Again we say nothing about 
the complexity of the minimization under the oracle model.
So we consider the VSCP situation.
Recall Section~\ref{subsec:VCSP}. 
By an {\em L-convex constraint} on $\pmb{\mGamma}$ we mean 
an L-convex function $g$ on 
$\pmb{\mGamma}_{i_1} \times \pmb{\mGamma}_{i_2} \times \cdots \times \pmb{\mGamma}_{i_k}$
for some $i_1 < i_2 < \cdots < i_k$.
By Lemma~\ref{lem:properties_L-convex}, 
the sum of L-convex constraints ${\cal G}$ 
is an L-convex function on $\pmb{\mGamma}$.
By L-optimality criterion, 
the optimality check of given vertex $p$ 
reduces to the minimization of 
the sum of submodular constraints
over modular semilattices 
${\cal L}^s_{p}(\pmb{\mGamma}) = {\cal L}^s_{p_1} \times {\cal L}^s_{p_2} \times \cdots \times {\cal L}^s_{p_n}$ for $s \in \{-,+\}$.
By Theorem~\ref{thm:minimized}, we obtain:
\begin{Thm}
Let $\pmb{\mGamma}$ be 
the product of modular complexes 
$\pmb{\mGamma}_i$ $(i=1,2,\ldots,n)$, and 
let ${\cal G}$ be a set of L-convex constraints on $\pmb{\mGamma}$.
Define $g: V_{\mGamma} \to \RR$ by
\[
g(p) := \sum_{h \in {\cal G}} h( p_{I_{h}}) \quad (p = (p_1,p_2,\ldots,p_n) \in V_{\mGamma}).
\] 
For a given vertex $p$,
there exists an algorithm, in time polynomial in $n$ and $|{\cal G}|L^{K}$, 
to find $q \in {\cal L}^-_p \cup {\cal L}^+_p$ with $g(q) < g(p)$ 
or conclude that $p$ is a global minimizer of $g$,  where $L := \max_{1 \leq i \leq n, s\in \{+,-\}} |{\cal L}_{p_i}^{s}|$ and $K := \max_{g \in {\cal G}} k_g$.
\end{Thm}
\paragraph{Steepest descent algorithm.}
Theorem~\ref{thm:L-optimality}, Lemma~\ref{lem:submo_L_p}, and 
Theorem~\ref{thm:minimized}
naturally lead us 
to a descent algorithm for 
L-convex functions on modular complexes, 
analogous to the steepest descent algorithm 
for L-convex function minimization 
in discrete convex analysis.
Starting from an arbitrary point $p$, 
each descent step is to find, for $s\in \{-,+\}$,
an optimal solution $q^{s}$ of the problem:
\begin{myitem}
Minimize $g(q)$ 
over $q \in {\cal L}_p^{s}$. \label{eqn:sigma}
\end{myitem}\noindent
As mentioned already, this is a submodular function minimization.
If $g(p) = g(q^+) = g(q^-)$, then $p$ is optimal. 
Otherwise, take $s \in \{-,+\}$ 
with $g(q^{s}) = \min \{g(q^{-}), g(q^{+}) \} (< g(p))$,  
let $p := q^{s}$ ({\em steepest direction}), 
and repeat the descent step. 
After a finite number of descent steps,  
we can obtain an optimal solution (a minimizer of $g$).

In the case where $f$ 
is an L$^\natural$-convex function 
on a box subset $B$ of $\ZZ^n$, 
Murota~\cite{Murota00} proved that, 
by appropriate choices of steepest directions,
the number of the descent steps is bounded by 
$l_{1}$-diameter of $B$;
later Kolmogorov and Shioura~\cite{KolmogorovShioura} 
improved this bound.
We do not know whether
a similar upper bound exists for 
L-convex function minimizations
on general modular complexes.
This issue will be studied in~\cite{HHprepar}.

\subsection{Proofs}\label{subsec:L-convex_proof}

\subsubsection{Proof of Proposition~\ref{prop:modularlattice} and Theorem~\ref{thm:lattice}}\label{subsub:proof_of_lattice}
We can assume that $h$ is the uniform unit edge weight, 
and $d_{\mGamma,h} = d_{\mGamma}$ is denoted by $d$.
A path $(p_0,p_1,p_2,\ldots,p_k)$ is said to be {\em ascending} 
if $p_i \swarrow p_{i+1}$ for $i=0,\ldots,k-1$.
\begin{Lem}\label{lem:ascending}
For $p,q \in V_\mGamma$ with $p \preceq q$,
a $(p,q)$-path $P$ is shortest 
if and only if $P$ is an ascending path from $p$ to $q$.
In particular, $I(p,q) = [p,q]$,  any maximal chain in $[p,q]$ 
has the same length, and the rank $r$ of $[p,q]$ is given by $r(a) = d(a,p)$.
\end{Lem}
\begin{proof}
Suppose $p \preceq q$. Take
an ascending path $P = (p = p_0,p_1,p_2,\ldots,p_k = q)$.
We use the induction on length $k$; 
the statement for $k=1$ is obvious.

(If part). We show $d(p,q) = k$.
Suppose for contradiction that $d(p,q) < k$.
By induction and bipartiteness, we have
$d(p, p_{k}) = d(p,p_{k-1}) - 1$ and $d(p,p_{k-1}) = k - 1$.
By the quadrangle condition (Lemma~\ref{lem:quadrangle}~(2))
for $p_{k-1},p_{k},p_{k-2},p$, 
there is a common neighbor $q^{*}$ of $p_{k},p_{k-2}$
with $d(p,q^{*}) = d(p,p_{k-1}) - 2 = k-3$.
Consider the 4-cycle of $p_{k-1},p_{k},q^{*},p_{k-2}$.
By orientability $p_{k-2} \swarrow q^{*}$ must hold.
Hence we obtain an ascending path
$(p = p_0, p_1, \ldots, p_{k-2}, q^{*})$
of length $k-1$ with $d(p,q^{*}) = k-3$. 
A contradiction.

(Only if part). 
Take any shortest path
$Q = (p = q_0,q_1,q_2,\ldots,q_{k'} = q)$ between $p$ and $q$.
We show that $Q$ is ascending.
By the if part, necessarily $k' = k$.
It suffices to show that $p \preceq q_{k-1} \swarrow q$; 
by induction $(p = q_0,q_1,q_2,\ldots,q_{k-1})$ is ascending, and hence $Q$ is  ascending.
We can assume that $p_{k-1} \neq q_{k-1}$.
By the quadrangle condition for $q,p_{k-1},q_{k-1},p$, 
there is a common neighbor $q^{*}$ of $p_{k-1},q_{k-1}$ 
with $d(q^{*},p) = d(q,p) -2$.
Then $q^{*} \in I(p,p_{k-1}) = [p,p_{k-1}]$ (by induction).
This means $p \preceq q^* \swarrow p_{k-1}$, 
which in turn implies $q^* \swarrow q_{k-1} \swarrow q$ 
by the orientation of 4-cycle $(p_{k-1},q,q_{k-1},q^{*})$.
Thus $p \preceq q_{k-1} \swarrow q$, as required. 
\end{proof}
Let $(p)^{\uparrow}$ and $(p)^{\downarrow}$ denote 
the principal filter $\{q \in V_{\mGamma} \mid q \succeq p \}$ and 
the principal ideal $ \{q \in V_{\mGamma} \mid q \preceq p \}$
of $p$, respectively.
\begin{Lem}\label{lem:filter}
For $a,b \in (p)^{\uparrow}$,
there uniquely exists a median $m$ of $p,a,b$, 
which coincides with $a \wedge b$.
Similarly, for $a,b \in (p)^{\downarrow}$
there uniquely exists a median $m$ of $q,a,b$, 
which coincides with $a \vee b$.
Both $(p)^{\uparrow}$ and $(p)^{\downarrow}$ are convex, and modular semilattices. 
\end{Lem}
\begin{proof}
It suffices to prove the claim for $(p)^{\uparrow}$.
Suppose that $a,b,p$ have two distinct medians $c,c'$.
Take a median $m$ of $c,c',p$.
Let $k := d(c,m) = d(c',m) > 0$.
We can take an ascending path $(m = m_0,m_1,\ldots,m_k = c)$
from $m$ to $c$, 
and also can take a neighbor $m'$ of $m$ 
with $d(m,c') = 1 + d(m',c')$; necessarily $m \swarrow m'$.
By the quadrangle condition for $m, m_1,m',a$, 
there is a common neighbor $m_1'$
of $m_1,m'$ such that $d(a,m'_1) = d(a,m) -2$.
Also by the quadrangle condition for $m, m_1,m',b$
there is a common neighbor $m_1''$ 
of $m_1,m'$ such that $d(b,m_1'') = d(b,m) - 2$.
By $m_1 \searrow m \swarrow m'$ 
and the orientability,
we have $m_1 \swarrow m'_1 \searrow m'$ 
and $m_1 \swarrow m''_1 \searrow m'$. 
Hence $m'_1 = m''_1$ must hold.
Similarly,
by the quadrangle condition for $m_1,m_2,m_1',a$ 
and for $m_1,m_2,m_1',b$, 
we can find a common neighbor $m_2'$ of $m_2,m'_1$ 
such that $d(m_2,a) = d(m_2',a) + 1$ and $d(m_2,b) = d(m'_2,b) + 1$.
Necessarily $m_2 \preceq m_2' \preceq a,b$.
Repeat this process 
to obtain a neighbor $m'_{k}$ of $m_{k} (= c)$ such that 
$d(c,a) = d(m'_{k},a) + 1$ and $d(c,b) = d(m'_{k},b) + 1$.
This implies that 
$d(a,b) \leq d(a, m'_k) + d(m'_k,b) = 
d(a,c) + d(c,b) -2 = d(a,b) -2$; 
a contradiction.

We show $m = a \wedge b$ in the poset $(p)^{\uparrow}$.
Indeed, take an arbitrary $p' \in (p)^{\uparrow}$ 
with $a \succeq p' \preceq b$.
Consider a median $m'$ of $a,b,p'$.
Since there is an ascending path 
from $p$ to $m'$ using $p'$, 
$m'$ is also a median of $a,b,p$, and $m' = m$ 
by the uniqueness. Hence $p' \preceq m$.

We next show the convexity of $(p)^{\uparrow}$ 
by verifying (3) in Lemma~\ref{lem:2convexity}.
Clearly the subgraph of $\mGamma$ 
induced by $(p)^{\uparrow}$ is connected.
Take $a,b \in (p)^{\uparrow}$ with $d(a,b) = 2$.
We show that $I(a,b) \subseteq (p)^{\uparrow}$. 
From Lemma~\ref{lem:ascending},
this is obvious when $a \preceq b$ or $b \preceq a$.
Thus we may assume 
$a \not \preceq b$ and $b \not \preceq a$.
Consider $a \wedge b$ in $(p)^{\uparrow}$
(the existence of $a \wedge b$ is guaranteed as above).
By $d(a,b) = 2$, $a$ and $b$ cover $a \wedge b$.
By the admissibility of the orientation $o$,
if $a$ and $b$ have another common neighbor $c \neq a \wedge b$,
then  $c$ covers $a$ and $b$, i.e., $c = a \vee b \in (p)^{\uparrow}$.
Hence $I(a,b) \subseteq (p)^{\uparrow}$, and $(p)^{\uparrow}$ is convex.

For arbitrary $a,b \in (p)^{\uparrow}$, 
every maximal common lower bound $m'$ of $a,b$
(not necessarily in $(p)^{\uparrow}$) 
belong to $I(a,b)$; consider a median $m',a,b$, which is equal to $m'$ by the above argument. 
By the convexity, $m'$ belongs to $(p)^{\uparrow}$, 
and is equal to $a \wedge b$ in $(p)^{\uparrow}$.
This means that $a,b$ have the meet in the poset $(V_{\mGamma}, \preceq)$.

Now $(p)^{\uparrow}$ is a semilattice, 
and convex in $\mGamma$. The covering graph of $(p)^{\uparrow}$ 
is equal to the subgraph $\mGamma[(p)^\uparrow]$ induced by $(p)^{\uparrow}$, 
which is modular (by convexity).
By Theorem~\ref{thm:BVV}, $(p)^{\uparrow}$ is a modular semilattice.
\end{proof}
\paragraph{Proof of Proposition~\ref{prop:modularlattice}.}
(1). 
Notice that the convexity is closed under the intersection.
By $[p,q] = (p)^{\uparrow} \cap (q)^{\downarrow}$ and by the previous lemma,
we have the convexity of $[p,q]$. 
Also $[p,q]$ is an interval of a modular semilattice, and hence a modular lattice.

(2).  If $(p,q)$ is Boolean, then 
$q$ is the join of atoms in $[p,q]$, and $[p,q]$ is 
a complemented modular lattice (Theorem~\ref{thm:complemented}).
Conversely, if $[p,q]$ is a complemented modular lattice, then
there is a Boolean sublattice of full rank (generated by a base), 
and $(p,q)$ is Boolean. 

We can check 
whether a pair $(p,q)$ is Boolean by the following procedure. 
First construct the partial order $\preceq$, 
and the poset $(V_{\mGamma}, \preceq)$.
If $p \not \preceq q$, then $(p,q)$ is not Boolean.
Suppose $p \preceq q$. Construct (or enumerate) $[p,q]$.
For each $u\in [p,q]$, 
check the existence of an element (complement) $v \in [p,q]$ with 
$u \wedge v = p$ and $u \vee v = q$.
If every element has a complement, then $[p,q]$ is complemented modular, and
$(p,q)$ is Boolean.
Otherwise, $(p,q)$ is not Boolean. 
This procedure can be done in time polynomial in $|V_{\mGamma}|$.
$\Box$

\paragraph{Proof of Theorem~\ref{thm:lattice}.}
It suffices to consider only ${\cal L}_p^+$.
The statement that ${\cal L}_p^+$ is a semilattice 
immediately
follows from Lemma~\ref{lem:filter} and (\ref{eqn:v}).
Next we show the convexity. 
In view of Lemma~\ref{lem:2convexity},
take $a,b \in {\cal L}^+_p$ with $d(a,b) = 2$, 
and take any common neighbor $c$ of $a,b$.
We show $c \in {\cal L}^+_p$.
This is obvious if
$a \swarrow c \swarrow b$ or $b \swarrow c \swarrow a$.
Also, if $a \searrow c \swarrow b$, 
then $c = a \wedge b \in {\cal L}^+_{p}$
(if $c \neq a \wedge b$, then 
4-cycle $(a,c,b, a \wedge b)$ violates the admissibility of the orientation $o$).

So suppose that $a \swarrow c \searrow b$.
By Lemma~\ref{lem:ascending}, both $[p,a]$ and $[p,b]$ belong to $[p,c]$.
In particular, $c$ is the join of $a$ and $b$ in $[p,c]$.
Here $a$ is the join of atoms in $[p,a]$ and $b$ is the join of atoms in $[p,b]$.
This means that $c$ is the join of atoms in $[p,c]$.
Hence modular lattice $[p,c]$ is complemented (Theorem~\ref{thm:complemented}), 
and hence $(p,c)$ is Boolean.

The subgraph of $\mGamma$ induced 
by any convex set is again a modular graph.
Therefore the covering graph of ${\cal L}^+_{p}$ is modular.
By Theorem~\ref{thm:BVV}, 
${\cal L}^+_{p}$ is a modular semilattice.
In particular, each $[p,q]$ for each $q \in {\cal L}^+_p$ 
is a complemented modular lattice, 
and ${\cal L}^+_{p}$ is a complemented modular semilattice.
$\Box$

\subsubsection{Proof of Theorem~\ref{thm:2subdivision} and Proposition~\ref{prop:1/2}}
\label{subsub:proof_2subdivision}
We start with preliminary results.
By Proposition~\ref{prop:modularlattice}, interval $[p,q]$ is a modular lattice, 
and is convex in $\mGamma$.
So we can consider the projection $\pr_{[p,q]}: V_{\mGamma} \to [p,q]$ 
(see Section~\ref{subsec:gated}).
\begin{Lem}\label{lem:[p,q][p',q']}
For $p,q,p',q' \in V_{\mGamma}$ 
with $p \preceq q$ and $p' \preceq q'$, 
let $u,v,u',v'$ be defined by
\begin{equation}\label{eqn:uvu'v'}
u := {\rm Pr}_{[p, q]} (p'), \ 
v := {\rm Pr}_{[p,q]} (q'), u' := {\rm Pr}_{[p',q']} (p), \ v' := {\rm Pr}_{[p',q']} (q).
\end{equation}
Then we have:
\begin{itemize}
\item[{\rm (1)}]  $u \preceq v$, $u' \preceq v'$, $\pr_{[p',q']}([p,q]) = [u', v']$, and $\pr_{[p,q]}([p',q']) = [u, v]$.
\item[{\rm (2)}] $[u,v]$ is isomorphic to $[u',v']$ by map $w \mapsto \pr_{[p',q']}(w)$. 
In particular, $p \sqsubseteq q$ implies $u' \sqsubseteq v'$.
\end{itemize}
\end{Lem}
\begin{proof}
The image of a convex set by the projection is again convex (Theorem~\ref{thm:gatedset}), 
and a convex set in a modular lattice is exactly an interval (Lemma~\ref{lem:interval}~(2)).
Hence $\pr_{[p',q']}([p,q])$ is equal to a subinterval $[a',b']$ of $[p',q']$, 
and $\pr_{[p,q]}([p',q'])$ is equal to a subinterval $[a,b]$ of $[p,q]$.
By Theorem~\ref{thm:gatedset}, 
$\Pr_{[p,q]}$ is an isometry between $[a,b]$ and $[a',b']$, 
and consequently this induces a graph isomorphism 
between the subgraphs induced by $[a,b]$ and $[a',b']$.
By Lemma~\ref{lem:commonorbit}, 
for $g,h \in [a,b]$, 
if $g \swarrow h$ then $\pr_{[p',q']}(g) \swarrow \pr_{[p',q']}(h)$.
Therefore, a poset $[a,b]$ is isomorphic to $[a',b']$ by $\Pr_{[p',q']}$.
Necessarily $\pr_{[p,q]}(a') = a$.
Notice that $a$ is the gate of $p$ at $[a,b]$ (by Lemma~\ref{lem:ascending}). 
By Theorem~\ref{thm:gatedset}~(3), 
$a = \pr_{[a,b]} (p) = \pr_{[p,q]} \circ \pr_{[p',q']}(p) = \pr_{[p,q]}(u')$.
Then $\pr_{[p,q]}(u') = a = \pr_{[p,q]}(a')$ implies
$a' = u'$ (by Theorem~\ref{thm:gatedset}~(1)).
Similarly $b' = v'$, $a = u$, and $b = v$.
Thus we obtain (1) and (2).
\end{proof}
We use the same notation $d$ for $d_{\mGamma}$ and $d_{\mGamma^*,1/2}$
(since they can be distinguished by the arguments).
\begin{Lem}\label{lem:atleast}
$\displaystyle
d(q/p, q'/p') = \frac{d(p,p') + d(q,q')}{2}$ for 
$q/p,q'/p' \in V_{\mGamma^*}$.
\end{Lem}
\begin{proof}
Take a path 
$P = (q/p = q_0/p_0,q_1/p_1,\ldots, q_k/p_k = q'/p')$ 
between $q/p$ and $q'/p'$ in $\mGamma^*$.
The length of $P$ is equal to 
$\sum_{i=0}^{k-1}(d(q_i,q_{i+1}) + d(p_i,p_{i+1}))/2 \geq 
(d(q,q') + d(p,p'))/2$;
hence $(\geq)$ holds. 

We show the equality by the induction on 
$d(p,p') + d(q,q')$.
Define $u,v,u',v'$ by~(\ref{eqn:uvu'v'}).
Then $d(p,p') = d(p,u) + d(u,u') + d(u',p')$ 
and $d(q,q') = d(q,v) + d(v,v') + d(v',q')$ hold. 
Suppose $p \prec u$.
Take an atom $a$ of $[p,u]$.
Then $(a,q)$ is Boolean by (\ref{eqn:v}), and $q/a$ is adjacent to $q/p$ in $\mGamma^*$.
Also $d(a,u) = d(p,u) -1$ (Lemma~\ref{lem:ascending}), 
implying $d(a,p') = d(p,p') - 1$.
By induction, $d(q/a,q'/p') = (d(a,p') + d(q,q'))/2$, 
and hence $d(q/p,q'/p') \leq (d(a,p') + d(q,q'))/2 + 1/2 = (d(p,p') + d(q,q'))/2$, as required.

Consequently we can assume $(p,q,p',q') = (u,v,u',v')$.
In particular, $d(p,q) = d(p',q')$, $d(p,p') = d(q,q')$, 
and $d(p,q') = d(p,q) + d(q,q') = d(q,p')$ must hold.
Take a neighbor $a$ of $p$ in $I(p,p')$; then $a \not \in [p,q]$.
Take a median $b$ of $q,a,q'$.
Then $b$ must be a neighbor of $q$ and satisfies
$d(p,q) = d(a,b) = d(p,b) - 1 = d(a,q) - 1$.
If $q \swarrow b$, then $p \swarrow a \preceq b$ (Lemma~\ref{lem:ascending}), 
and the join $q \vee a$ in $[p,b]$ is equal to $b$; 
this means that $[p, b]$ is a complemented modular lattice, 
$(p,b)$ is Boolean, and $b/p$ is adjacent to $q/p$.
Applying the induction to $b/p$ and $q'/p'$, 
we obtain the equality ($=$).
Similarly, if $b \swarrow q$, then 
$(a,q)$ is Boolean
and apply the induction to $q/a$ and $q'/p'$. 
\end{proof}

\paragraph{Proof of Theorem~\ref{thm:2subdivision}.}
Any 4-cycle in $\mGamma^*$ is represented 
as $(q/p,q'/p, q'/p', q/p')$ for some edges $pp',qq'$ in $\mGamma$, or 
$(x/p,y/p, z/p, w/p)$ or $(p/x, p/y, p/z, p/w)$ 
for 4-cycle $(x,y,z,w)$ and vertex $p$ in $\mGamma$.
This immediately implies that $o^*$ 
is an admissible orientation 
and $h^*$ is orbit-invariant.

To show that $\mGamma^*$ is modular,
we are going to verify that $\mGamma^*$ 
satisfies the two conditions of Lemma~\ref{lem:quadrangle}.
If $q/p$ and $q'/ p'$ are joined by an edge,
then $d_{\mGamma}(p,q)$ and $d_{\mGamma}(p',q')$ 
have different parity. This implies that $\mGamma^*$ is bipartite.

We next verify
the quadrangle condition (Lemma~\ref{lem:quadrangle}~(2)).
Take boolean pairs $q/p$ and $q'/p'$.
Suppose further that we are given 
two neighbors $q_1/p_1$ and $q_2/p_2$ 
of $q/p$ with $d(q/p,q'/p') 
= 1/2 + d(q_1/p_1, q'/p') = 1/2 + d(q_2/p_2, q'/p')$.
Our goal is to show the existence of 
a common neighbor $q^{*}/p^{*}$
of $q_1/p_1, q_2/p_2$ 
with $d(q/p, q'/p') = 1 + d(q^*/p^*, q'/p')$.

It suffices to consider the following three cases:
\begin{itemize}
\item[(i)] $p_1 = p = p_2$.
\item[(ii)] $p_1 = p$, $q_1 \swarrow q$, and $q_2 = q$. 
\item[(iii)] $p_1 \swarrow p = p_2$ and $q_1 = q \swarrow q_2$. 
\end{itemize}

\noindent
{\bf Case (i).} 
By Lemma~\ref{lem:atleast}, 
we have $d(q,q') = 1 + d(q_i,q')$ for $i=1,2$.
By Lemma~\ref{lem:quadrangle}~(2) for $\mGamma$, 
there is a common neighbor $q^*$ of $q_1,q_2$ 
with $d(q,q') = 2 + d(q^*,q')$.
Here $p \sqsubseteq q_i$ $(i=1,2)$, 
and hence $q_i \in {\cal L}^+_p$.  
By the convexity of ${\cal L}^+_{p}$ (Theorem~\ref{thm:lattice}), 
we have $q^{*} \in {\cal L}^+_{p}$, implying $p \sqsubseteq q^*$.
Again, by Lemma~\ref{lem:atleast},
we have $d(q/p,q'/p') = 1 + d(q^*/p,q'/p')$, as required.\\

\noindent
{\bf Case (ii).} We show $p_2 \sqsubseteq q_1$, which implies that 
$q_1/p_2$ is a required common neighbor (by Lemma~\ref{lem:atleast}).
If $p_2 \swarrow p$, then $p_2 \sqsubseteq q$ 
and $p_2 \swarrow p \preceq q_1 \swarrow q$ 
imply $p_2 \sqsubseteq q_1$ (by (\ref{eqn:v})).
Suppose $p \swarrow p_2$.
By Lemma~\ref{lem:[p,q][p',q']},  
$\pr_{[p,q]}([p',q'])$ is equal to interval $[a,b]$ for
$a = \pr_{[p,q]}(p')$ and $b = \pr_{[p,q]}(q')$. 
Then $d(p_2,p') = d(p,p') - 1$ implies $p_2 \preceq a$.
Similarly $b \preceq q_1$.
Thus $p \preceq p_2 \preceq a \preceq b \preceq q_1 \preceq q$
and $p \sqsubseteq q$ imply $p_2 \sqsubseteq q_1$ 
(by (\ref{eqn:v})), as required.\\

\noindent
{\bf Case (iii).}
We show $p_1 \sqsubseteq q_2$;
then $q_2/p_1$ is a common neighbor as required.
We use the induction on $d(p_1,q_2)$;
in the case of $d(p_1,q_2) = 2$ the subsequent argument shows that
$I(p_1,q_2) = [p_1,q_2]$ contains an element different from $p_1,q_2, p$, 
and hence $(p_1,q_2)$ is Boolean.
Since $[p_1,q_2]$ is a convex set, 
we can consider $\pr_{[p_1,q_2]}([p',q'])$, which is equal to $[u,v]$ 
for $u:= \pr_{[p_1,q_2]}(p')$, $v:=  \pr_{[p_1,q_2]}(q')$ (Lemma~\ref{lem:[p,q][p',q']}).
Then necessarily $p \not \preceq u$; otherwise 
$d(p_1,p') = d(p_1,u) + d(u,p') = 1 + d(p,u) + d(u,p') = 1 + d(p,p')$, 
contradicting the first assumption $d(p_1,p') = d(p,p') - 1$.

Suppose $p_1 \neq u$.
Take an atom $a \in [p_1,u]$.
If $a \not \preceq q$, 
then $q_2$ is equal to the join $a \vee q$, and consequently 
the join of atoms in $[p_1,q_2]$. Thus $[p_1,q_2]$ is complemented, and 
$p_1 \sqsubseteq q_2$ (Proposition~\ref{prop:modularlattice}).
Suppose that $a \preceq q$.
Consider the join $a \vee p$ in $[p_1,q_2]$, which belongs to $[p,q]$.
Then $d(a \vee p,p') = d(a \vee p, u) + d(u,p') = 1 + d(a,u) + d(u,p') = 1+ d(a,p')$ 
(by $a \vee p \not \prec u$ and Lemma~\ref{lem:ascending}).
By (\ref{eqn:v}),  both $(a,q_1)$ and $(a \vee p, q_2)$ are Boolean.
By induction, $a  \sqsubseteq q_2$, and $[a, q_2]$ is complemented modular.
Thus we can take an atom $b (\neq a \vee p)$ of $[a,q_2]$ with $q \vee b = q_2$.
Necessarily $p \not \preceq b$ 
(otherwise both $a$ and $p$ are covered by $c$ and $a \vee p$; contradicting 
the fact that $[p_1,q_2]$ is a lattice).
By modularity equality with $p \wedge b = p_1$,  
the join $p \vee b$ has rank $3$ in $[p_1,q_2]$.
We can take $c$ with $p  \swarrow c \swarrow b \vee p$.
Then $d(p,b) = 3$, and $d(p_1,b) = d(c,b) = 2$.
By the quadrangle condition, there is a common neighbor 
$w$ of $p_1,c, b$, which is an atom of $[p_1,q]$.
Then $a \vee w = b$, and $q_2 = q \vee b = q \vee a \vee w$.
This means that $q_2$ is the join of atoms of $[p_1,q_2]$; 
thus $p_1 \sqsubseteq q_2$.

By the same argument (for $v, q_2$),
we can assume that $u = p_1$ and $v = q_2$.
Since $[u,v]$ is isomorphic to a subinterval of complemented modular lattice $[p',q']$ 
(Lemma~\ref{lem:[p,q][p',q']}),
$[p_1, q_2] (= [u,v])$ is also a complemented modular lattice, implying $p_1 \sqsubseteq q_2$.
$\Box$
\paragraph{Proof of Proposition~\ref{prop:1/2}.}
We have proved (\ref{eqn:isometric_embedding}) for the case $h=1$ in Lemma~\ref{lem:atleast}. 
By Theorem~\ref{thm:2subdivision} shown above, $\mGamma^*$ is now a modular graph.
Let $P$ be a shortest path with respect to $h^*$.
By Lemma~\ref{lem:shortest},
this is also shortest 
with respect to uniform edge-length $1/2$.
Necessarily the paths obtained from 
$(q=q_0,q_1,q_2,\ldots,q_k = q')$ and 
$(p = p_0,p_1,p_2,\ldots,p_k = p')$
(by identifying repetitions)
are both shortest in $\mGamma$ 
with respect to uniform edge-length $1$.
Again, by Lemma~\ref{lem:shortest}, 
they are shortest relative to $h$, 
and have the lengths $d_{\mGamma,h}(p,p')$ and $d_{\mGamma,h}(q,q')$, 
respectively. 
Thus (\ref{eqn:isometric_embedding}) holds.
$\Box$

\subsubsection{Proof of Theorems~\ref{thm:d_is_submodular} and~\ref{thm:d_is_L-convex}}\label{subsub:d_is_L-convex}
Consider the 2-subdivision $(\pmb{\mGamma} \times {\pmb\mGamma})^{*}$, 
which is identified with $\pmb{\mGamma}^* \times \pmb{\mGamma}^*$ 
by correspondence $(q,q')/(p,p') \leftrightarrow (q/p,q'/p')$
(Lemma~\ref{lem:product}). 
Consider 
$\overline{d_{\mGamma,h}}: V_{(\mGamma \times \mGamma)^{*}} \to \RR$.
Then we have
\[
\overline{d_{\mGamma, h}}((q,q')/(p,p')) = 
\frac{ d_{\mGamma,h}(p,p') + d_{\mGamma,h}(q,q')}{2} 
= d_{\mGamma^*, h^*}(q/p, q'/p'),
\]
where the first equality is the definition (\ref{eqn:Lovasz}) 
and the second follows from Proposition~\ref{prop:1/2}.
Hence it suffices to show that 
$d_{\mGamma^*,h^*}: V_{\mGamma^*} \times V_{\mGamma^*} \to \RR$ 
is submodular on ${\cal L}^+_{(a/a,b/b)}(\pmb{\mGamma}^* \times \pmb{\mGamma}^*) 
= {\cal L}^+_{a/a}(\pmb{\mGamma}^*)  \times {\cal L}^+_{b/b}(\pmb{\mGamma}^*)$ 
for every $(a,b) \in V_{\mGamma} \times V_{\mGamma}$.
Therefore, by taking $\pmb{\mGamma}^*$ as $\pmb{\mGamma}$, 
Theorem~\ref{thm:d_is_L-convex} follows from the following.
\begin{Lem}
The distance function $d_{\mGamma, h}$ 
is submodular on 
${\cal L}^+_a \times {\cal L}^+_b$ for every $a,b \in V_{\mGamma}$.
\end{Lem} 
\begin{proof}
By Proposition~\ref{prop:2section}, 
it suffices to show the following,
where we denote $d_{\mGamma,h}$ by $d$, 
and denote the valuation on ${\cal L}^+_a$
(defined in (\ref{eqn:valuations})) by $v$.

\begin{itemize}
\item[(1)] For every $u \in {\cal L}^+_b$ and
every antipodal pair $(p,q)$ in ${\cal L}^+_a$, we have
\[
v[p \wedge q,q] d (p,u) + v[p \wedge q,p] d(q,u) \geq 
(v[p \wedge q,p] + v[p \wedge q, q]) d(p \wedge q, u).
\]
\item[(2)] 
For every $u \in {\cal L}^+_{b}$ and every $2$-bounded pair $(p,q)$ in ${\cal L}^+_{a}$, 
we have
\[
d(p,u) + d(q,u) \geq d(p \wedge q,u) + d(p \vee q, u).
\]
\item[$(2')$] For every $p,q \in {\cal L}^+_{a}$ with $p \swarrow q$ 
and every $p', q' \in {\cal L}^+_b$ with $p' \swarrow q'$,
we have
\[
d(q,p') + d(p,q') \geq d(q,q') + d(p, p').
\]
\end{itemize}
Note that (2) and $(2')$ correspond 
to the submodularity condition for $2$-bounded pairs.

(1). We may assume that $p \wedge q = a$ 
(by considering ${\cal L}_{p \wedge q}^+$) and $v(a) = 0$.
Take a median $m$ of $p,q,u$.
By $m \in I(p,q)$ and Lemma~\ref{lem:I(p,q)}, 
there are $p' \in [a,p]$ 
and $q' \in [a,q]$ with $m = p' \vee q'$.
Let $D := d(m,u)$.
Then we have $d(p,u) =  v[p',p] + v(q') + D$, 
$d(q,u) =  v[q',q] + v(p') + D$, 
and $d(a,u) \leq v(p') + v(q') + D$.
Hence we get
\begin{eqnarray*}
&& v(q) d(p, u) + v(p) d(q, u) - (v(p) + v(q))d(a, u) \\
&& \geq  v(q) \{ v[p',p] + v(q') + D\} 
+ v(p) \{ v[q',q] + v(p') + D\}  \\ 
&& \quad \quad \quad - \{v(p) + v(q)\} \{v(p') + v(q') + D\} \\
&& = v(q) v[p',p] +  v(p) v[q',q] - v(p)v(q') - v(q)v(p') \\
&& = 2 v[q',q] v[p',p] - 2 v(p')v(q'), 
\end{eqnarray*}
where we use $v(p) = v(p') + v[p',p]$ and $v(q) = v(q') + v[q',q]$ for the last equality.
This must be nonnegative
since $(p,q)$ is antipodal; see (\ref{eqn:antipodal})

(2). Recall the notion of gated sets 
(Section~\ref{subsec:gated}); 
$[p \wedge q, p \vee q]$ is convex, and is gated 
(Lemmas~\ref{lem:2convexity} and \ref{lem:interval}).
Let $m:= \pr_{[p \wedge q, p \vee q]}(u)$, and $D := d(m,u)$. 
Then we have $d(x,u) = d(x,m) + D$ for 
$x \in \{p,q,p \wedge q, p \vee q\}$.
There are three cases: (i) 
$m \in \{p,q\}$, (ii) $m \in \{p \wedge q, p \vee q\}$, 
and (iii) $m \not \in \{p,q, p \wedge q, p \vee q\}$.
Note that 
$(p,p \wedge q, q, p \vee q)$ forms a 4-cycle 
since $(p,q)$ is $2$-bounded.
Let $\alpha := v[p,p \vee q] = v[p \wedge q,q]$ 
and $\beta:= v[p \wedge q, p] = v[q, p \vee q]$.
Consider the case (i). 
Then $\{ d(p \wedge q, m), d(p \vee q, m)\} = \{\alpha, \beta\}$ and 
$\{d(p,m), d(q,m)\} = \{0, \alpha+\beta \}$. 
Hence 
$d(p,u) + d(q,u) - d(p \wedge q,u) - d(p \vee q, u) = 0$.
Consider the case (ii).
Then $\{d(p, m), d(q, m)\} = \{ \alpha, \beta\}$, 
and $\{d(p \wedge q,m), d(p \vee q, m)\} = \{0, \alpha+\beta\}$.
Hence 
$d(p,u) + d(q,u) - d(p \wedge q,u) - d(p \vee q, u) = 0$.
Consider the case (iii). 
Then $m$ is a common neighbor of $p \wedge q,p \vee q$ 
different from $p,q$.
Hence all edges in $[p \wedge q, p \vee q]$ 
belong to the same orbit.
Thus $\alpha = \beta$, $d(p,m) = d(q,m) = 2\alpha$, 
$d(p \wedge q,m) = d(p \vee q, m) = \alpha$, 
and
$d(p,u) + d(q,u) - d(p \wedge q,u) - d(p \vee q, u) = 2\alpha > 0$.

$(2')$. Consider $\pr_{\{p',q'\}}(\{p,q\})$. 
Let $D := d(\{p,q\}, \{p',q'\})$.
There are two cases: (i) $|\pr_{\{p',q'\}}(\{p,q\})| = 1$
and (ii) $\{p', q'\} = \pr_{\{p',q'\}}(\{p,q\})$. 
Consider the case (i).
For $u,v,u',v'$ with
$\{u,v\} = \{p,q\}$ and $\{u',v'\} = \{p',q'\}$, 
we have $d(v,u') = D$, $d(u,u') = D + h(uv)$, 
$d(v,v') = D + h(u'v')$, 
and $d(u,v') = D + h(uv) + h(u'v')$.
Thus $d(u,u') + d(v,v') = d(u,v') + d(v,u')$, 
and 
the equality holds in $(2')$.
Consider the case (ii). 
By Theorem~\ref{thm:gatedset} and Lemma~\ref{lem:commonorbit},
we have $p' = \pr_{\{p',q'\}}(p)$, 
$q' = \pr_{\{p',q'\}}(q)$, $d(p,p') = d(q,q') = D$, and
that $pq$ and $p'q'$ must 
belong to the same orbit $Q$.
Then $d(p,q') = d(q,p') = D + h_Q$.
Therefore $(2')$ holds.
\end{proof}

The above proof works even when ${\cal L}^{+}_{a}$ and ${\cal L}^{+}_{b}$ are replaced by 
the principal filters $(a)^{\uparrow}$ and $(b)^{\uparrow}$, respectively, 
since they are convex and are  (not necessarily complemented)
modular semilattices (Lemma~\ref{lem:filter}).
Therefore Theorem~\ref{thm:d_is_submodular} follows from Theorem~\ref{thm:BVV} and:
\begin{Lem}
The distance function $d_{\mGamma,h}$ 
is submodular on 
$(a)^{\uparrow} \times (b)^{\uparrow}$ for every $a,b \in~V_{\mGamma}$.
\end{Lem}

\subsubsection{Proof of L-optimality criterion (Theorem~\ref{thm:L-optimality})}
\label{subsub:L-optimality}
Let $\pmb{\mGamma} = (\mGamma,o,h)$ be a modular complex 
and let $g$ be a function on $V_\mGamma$.
Let $\bar{\mGamma}$ denote the graph obtained from 
$\mGamma$ by joining all Boolean pairs $(p,q)$ 
(with $d_{\mGamma}(p,q) \geq 2$).
Namely $\bar{\mGamma}$ 
is the 1-skeleton graph of the complex $\mit{\Delta}(\pmb{\mGamma})$.
For $\alpha \in \RR$, 
the {\em level-set subgraph} $\bar{\mGamma}_{g, \alpha}$
is the subgraph of $\bar{\mGamma}$ 
induced by the set of vertices $p$ 
with $g(p) \leq \alpha$.
The following connectivity property 
of $\bar{\mGamma}_{g, \alpha}$ rephrases 
the L-optimality criterion (Theorem~\ref{thm:L-optimality}).
\begin{Prop}\label{prop:connected}
Let $g$ be an L-convex function on $\pmb{\mGamma}$, and let $l := \min_{p \in V_{\mGamma}} g(p)$.
For every $\alpha \geq l$,  the level-set graph $\bar{\mGamma}_{g,\alpha}$ 
is connected. In addition, if $\alpha >  l$, then
every vertex in $g^{-1}(\alpha)$ is adjacent to a vertex of
$\bar{\mGamma}_{g,\alpha} \setminus g^{-1}(\alpha)$ 
in $\bar{\mGamma}_{g,\alpha}$.
\end{Prop}
In a crucial step of the proof, 
we use the following general property 
of submodular functions 
on a modular semilattice, where
a sequence $(p_0,p_1,\ldots, p_m)$ of elements in a poset
is said to be {\em comparable} if $p_i \preceq p_{i+1}$ or  $p_{i+1} \preceq p_{i}$ for $i=0,1,2,\ldots,m$.
\begin{Lem}\label{lem:2term}
Let $f$ be a submodular function 
on a modular semilattice ${\cal L}$.
For $p,q \in {\cal L}$ and $\alpha \in \RR$, 
if $f(p) \leq \alpha$ and $f(q)  < \alpha$, 
there exists a comparable sequence
$(p = p_0, p_1,p_2,\ldots, p_m = q)$ such that 
$f(p_i) < \alpha$ for $i=1,2,\ldots, m$.
\end{Lem}
\begin{proof}
We can assume that $\alpha = 0$ by letting $f \leftarrow f - \alpha$.
Also we may assume that $p$ and $q$ 
are incomparable; $\Conv I(p,q)$ is a polygon, 
and hence $[C(p;p,q)] < 1$ and  $[C(q;p,q)] < 1$.
Consider inequality (\ref{eqn:gsubmo}):
\[
(1 - [C(p;p,q)]) f(p) + (1 - [C(q;p,q)]) f(q) 
\geq f(p \wedge q) + \sum_{u \in {\cal E}({p,q}) \setminus \{p,q\}} [C(u;p,q)] f(u).
\]
Then
the left hand side is negative, and hence the right hand side is negative.
If $f(p \wedge q) < 0$, 
then $(p, p \wedge q, q)$ is a required sequence.
Otherwise there exists 
$u \in {\cal E}({p,q}) \setminus \{p,q\}$ with $f(u) < 0$ 
(since $[C(u;p,q)]$ is nonnegative).
By an inductive argument (on distance between $p$ and $q$), 
there are comparable sequences $(p,p_1,p_2,\ldots,p_k=u)$ 
and $(u,q_1,q_2,\ldots,q_{k'} = q)$ 
with $f(p_i) < 0$ and $f(q_j) < 0$.
Concatenating them, 
we obtain a required sequence. 
\end{proof}
\paragraph{Proof of Proposition~\ref{prop:connected}.}
Suppose (indirectly) 
that $\bar{\mGamma}_{g,\alpha'}$ 
is disconnected for some $\alpha'$.
For a sufficiently large $\alpha$, 
the graph $\bar{\mGamma}_{g, \alpha}$ is equal to $\bar{\mGamma}$, 
and is connected (since $\mGamma$ is finite and $g$ has no infinite value).
Also, for a sufficiently small $\epsilon > 0$, 
it holds $\bar{\mGamma}_{g, \alpha - \epsilon} 
= \bar{\mGamma}_{g,\alpha} \setminus g^{-1}(\alpha)$.
This implies that there exists $\alpha \geq l$
such that $\bar{\mGamma}_{g, \alpha}$ is connected, 
and $\bar{\mGamma}_{g, \alpha} \setminus g^{-1}(\alpha)$ 
is disconnected.
Then there exists a pair of vertices $p,p'$ belonging 
to different components in 
$\bar{\mGamma}_{g, \alpha} \setminus g^{-1}(\alpha)$; 
in particular $g(p) < \alpha$ and $g(p') < \alpha$.
Take such a pair $(p,p')$ with
$k := d_{\bar{\mGamma}_{g, \alpha}}(p,p')$ minimum.
There exists a path $(p = p_0,p_1,\ldots,p_k = p')$ 
in $\bar{\mGamma}_{g, \alpha}$
with $g(p_i) = \alpha$ for $i=1,2,\ldots,k-1$.

We first show $k=2$.
Consider ${\cal L}^*_{p_1}$ and $\bar g$ on ${\cal L}^*_{p_1}$.
We may assume that $p_1 \sqsubseteq p_2$.
Let $u := p/p_1$ if $p_1 \sqsubseteq p$ and $u:= p_1/p$ 
if $p \sqsubseteq p_1$.
Then $\bar g(u) < \alpha$ and 
$\bar g(p_2/p_1) \leq \alpha$.
Therefore, by Lemma~\ref{lem:2term}, 
there exists a comparable sequence
$(u = u_0,u_1,\ldots,u_{m-1}, u_m = p_2/p_1)$ 
in ${\cal L}^*_{p_1}$ such that $\bar g (u_i) < \alpha$ 
for $i=0,1,2,\ldots, m-1$.
Consider $u_{m-1}$, which is equal to $q'/q$ for some 
$q,q' \in V_{\mGamma}$ with $q \sqsubseteq q'$. 
Then (i) $p_2/p_1 \sqsubseteq_{*} q'/q$ or (ii) $q'/q \sqsubseteq_{*} p_2/p_1$.

Consider case (i).
By (\ref{eqn:p'pqq'}), we have 
$q \preceq p_1 \preceq p_2 \preceq q'$.
By (\ref{eqn:v}), 
we have $q \sqsubseteq p_i \sqsubseteq q'$ for $i=1,2$.
Thus both $q$ and $q'$ are 
adjacent to each of $p_1$ and $p_2$ (in $\bar{\mGamma}$).
By $\bar g (u_{m-1}) < \alpha$, 
we have $g(q) < \alpha$ or $g(q')  < \alpha$.
Say $g(q) < \alpha$; $q$ is adjacent to $p_1$ and $p_2$ 
in $\bar{\mGamma}_{g,\alpha}$.
If $q$ and $p'$ belongs different components in 
$\bar{\mGamma} \setminus g^{-1}(\alpha)$, 
then path $(q,p_2,p_3,\ldots, p_k = p')$ violates 
the minimality assumption.
This means that $q$ and $p'$ 
belong to the same component, 
which is different from the component that $p$ belongs to.
Thus we could have chosen path $(p,p_1,q)$ of length $2$. 
This implies that $k=2$ and $g(p_2) < \alpha$. 

Consider case (ii). By (\ref{eqn:p'pqq'}) and $q'/q \in {\cal L}_{p_1}^*$, 
we have $p_1 \preceq q \preceq q' \preceq p_2$ and $q \preceq p_1 \preceq q'$.
Hence $q = p_1$, and $p_1 \sqsubseteq q' \sqsubseteq p_2$ (by (\ref{eqn:v})).
Also, we have $g(q') < \alpha$,  
and $q'$ is adjacent to each of $p_1$ and $p_2$.
As above, by the minimality, 
we must have $k = 2$ and $g(p_2) < \alpha$.

Suppose that 
$u_i$ is represented by $u_i = q'_i/q_i$ 
for $q_i,q'_i \in V_{\mGamma}$ with $q_i \sqsubseteq q'_i$
($i=0,1,2,\ldots,m$).
Then $q_i \preceq q_{i+1} \preceq q'_{i+1} \preceq q'_{i}$ 
or $q_{i+1} \preceq q_{i} \preceq q'_{i} \preceq q'_{i+1}$.
Again by (\ref{eqn:v}),
both $q_i$ and $q'_i$ are adjacent to 
each of $q_{i+1}$ and $q'_{i+1}$ in $\bar{\mGamma}$.
Also, by $\bar g(q'_i/q_i) < \alpha$, 
at least one of $g(q_i)$ and $g(q'_i)$
is less than $\alpha$.
This means that there is a path in 
$\bar{\mGamma}_{g, \alpha} \setminus g^{-1}(\alpha)$ 
connecting $p$ and $p'$.
This is a contradiction to the initial assumption 
that $p$ and $p'$ belong to 
distinct components in 
$\bar{\mGamma}_{g, \alpha} \setminus g^{-1}(\alpha)$.

We show the latter part.
Take $p \in g^{-1}(\alpha)$. 
Then there is a pair of $q \in V_{\mGamma}$ 
and a path $(q = p_0, p_1,p_2, \ldots, p_k = p)$ 
in $\bar{\mGamma}$ such that $g(q) < \alpha$ and
$g(p_i) = \alpha$ for $i=1,2,\ldots,k$.
Take such a pair with the minimum length $k$. 
We show $k=1$.
Suppose that $k \geq 2$.
As above, by considering $\bar g$ on ${\cal L}^*_{p_1}$, 
we can find a neighbor $q'$ of $p_2$ with $f(q') < \alpha$.
This is a contradiction to the minimality of $k$. 
Hence $k=1$, implying the latter statement. $\Box$

\section{Minimum 0-extension problems}
\label{sec:0extension}
In this section, we study, 
from the viewpoint developed in the previous sections, 
the minimum $0$-extension problem {\bf 0-Ext}$[\mGamma]$ 
on an orientable modular graph $\mGamma$.
In Section~\ref{subsec:location}, 
we verify that {\bf 0-Ext}$[\mGamma]$ can be formulated as 
an L-convex function minimization on a modular complex ${\pmb{\mGamma}}^n$.
In Section~\ref{subsec:criterion},
we present a powerful optimality criterion 
(Theorem~\ref{thm:optimality})
for {\bf 0-Ext}$[\mGamma]$ 
by specializing the L-optimality criterion 
(Theorem~\ref{thm:L-optimality}).
In Section~\ref{subsec:proof_main}
we prove the main theorem (Theorem~\ref{thm:main}) of this paper.
In Section~\ref{subsec:nongraphical} 
we consider the minimum 0-extension problem for metrics, 
not necessarily graph metrics,
and extend Theorem~\ref{thm:main} to metrics.

\subsection{L-convexity of multifacility location functions}\label{subsec:location}
Let $\mGamma$ be an orientable modular graph 
with an orbit-invariant function $h$.
We are given a finite set $V$ 
with $V_{\mGamma} \subseteq V$. 
Suppose that $V \setminus V_{\mGamma}  = \{1,2,\ldots,n\}$.
For a nonnegative cost $c:{V \choose 2} \to \QQ_{+}$, 
let $(c \cdot d_{\mGamma,h})$  be a function on  $({V_{\mGamma}})^n$ defined by
\begin{eqnarray*}
 (c \cdot d_{\mGamma,h}) (\rho) & := & \sum_{s \in V_{\mGamma}} \sum_{1 \leq j \leq n} c(s j) d_{\mGamma,h}(s, \rho_j)
+ \sum_{1 \leq i < j \leq n} c(i j) d_{\mGamma,h}(\rho_i, \rho_j) \\
&&  \hspace{4cm} (\rho = (\rho_1,\rho_2,\ldots,\rho_n) \in (V_{\mGamma})^n).
\end{eqnarray*}
Such a function is called a {\em multifacility location function} on $\mGamma$.
A point $\rho$ in $(V_{\mGamma})^n$ is called {\em a location}.
Consider the following natural weighted version of {\bf 0-Ext}$[\mGamma]$:
\begin{description}
\item[{{\bf Multifac}$[\mGamma,h; V,c]$}:] \quad \quad Minimize $(c \cdot d_{\mGamma,h}) (\rho)$
over all locations $\rho \in (V_{\mGamma})^n$, 
\end{description}
where the unweighted version corresponds to $h = 1$.

Fix an admissible orientation $o$ of $\mGamma$.
By a natural identification $(V_{\mGamma})^n \simeq V_{\mGamma^n}$, 
a location is regarded as a vertex in $V_{\mGamma^n}$.
In particular, $(c \cdot d_{\mGamma,h})$ is regarded as a function on $V_{\mGamma^n}$.
By Theorem~\ref{thm:d_is_L-convex} and 
Lemma~\ref{lem:properties_L-convex}, we have:
\begin{Thm}\label{thm:L-convex}
Multifacility location function 
$(c \cdot d_{\mGamma, h})$ is an L-convex function on modular complex 
$\pmb{\mGamma}^n$.
\end{Thm}
Therefore {\bf Multifac}$[\mGamma,h; V, c]$ 
is an L-convex function minimization on modular complex $\pmb{\mGamma}^n$.
So we can apply the results in the previous section to {\bf Multifac}$[\mGamma,h; V, c]$.

\subsection{Optimality criterion and orbit-additivity}
\label{subsec:criterion}
Let $\rho = (\rho_1,\rho_2,\ldots,\rho_n)$ be a location.
A location $\rho' = (\rho'_1,\rho'_2,\ldots,\rho'_n)$ is said to be
a {\em forward neighbor} of 
$\rho$ if $\rho'_i \in {\cal L}_{\rho_i}^+$ for all $i$, 
and
is said to be a {\em backward neighbor} of $\rho$ 
if $\rho'_i \in {\cal L}_{\rho_i}^-$ for all $i$.
A forward or backward neighbor 
is simply called a {\em neighbor}.
This terminology is due to \cite{HH11folder}.
By Lemma~\ref{lem:product}, 
the set of forward (resp. backward) neighbors of $\rho$
is equal to ${\cal L}^+_{\rho}(\pmb{\mGamma}^n) 
= {\cal L}^+_{\rho_1} \times {\cal L}^+_{\rho_2} \times \cdots \times {\cal L}^+_{\rho_n}$ (resp. 
${\cal L}^-_{\rho}(\pmb{\mGamma}^n) 
= {\cal L}^-_{\rho_1} \times {\cal L}^-_{\rho_2} \times \cdots \times {\cal L}^-_{\rho_n}$).

We introduce a sharper concept of a neighbor using orbits. 
Recall (\ref{eqn:L|U}) in Section~\ref{subsec:lattice} for definition of 
${\cal L}|Q$. 
Let $Q$ be an orbit of $\mGamma$; 
consider ${\cal L}_{\rho}^{\pm}|Q := {\cal L}_{\rho}^{\pm}| (Q \cap \{\mbox{edges belonging to ${\cal L}^{\pm}_{\rho}$}\})$. 
A forward neighbor $\rho'$ of $\rho$ is called a {\em forward $Q$-neighbor} of $\rho$
if $\rho'_i \in {\cal L}^+_{\rho_i}|Q$ for all $i$, and 
is called a {\em backward $Q$-neighbor} of $\rho$
if $\rho'_i \in {\cal L}^-_{\rho_i}|Q$ for all $i$.
The set of all forward (reps. backward) $Q$-neighbors of $\rho$ is denoted by
${\cal L}^+_\rho|Q = {\cal L}^+_{\rho_1}|Q \times {\cal L}^+_{\rho_2}|Q \times \cdots \times {\cal L}^+_{\rho_n}|Q$ 
(resp. ${\cal L}^-_\rho|Q = {\cal L}^-_{\rho_1}|Q \times {\cal L}^-_{\rho_2}|Q \times \cdots \times {\cal L}^-_{\rho_n}|Q$).
A forward or backward $Q$-neighbor 
is simply called a {\em $Q$-neighbor}.

The main result in this section 
is the following optimality criterion, 
which has been shown  
for some special cases of orientable modular graphs: 
trees by Kolen~\cite[Chapter 3]{Kolen},  
median graphs by Chepoi~\cite[p.11--12]{Chepoi96}, and
frames by Hirai~\cite[Section 4.1]{HH11folder}.
\begin{Thm}\label{thm:optimality}
Let $\mGamma$ be an orientable modular graph with 
an admissible orientation $o$ and a positive orbit-invariant function $h$.
For a location $\rho$,
the following conditions 
are equivalent:
\begin{itemize}
\item[{\rm (1)}] $\rho$ is optimal to 
{\bf Multifac}$[\mGamma,h;V,c]$.
\item[{\rm (2)}] $\rho$ is optimal to 
{\bf Multifac}$[\mGamma,1 ;V,c]$.
\item[{\rm (3)}] For every neighbor $\rho'$ of $\rho$, 
we have $(c \cdot d_{\mGamma,h})(\rho) \leq (c \cdot d_{\mGamma, h})(\rho')$. That is
\[
(c \cdot d_{\mGamma,h})(\rho) = \min \{ (c \cdot d_{\mGamma,h})(\rho') \mid \rho' \in {\cal L}^{+}_{\rho}\}
= \min \{ (c \cdot d_{\mGamma,h})(\rho') \mid \rho' \in {\cal L}^{-}_{\rho}\}.
\]
\item[{\rm (4)}] For every neighbor $\rho'$ of $\rho$, 
we have $(c \cdot d_{\mGamma,1})(\rho) \leq (c \cdot d_{\mGamma,1})(\rho')$. That is
\[
(c \cdot d_{\mGamma,1})(\rho) = \min \{ (c \cdot d_{\mGamma,1})(\rho') \mid \rho' \in {\cal L}^{+}_{\rho}\}
= \min \{ (c \cdot d_{\mGamma,1})(\rho') \mid \rho' \in {\cal L}^{-}_{\rho}\}.
\]
\item[{\rm (5)}] For every orbit $Q$ and 
every $Q$-neighbor $\rho'$ of $\rho$, 
we have $(c \cdot d_{\mGamma,1})(\rho) \leq (c \cdot d_{\mGamma,1})(\rho')$.
That is,  for every orbit $Q$, we have
\[
(c \cdot d_{\mGamma,1})(\rho) = 
\min\{ (c \cdot d_{\mGamma,1})(\rho') \mid \rho' \in {\cal L}^{+}_{\rho}|Q\} = 
\min\{ (c \cdot d_{\mGamma,1})(\rho') \mid \rho' \in {\cal L}^{-}_{\rho}|Q\}.
\]
\end{itemize}
\end{Thm}
Any orbit $Q$ in $\mGamma$ forms an orbit-union in the covering graph of each ${\cal L}^+_p$.
By Lemma~\ref{lem:interval}~(3), ${\cal L}^+_p|Q$ is a complemented modular lattice, 
and convex in ${\cal L}^+_p$. 
Thus all conditions (3),(4), and (5) are checked 
by submodular function minimization on modular semilattice.

Before the proof, 
we explain consequences of Theorem~\ref{thm:optimality}. 
The first consequence is  that
in solving {\bf Multifac}$[\mGamma, h; V,c]$,
we may replace $h$ with the unit function, 
even when $h$ is not positive.
\begin{Thm}\label{thm:suffice}
For every nonnegative orbit-invariant function $h$, 
every optimal location in {\bf Multifac}$[\mGamma, 1; V,c]$
is optimal to {\bf Multifac}$[\mGamma, h; V,c]$. 
\end{Thm}
\begin{proof}
Let $\rho$ be an optimal location for {\bf Multifac}$[\mGamma, 1; V,c]$.
Take an arbitrary positive $\epsilon > 0$.
Consider the positive orbit invariant function 
$h + \epsilon 1$.
By Theorem~\ref{thm:optimality},
$\rho$ is optimal to 
${\bf Multifac}[\mGamma, h + \epsilon 1 ;V,c]$. 
Hence, for an arbitrary location $\rho'$, 
we have 
\[
(c \cdot d_{\mGamma,h})(\rho') + 
\epsilon (c \cdot d_{\mGamma,1})(\rho')  
 = (c \cdot d_{\mGamma, h + \epsilon 1})(\rho') \geq 
(c \cdot d_{\mGamma, h + \epsilon 1})(\rho)
=  (c \cdot d_{\mGamma, h})(\rho) + \epsilon (c \cdot d_{\mGamma,1})(\rho). 
\]
Since $\epsilon > 0$ was arbitrary, 
we have $(c \cdot d_{\mGamma,h})(\rho') 
\geq (c \cdot d_{\mGamma,h})(\rho)$.
\end{proof}
The second consequence is 
a decomposition property of 
{\bf Multifac}$[\mGamma, h;V,c]$. 
For an orbit $Q$, define $(c \cdot d_{\mGamma/Q,h}): (V_{\mGamma/Q})^n \to \RR$ by
\begin{eqnarray*}
 (c \cdot d_{\mGamma/Q,h}) (\rho) & := & \sum_{s \in V_{\mGamma}} \sum_{1 \leq j \leq n} c(s j) d_{\mGamma/Q,h}(s/Q, \rho_j)
+ \sum_{1 \leq i < j \leq n} c(i j) d_{\mGamma/Q,h}(\rho_i, \rho_j) \\
&&  \hspace{4cm} (\rho = (\rho_1,\rho_2,\ldots,\rho_n) \in (V_{\mGamma/Q})^n).
\end{eqnarray*}
See Section~\ref{subsec:orbits} for the definition of $\mGamma/Q$.
Consider the following problem on $\mGamma/Q$: 
\begin{equation}\label{eqn:Gamma/Q}
\mbox{Minimize 
$(c \cdot d_{\mGamma/Q,h})(\rho)$
over all locations $\rho \in (V_{\mGamma/Q})^n$.}
\end{equation}
The optimal value of (\ref{eqn:Gamma/Q})
is denoted by $\tau_Q(\mGamma,h; V,c)$, whereas
the optimal value of the original problem 
{\bf Multifac}$[\mGamma, h; V,c]$ 
is denoted by $\tau({\mGamma, h; V,c})$.
Then we have
\begin{equation}\label{eqn:additive}
\tau({\mGamma,h;V,c}) \geq \sum_{Q:\footnotesize{\mbox{orbit}}} 
h_Q \tau_Q({\mGamma,1; V,c}).
\end{equation}
Indeed, for any optimal location $\rho$ in {\bf Multifac}$[\mGamma, h; V,c]$, 
define a location $\rho/Q$ for $\mGamma/Q$ by 
$\rho/Q :=  (\rho_1/Q, \rho_2/Q,\ldots, \rho_n/Q)$.
By (\ref{eqn:into}) we have
\begin{equation}\label{eqn:additive2}
\tau({\mGamma, h; V,c}) = (c \cdot d_{\mGamma,h})(\rho) 
= \sum_{Q:\footnotesize{\mbox{orbit}}} h_Q (c \cdot d_{\mGamma,1})(\rho/Q) 
\geq \sum_{Q:\footnotesize{\mbox{orbit}}} h_Q \tau_Q({\mGamma,1; V,c}).
\end{equation}
Note that
problems {\bf Multifac}$[\mGamma, h;V,c]$ and 
(\ref{eqn:Gamma/Q}) can be considered
for a possibly nonorientable modular graph, 
and the inequality relation (\ref{eqn:additive}) still holds; 
see \cite{Kar04a}. 
A modular graph $\mGamma$ 
is said to be {\em orbit-additive}
if (\ref{eqn:additive}) holds in equality.
Karzanov~\cite[Section 6]{Kar04a} 
conjectured that every orientable modular graph 
is orbit-additive. 
We can solve this conjecture affirmatively.
\begin{Thm}\label{thm:additive}
Every orientable modular graph is orbit-additive.
\end{Thm}
\begin{proof}
Take an optimal solution $\rho$ in 
${\bf Multifac}[\mGamma,1;V,c]$.
By Theorem~\ref{thm:suffice}, $\rho$ is also optimal to 
${\bf Multifac}[\mGamma, 1_Q; V,c]$ for every orbit $Q$, 
where $1_Q$ is the orbit-invariant function 
taking $1$ on $Q$ and $0$ on $E_{\mGamma} \setminus Q$.
Here ${\bf Multifac}[\mGamma, 1_Q; V,c]$ 
is equivalent to (\ref{eqn:Gamma/Q}).
Hence the inequality in (\ref{eqn:additive2}) 
holds in equality.
\end{proof}
\begin{Rem}
If problem (\ref{eqn:Gamma/Q}) is solvable in 
(strongly) polynomial 
time for each orbit, 
then by Theorem~\ref{thm:additive} 
we can evaluate $\tau$ in (strongly) polynomial time, 
and hence {\bf 0-Ext}$[\mGamma]$ is solvable in 
(strongly) polynomial time 
by the variable-fixing technique (see the augment after Theorem~\ref{thm:TZ}). 
%
As was suggested by Karzanov~\cite{Kar04a},
this approach is applicable to the case where 
each orbit graph of $\mGamma$ is a frame.
Then (\ref{eqn:Gamma/Q}) is a 0-extension problem on a frame,
is solvable in strongly polynomial time, and 
hence {\bf 0-Ext}$[\mGamma]$ is solvable 
in strongly polynomial time;
the (strong polytime) tractability 
of this class of orientable modular graphs 
was conjectured by~\cite{Kar04a}.
It should be noted that our proof of the main theorem
gives only a weakly polynomial time algorithm.
\end{Rem}

\paragraph{Proof of Theorem~\ref{thm:optimality}.}

(1) $\Leftrightarrow$ (3) and (2) $\Leftrightarrow$ (4)
follow from Theorems~\ref{thm:L-optimality} 
and \ref{thm:L-convex}. 
(4) $\Rightarrow$ (5) is obvious.

(3) $\Rightarrow$ (5). 
Suppose that $\rho'$ is a $Q$-neighbor of $\rho$. 
Then $\rho'/R = \rho/R$ for orbit $R$ different from $Q$, 
and $d_{\mGamma/Q,h} = h_Q d_{\mGamma/Q,1}$.
By (\ref{eqn:into}) we have
\begin{eqnarray*}
0 & \leq & (c \cdot d_{\mGamma,h})(\rho') - (c \cdot d_{\mGamma,h})(\rho) 
=  \sum_{R:\footnotesize\mbox{\rm orbit}} 
(c \cdot d_{\mGamma/R,h})(\rho'/R) - (c \cdot d_{\mGamma/R,h})(\rho/R) \\
&= & (c \cdot d_{\mGamma/Q,h})(\rho'/Q) - (c \cdot d_{\mGamma/Q,h})(\rho/Q) \\ 
&= & h_Q \left\{ (c \cdot d_{\mGamma/Q,1})(\rho'/Q) - (c \cdot d_{\mGamma/Q,1})(\rho/Q)\right\} \\
& = & h_Q \sum_{R:\footnotesize\mbox{\rm orbit}} (c \cdot d_{\mGamma/R,1})(\rho'/R) - (c \cdot d_{\mGamma/R,1})(\rho/R) \\
&= &
h_Q \left\{ (c \cdot d_{\mGamma,1})(\rho') - (c \cdot d_{\mGamma,1})(\rho) \right\}.
\end{eqnarray*}

(5) $\Rightarrow$ (3) and (5) $\Rightarrow$ (4).
By Lemma~\ref{lem:interval}~(3), 
${\cal L}_p^+|Q$ is convex in ${\cal L}_p^+$, and we can define
$q |Q \in {\cal L}_p^+|Q$ for $q \in {\cal L}_p^+$. 
\begin{Lem}
For $p,q \in V_{\mGamma}$, $p' \in {\cal L}^+_p$, and $q' \in {\cal L}^+_q$, we have
\begin{equation*}
d_{\mGamma,h}(p',q') - d_{\mGamma,h}(p,q) = 
\sum_{Q: {\rm orbit}} h_Q \left\{ d_{\mGamma,1}(p'|Q, q'|Q) - d_{\mGamma,1}(p,q) \right\}.
\end{equation*}
\end{Lem}
\begin{proof}
First we remark
\[
(p'|Q)/R = \left\{ 
\begin{array}{ll}
p'/Q & {\rm if}\ Q = R,\\
p/R  & {\rm otherwise}, 
\end{array}\right. \quad (p' \in {\cal L}^+_p,\ Q, R:\mbox{orbits}).
\]
Indeed, if $Q \neq R$, then $p'|Q$ and $p$ 
are joined by edges in $Q \subseteq E_{\mGamma} \setminus R$, 
and hence $(p'|Q)/R = p/R$.
If $Q = R$, then $p'$ and $p'|Q$ 
are joined by $E_{\mGamma} \setminus Q$ (Lemma~\ref{lem:interval}~(3)), 
and hence $p'/Q = (p'|Q)/Q$.
Thus we have
\begin{eqnarray*}
&& \sum_{Q: {\rm orbit}} h_Q \left\{ d_{\mGamma,1}(p'|Q, q'|Q) - d_{\mGamma,1}(p,q) \right\} \\
&& = \sum_{Q: {\rm orbit}} h_Q \sum_{R: {\rm orbit}} 
\left\{  d_{\mGamma/R,1}((p'|Q)/R, (q'|Q)/R) - d_{\mGamma/R,1}(p/R,q/R) \right\} \\
&& = \sum_{Q: {\rm orbit}}  
h_Q \left\{ d_{\mGamma/Q,1}(p'/Q, q'/Q) - d_{\mGamma/Q,1}(p/Q,q/Q) \right\} \\
&& = d_{\mGamma,h}(p',q') - d_{\mGamma,h}(p,q).
\end{eqnarray*}
\end{proof}

Let $\rho$ be a location, and let $\rho'$ be a forward neighbor of $\rho$.
For each orbit $Q$, define the forward $Q$-neighbor $\rho'|Q$ 
by $(\rho'|Q)_i := \rho'_i|Q$ for $i=1,2,\ldots,n$.
By the above lemma, we have
\begin{equation*}
(c \cdot d_{\mGamma,h})(\rho') - (c \cdot d_{\mGamma,h})(\rho) = 
\sum_{Q:{\rm orbit}} h_Q \left\{ (c \cdot d_{\mGamma,1})(\rho'|Q) - (c \cdot d_{\mGamma,1})(\rho) \right\}.
\end{equation*}
From this, we have (5) $\Rightarrow$ (3) and (5) $\Rightarrow$ (4).

\subsection{Proof of the main theorem (Theorem~\ref{thm:main})}\label{subsec:proof_main}
In this section, we complete the proof 
of the main theorem (Theorem~\ref{thm:main}) stating 
that {\bf 0-Ext}$[\mGamma]$ 
for every orientable modular graph $\mGamma$ 
can be solved in polynomial time.
By Theorem~\ref{thm:L-convex}, an instance {\bf Multifac}$[\mGamma,1;V,c]$ of
{\bf 0-Ext}$[\mGamma]$ is the problem of 
minimizing the sum $(c \cdot d_{\mGamma,1})$ of
L-convex constraints on $\pmb{\mGamma}^n$, 
where the arity of each constraint is $2$.
Hence, every location $\rho = (\rho_1,\rho_2,\ldots, \rho_n)$, and sign $s \in \{-,+\}$, 
$(c \cdot d_{\mGamma,1})$ is 
the sum of arity-2 submodular 
constraints on ${\cal L}_{\rho}^{s} = 
{\cal L}_{\rho_1}^{s} \times {\cal L}_{\rho_2}^{s} \times \cdots \times {\cal L}_{\rho_n}^{s}$, 
where each semilattice ${\cal L}_{\rho_i}^{s} = {\cal L}_{\rho_i}^{s}(\pmb{\mGamma})$ is 
constructed in polynomial time (Proposition~\ref{prop:modularlattice}). 
By Theorem~\ref{thm:minimized},
we can minimize $(c \cdot d_{\mGamma,1})$ 
over ${\cal L}^+_{\rho} \cup {\cal L}^-_{\rho}$ in polynomial time.
Therefore we can assume 
that we have a {\em descent oracle},  
an oracle that returns an optimal solution 
of this (local) problem.

By Theorem~\ref{thm:L-optimality} and the steepest descent algorithm, 
we can obtain a global optimal solution.
As mentioned already, we do not know whether 
the number of descent steps
is polynomially bounded.
Fortunately,
in the case of multifacility location functions,  
a cost-scaling approach gives 
a {\em weakly} polynomial bound on
the number of descent steps. 
Now the main theorem (Theorem~\ref{thm:main}) follows from 
the following.

\begin{Prop}
Suppose that $c$ is integer-valued.  
${\bf Multifac}[\mGamma,1; V,c]$ can be solved 
with $O(|V|^2 \diam \mGamma \log C)$ 
calls of the descent oracle, 
where $C := \max \{ c(xy) \mid xy \in {V \choose 2}\}$ and 
$\diam \mGamma$ denotes the diameter of $\mGamma$.
\end{Prop}
\begin{proof}
Let $\lfloor c/2 \rfloor: {V \choose 2} \to \ZZ_+$ 
be defined by $\lfloor c/2 \rfloor(xy) := \lfloor c(xy)/2 \rfloor$ 
(the largest integer not exceeding $c(xy)/2$).
Let $\rho$ be an optimal location in ${\bf Multifac}[\mGamma,1; V, \lfloor c/2 \rfloor]$, 
and let $\rho^*$ be an optimal location in ${\bf Multifac}[\mGamma,1; V, c]$.
We show
\begin{equation}\label{eqn:steps}
(c \cdot d_{\mGamma,1})(\rho) -  (c \cdot d_{\mGamma,1})(\rho^*)
\leq |V|^2 \diam \mGamma.
\end{equation}
If (\ref{eqn:steps}) is true, then the number of 
the descent steps from an initial starting point $\rho$
is bounded by $|V|^2 \diam \mGamma$.
Consequently, by recursive scaling,
we obtain an optimal solution for 
${\bf Multifac}[\mGamma,1; V,c]$
in $O(|V|^2 \diam \mGamma \log C)$ descent steps.

%
Let $\epsilon$ be a $\{0,1\}$-valued 
cost defined by $\epsilon := c - 2 \lfloor c/2 \rfloor$.
%
Then we have
\begin{eqnarray*}
&& (c \cdot d_{\mGamma,1})(\rho) - (c \cdot d_{\mGamma,1})(\rho^*) \\
&&  =   (2 \lfloor c/2 \rfloor \cdot d_{\mGamma,1})(\rho) - (2 \lfloor c/2 \rfloor \cdot d_{\mGamma,1})(\rho^*)  + (\epsilon \cdot d_{\mGamma,1})(\rho) -  
(\epsilon \cdot d_{\mGamma,1})(\rho^*)  \\ 
&& \leq (n|V_{\mGamma}| + n(n-1)/2) \diam \mGamma,
\end{eqnarray*}
where we use the facts that $\rho$ is also an optimal location in 
${\bf Multifac}[\mGamma,1; V, 2 \lfloor c/2 \rfloor]$ and that 
each term in $(\epsilon \cdot d_{\mGamma,1})(\rho')$ (for any location $\rho'$) is at most  $\diam \mGamma$.
\end{proof}

\subsection{Minimum $0$-extension problems for metrics}\label{subsec:nongraphical}
Let $\mu$ be a metric on a finite set $S$ (not necessarily a graph metric).
We can naturally consider the minimum $0$-extension 
problem {\bf 0-Ext}$[\mu]$ for a general $\mu$ formulated as:
{\em Given a set $V \supseteq S$ and 
$c: {V \choose 2} \to \QQ_+$, 
find a $0$-extension $(V,d)$ of $(S,\mu)$ 
with $\sum_{xy} c(xy) d(x,y)$ minimum}. 

Metric $\mu$ is said to be {\em modular} 
if $(S,\mu)$ is a modular metric space 
(see Section~\ref{sec:preliminaries}).
Let $H_\mu$ be the graph on the vertex set $S$ 
with edge set $E_{H_{\mu}}$ given as: $xy \in E_{H_\mu}$ 
$\Leftrightarrow$ there is no $z \in S \setminus \{x,y\}$ 
with $\mu(x,z) + \mu(z,y) = \mu(x,y)$.
$H_\mu$ is called the {\em support graph} of $\mu$.
Karzanov~\cite{Kar04b} extended the hardness result 
(Theorem~\ref{thm:hard}) to the following. 
\begin{Thm}[\cite{Kar04b}]\label{thm:mu_hard}
If $\mu$ is not modular or 
$H_\mu$ is not orientable, 
then {\bf 0-Ext}$[\mu]$ is NP-hard. 
\end{Thm}
We can also consider LP-relaxation {\bf Ext}$[\mu]$ obtained
by relaxing $0$-extensions into extensions in {\bf 0-Ext}$[\mu]$.
Extending Theorem~\ref{thm:minimizable},
Bandelt, Chepoi, and Karzanov~\cite{BCK00} proved 
that {\bf Ext}$[\mu]$ is exact if and only if 
$\mu$ is modular and $H_\mu$ is frame.

Our framework
covers {\bf 0-Ext}$[\mu]$ 
for a metric $\mu$ such that $\mu$ is modular 
and $H_{\mu}$ is orientable.
Indeed $\mu$ induces 
the edge-length 
$\bar \mu$ on $H_{\mu}$ by 
$\bar \mu(pq) = \mu(p,q)$ $(pq \in E_{H_{\mu}})$.
From the definition of the support graph $H_{\mu}$, 
we have $\mu = d_{H_{\mu}, \bar \mu}$.
Moreover, it was shown in \cite{Bandelt88} (see \cite[Section 2]{Kar04a})  
that
\begin{myitem}
if $\mu$ is modular, then $H_{\mu}$ is a modular graph, 
and $\bar \mu$ is orbit-invariant.
\end{myitem}\noindent
Hence we can apply 
the argument in Section~\ref{sec:0extension} 
to {\bf Multifac}$[H_{\mu},\bar \mu; V,c]$ 
to obtain results for {\bf 0-Ext}$[\mu]$.
By Theorems~\ref{thm:main} and \ref{thm:optimality},
we obtain the converse of Theorem~\ref{thm:mu_hard}, 
which completes the classification of 
those metrics for which {\bf 0-Ext}$[\mu]$ is tractable.
\begin{Thm}\label{thm:main_mu}
If $\mu$ is modular and $H_\mu$ is orientable,
then {\bf 0-Ext}$[\mu]$ is solvable in polynomial time.
\end{Thm}

\section{Concluding remark}\label{sec:concluding}

In this paper, we established 
the ``P or NP-hard" classification of the minimum 0-extension problem.
This dichotomy result is related to a special case of 
a dichotomy theorem for {\em finite-valued CSP} 
due to Thapper and \v{Z}ivn\'{y}~\cite{ThapperZivny13STOC}.
Here we briefly explain their result and its relation to our result.

To describe their result, we formulate 
valued CSP in a setting slightly different from that in Section~\ref{subsec:VCSP}.  
Let $D$ be a finite set. A (finite-valued) {\em cost function} 
on $D$ is a function $f: D^{k} \to \QQ$ for some $k = k_f$.
A (finite-valued) {\em constraint language}, or simply, {\em language} on $D$ is
a set  $\mLambda$ of cost functions on $D$.
Let $x_1,x_2,\ldots,x_n$ be a set of variables.
By a {\em $\mLambda$-constraint}, 
we here mean a triple $(w, f, \sigma)$ of a nonnegative weight $w \in \QQ_+$, 
a cost function $f$ in $\mLambda$, and a map $\sigma:\{1,2,\ldots,k_f\} \to \{1,2,\ldots,n\}$.
For a constraint language $\mLambda$,
the problem ${\bf VCSP}[\mLambda]$ is formulated as:
\begin{description}
\item[{${\bf VCSP}[\mLambda]$}:] Given a set ${\cal C}$ of $\mLambda$-constraints, \\[0.5em]
minimize $\displaystyle \sum_{(w,f,\sigma) \in {\cal C}} w f(x_{\sigma(1)}, x_{\sigma(2)},\ldots, x_{\sigma(k_f)})$ 
over all $x = (x_1,x_2,\ldots,x_n) \in D^n$.
\end{description}
This is also a subclass of {\bf VCSP} studied in Section~\ref{subsec:VCSP}. 
Therefore we can consider the basic LP relaxation (BLP).
An {\em $m$-ary fractional polymorphism} for $\mLambda$ is 
a formal convex combination $\omega = \sum_{\vartheta} \omega(\vartheta) \vartheta$ 
of $m$-ary operations $\vartheta: D^m \to D$ such that
\begin{eqnarray*}
 \frac{1}{m} \{ f(x^1) + f(x^2) + \cdots + f(x^m) \} 
&\geq & \sum_{\vartheta} \omega(\vartheta) f(\vartheta (x^1,x^2,\ldots,x^m)) \\
&& \quad (f \in \mLambda,\ x^1,x^2,\ldots,x^m \in D^{k_f}),
\end{eqnarray*}
where an operation $\vartheta: D \times D \times \cdots \times D \to D$ is extended 
to an operation $D^k \times D^k \times \cdots \times D^k  \to D^k$ by
$(\vartheta(x^1,x^2,\ldots,x^k))_i := \vartheta(x^1_i,x^2_i,\ldots,x^k_i)$.
If $m= 1$, then $\omega$ is said to be {\em unary}, 
and if $m=2$, then $\omega$ is said to be {\em binary}.
A binary fractional polymorphism $\omega$ is said to be {\em symmetric} 
if its support consists of symmetric operations, i.e., operations $\vartheta$ satisfy 
$\vartheta (x,y) = \vartheta (y,x)$ for $x,y \in D$, and $\omega$ is said to be {\em idempotent} 
if its support consists of idempotent operations, i.e.,
operations $\vartheta$ satisfy $\vartheta(x,x) = x$ for $x \in D$.

Just after the developments \cite{Kolmogorov13ICALP, ThapperZivny12FOCS} (see \cite{KTZ13}),
Thapper and \v{Z}ivn\'y~\cite{ThapperZivny13STOC} 
established the following dichotomy theorem for finite-valued CSP.
Here a language $\mLambda$ is said be a {\em core} if 
for every unary fractional polymorphism, 
its support consists of injective operations;
It is shown in \cite{ThapperZivny13STOC} that ${\bf VCSP}[\mLambda]$ is polynomial time 
reducible to ${\bf VCSP}[\mLambda']$ for a core language $\mLambda'$.
\begin{Thm}[\cite{ThapperZivny13STOC}]\label{thm:TZ13}
Let $\mLambda$ be a finite-valued core language on $D$.
If $\mLambda$ admits a binary symmetric and idempotent fractional polymorphism, 
then BLP is exact for $\mLambda$, and ${\bf VCSP}[\mLambda]$ can be solved in polynomial time.
Otherwise ${\bf VCSP}[\mLambda]$ is NP-hard.
\end{Thm}
The minimum 0-extension problem can naturally be formulated as ${\bf VCSP}[\mLambda]$.
Let $\mu$ be a (rational-valued) metric on $D$; 
we can assume that $\mu(s,t) > 0$ whenever $s \neq t$.
The metric $\mu: D^2 \to \QQ_+$ itself is regarded as a binary cost function on $D$.
For $s \in D$, let $\mu_s: D \to \QQ_+$ 
be a unary cost function on $D$ defined by
\[
\mu_{s}(x) := \mu(s,x) \quad (x \in D).
\]
Let $\mLambda_{\mu}$ be a (finite-valued) constraint language defined by
\[
\mLambda_{\mu} := \{ \mu \} \cup \{ \mu_{s} \}_{s \in D}.
\]
By definition, we have:
\begin{Lem}
${\bf VCSP}[\mLambda_{\mu}]$ $=$ {\bf 0-Ext}$[\mu]$.
\end{Lem}
Moreover we have:
\begin{Lem}
$\mLambda_{\mu}$ is a core language.
\end{Lem}
\begin{proof}
Take a unary fractional polymorphism $\omega = \sum_{\vartheta} \omega(\vartheta) \vartheta$
for $\mLambda_{\mu}$.  
Then it holds
\[
0 = \mu(s, s) \geq \sum_{\vartheta} \omega(\vartheta) \mu(s, \vartheta(s)) \geq 0.
\]
Therefore $\vartheta(s) = s$ must hold. 
This means that the support of $\omega$ consists of 
the identity map on $D$, which is trivially injective. 
\end{proof}
Therefore Theorem~\ref{thm:TZ13} is applicable to {\bf 0-Ext}$[\mu]$.
In particular, Theorems~\ref{thm:mu_hard} and \ref{thm:main_mu}
can be viewed as a sharpening of Theorem~\ref{thm:TZ13} 
for constraint languages $\mLambda_\mu$.
For those metrics $\mu$ in Theorem~\ref{thm:main_mu},  
core language $\mLambda_{\mu}$ must have a 
binary symmetric and idempotent fractional polymorphism, 
and {\bf 0-Ext}$[\mu]$ must be solved directly by BLP 
(under the assumption P $\neq$ NP).
In Remark~\ref{rem:nonsemilattice},
we have verified this fact for the case where $\mu$ is the metric on a modular semilattice ${\cal L}$;
it is a good exercise to construct a binary symmetric and idempotent 
fractional polymorphism from (\ref{eqn:polymorphism_submo}).
We however could not find 
such a fractional polymorphism for the general case.
As seen in Section~\ref{sec:submo}, 
a fractional polymorphism can rather be complicated 
and consist of a large number of operations. 
To construct a fractional polymorphism as required, 
it might need a further thorough investigation on orientable modular graphs. 

\section*{Acknowledgments}
We thank the referee for helpful comments, and thank Kazuo Murota for careful reading 
and numerous helpful comments,  
Kei Kimura for discussion on Valued-CSP, 
Satoru Iwata for communicating the paper~\cite{Kuivinen} of Kuivinen, 
Akiyoshi Shioura for the paper~\cite{Kolmogorov} of Kolmogorov, and
Satoru Fujishige for the paper~\cite{HuberKolmogorov12} 
of Huber-Kolmogorov.
This research is partially supported by the Aihara Project, the FIRST
program from JSPS, by Global COE Program ``The research and training
center for new development in mathematics'' from MEXT, and
by a Grant-in-Aid for Scientific Research 
from the Ministry of Education, Culture, Sports,
Science and Technology of Japan.

\end{document}